\documentclass[A4paper,oneside,11pt]{article}
\usepackage{amssymb,amsthm,amsmath,marvosym,amsfonts,fontenc}
\usepackage[normalem]{ulem}
\usepackage[usenames,dvipsnames]{color}
\usepackage{enumerate,marginnote}
\usepackage{mdwlist}
\usepackage{pifont}
\usepackage{bm}
\usepackage{multind}\makeindex{general}\makeindex{mathsymbols}
\usepackage{caption}
\usepackage{fullpage} 
\usepackage{epsfig,psfrag}
\usepackage{epstopdf}
\usepackage{fancyhdr} 
\usepackage{xr}
\numberwithin{equation}{section}
\numberwithin{table}{section} 
\numberwithin{figure}{section}
\newtheorem{theorem}{Theorem}[section]
\newtheorem{subtheorem}{Theorem}[theorem]

\newtheorem{lemma}[theorem]{Lemma}
\newtheorem{conjecture}[theorem]{Conjecture}

\newtheorem{fact}[theorem]{Fact}
\newtheorem{remark}[theorem]{Remark}
\newtheorem{definition}[theorem]{Definition}
\newtheorem{setting}[theorem]{Setting}

\theoremstyle{remark}
\newtheorem{claim}[subtheorem]{Claim}
\newtheorem{subremark}[subtheorem]{Remark}

\newcommand{\By}[2]{\overset{\mbox{\tiny{#1}}}{#2}}
\newcommand{\ByRef}[2]{   \By{\eqref{#1}}{#2} }
\newcommand{\eqBy}[1]{    \By{#1}{=} }

\newcommand{\leBy}[1]{    \By{#1}{\le} }

\newcommand{\geByRef}[1]{ \ByRef{#1}{\ge} }

\newcommand{\BUG}[1]{#1}

\let\sm\setminus
\let\subset\subseteq 
 
\let\epsilon\varepsilon
\let\sharp\#
 
\renewcommand{\leq}{\leqslant}
\renewcommand{\le}{\leqslant}
\renewcommand{\geq}{\geqslant}
\renewcommand{\ge}{\geqslant}

\let\oldmarginpar\marginpar
\renewcommand\marginpar[1]{\-\oldmarginpar[\raggedleft\footnotesize #1]%
{\raggedright\footnotesize #1}}

\newcommand{\HIDDENPROOF}[1]{
}

\newcommand{\HIDDENTEXT}[1]{
}

\title{The Approximate Loebl--Koml\'os--S\'os
Conjecture IV:\\
Embedding techniques and the proof of the main result} 
\author{Jan
Hladk\'y
\thanks{\emph{Corresponding author.} Institute of Mathematics, Academy of Science of the Czech Republic. \v Zitn\'a 25, 110 00, Praha, Czech Republic. The Institute of Mathematics of the Academy of Sciences of the Czech Republic is supported by RVO:67985840. Email:
\texttt{honzahladky@gmail.com}.
The research leading to these results has received funding from the People Programme (Marie Curie Actions) of the European Union's Seventh Framework Programme (FP7/2007-2013) under REA grant agreement umber 628974. Much of the work was done while supported by an EPSRC postdoctoral fellowship while affiliated with DIMAP and Mathematics Institute, University of
Warwick.}
\quad 
J\'anos Koml\'os\thanks{Department of Mathematics, Rutgers University, 110 Frelinghuysen Rd., Piscataway, NJ~08854-8019, USA} 
\quad 
Diana Piguet\thanks{Institute of Computer Science, Czech Academy of Sciences, Pod Vod\'arenskou v\v e\v z\'i 2, 182~07 Prague, Czech Republic. With institutional support RVO:67985807. Supported by the Marie Curie fellowship FIST, DFG grant TA 309/2-1,  Czech Ministry of Education project 1M0545, EPSRC award EP/D063191/1, and EPSRC Additional Sponsorship EP/J501414/1.
	The research leading to these results has received funding from the European Union Seventh
	Framework Programme (FP7/2007-2013) under grant agreement no. PIEF-GA-2009-253925.
    The work leading to this invention was supported by the European Regional Development Fund (ERDF), project ``NTIS -- New Technologies for Information Society'', European Centre of Excellence, CZ.1.05/1.1.00/02.0090.}
    \\ 
    Mikl\'os Simonovits\thanks{R\'enyi
    Institute, Budapest, Hungary. Supported by OTKA~78439, OTKA~101536, ERC-AdG.~321104} 
\quad 
Maya Stein\thanks{Department of Mathematical Engineering,
University of Chile, Santiago, Chile.  Supported by Fondecyt Iniciacion grant 11090141, Fondecyt Regular grant 1140766 and CMM Basal.}
\quad 
Endre Szemer\'edi\thanks{R\'enyi
	Institute, Budapest, Hungary. Supported by OTKA~104483 and ERC-AdG.~321104}}

\def\semiregular{regularized }
\def\semiregulars{regularized}

\def\kknnaagg{hub }
\def\kknnaaggss{hubs }
\def\kknnaaggNOSPACE{hub}
\def\kknnaaggssNOSPACE{hubs}
\def\KknnaaggssNOSPACE{Kubs}

\newcommand{\PARAMETERPASSING}[2]{{\mathrm{#1}\ref{#2}}}
\newcommand{\PARAMETERPASSINGR}[3]{{\mathrm{#1}\ref{#2}-\ref{#3}}}
\newcommand{\EXTERNALPARAMETERPASSING}[3]{{\mathrm{#1}.\mathrm{#2}\ref{#3}}}

\def\NN{\mathbb{N}}

\newcommand{\M}{\mathcal M}\newcommand{\V}{\mathcal V}
\newcommand{\eps}{\epsilon}

\def\probability{\mathbf{P}}

\def\mindeg{\mathrm{mindeg}}
\def\maxdeg{\mathrm{maxdeg}}
\def\density{\mathrm{d}}
\def\neighbour{\mathrm{N}}
\def\Vodd{V_\mathrm{odd}}
\def\Veven{V_\mathrm{even}}

\def\dist{\mathrm{dist}}
\def\children{\mathrm{Ch}}
\def\parent{\mathrm{Par}}
\def\shadow{\mathrm{shadow}}
\newcommand{\treeclass}[1]{\mathbf{trees}({#1})}
\def\seed{\mathrm{Seed}}
\newcommand{\LKSgraphs}[3]{\mathbf{LKS}({#1},{#2},{#3})}
\newcommand{\LKSmingraphs}[3]{\mathbf{LKSmin}({#1},{#2},{#3})}
\newcommand{\LKSsmallgraphs}[3]{\mathbf{LKSsmall}({#1},{#2},{#3})}

\newcommand{\smallvertices}[3]{\mathbb{S}_{{#1},{#2}}({#3})}
\newcommand{\largevertices}[3]{\mathbb{L}_{{#1},{#2}}({#3})}

\newcommand{\JUSTIFY}[1]{\mbox{\tiny{(#1)}}\quad}

\def\Gcapt{G_\nabla}
\def\GD{G_{\mathcal{D}}}
\def\Gblack{G_{\mathrm{reg}}}
\def\Gexp{G_{\mathrm{exp}}}
\def\BGblack{\mathbf{G}_{\mathrm{reg}}}
\def\smallatoms{\mathbb{E}}
\def\clusters{\mathbf{V}}
\def\class{\nabla}
\def\HugeVertices{\mathbb{H}}
\def\DenseSpots{\mathcal{D}}

\def\YA{\mathbb{YA}}
\def\YB{\mathbb{YB}}
\def\WantiC{V_{\leadsto\HugeVertices}}
\def\ghost{\mathrm{ghost}}

\def\Vgood{V_\mathrm{good}}

\def\exceptVertSplit{\bar V}
\def\exceptSemSplit{\bar \V}
\def\exceptClustSplit{\bar \clusters}

\def\gP{\mathbb{J}}
\def\gPatoms{\mathbb{J}_{\smallatoms}}

\def\shrubA{\mathcal S_{A}}
\def\shrubB{\mathcal S_{B}}

\def\Duplicate{\mathsf{Duplicate}}

\def\XA{\mathbb{XA}}
\def\XB{\mathbb{XB}}
\def\XC{\mathbb{XC}}
 
\def\BL{\mathbf L}

\def\Mgood{\M_{\mathrm{good}}}
\def\NAtom{{\mathcal N_{\smallatoms}}}

\newcommand{\colouringp}[1]{\mathbb{A}_{#1}}
\newcommand{\colouringpI}[1]{^{\restriction{#1}}}
\def\colouringpartition{(\colouringp{0},\colouringp{1},\colouringp{2})}
\newcommand{\proporce}[1]{\mathfrak{p}_{#1}}
\def\shadowsplit{\mathbb{F}}

\def\largeintoatoms{V_{\leadsto\smallatoms}}
\def\clustersintoatoms{\clusters_{\leadsto\smallatoms}}

\def\clustersize{\mathfrak{c}}

\def\LargeTen{\mathcal L^*}

\def\epsilonD{\pi}
\def\alphaD{\widehat{\alpha}}

\def\AXA{\mathbb{A}}
\def\aXa{\mathfrak{q}}
\def\VXV{\mathbb{B}}

\renewcommand{\today}{}
\date{}

\begin{document}
\pagenumbering{roman}
\maketitle
\begin{abstract}
This is the last paper of a series of four papers in which we prove the following relaxation of the
Loebl--Koml\'os--S\'os Conjecture: For every~$\alpha>0$
there exists a number~$k_0$ such that for every~$k>k_0$
 every $n$-vertex graph~$G$ with at least~$(\frac12+\alpha)n$ vertices
of degree at least~$(1+\alpha)k$ contains each tree $T$ of order~$k$ as a
subgraph. 

In the first two papers of this series, we
decomposed the host graph~$G$, and found a suitable combinatorial structure inside the decomposition. In the third paper, we refined this structure, and proved that any graph satisfying the conditions of the above approximate version of the Loebl--Koml\'os--S\'os Conjecture contains one of ten specific configurations.
In this paper we embed the tree $T$ in each of the ten configurations.
\end{abstract}


\newpage

\rhead{\today}

\newpage

\tableofcontents
\newpage
\pagenumbering{arabic}
\setcounter{page}{1}

\section{Introduction}\label{sec:Intro3}
This paper concludes a series of four papers
in which we provide an approximate solution of the Loebl--Koml\'os--S\'os Conjecture. The conjecture reads as follows.
 
\begin{conjecture}[Loebl--Koml\'os--S\'os Conjecture 1995~\cite{EFLS95}]\label{conj:LKS}
Suppose that $G$ is an $n$-vertex graph with at least $n/2$ vertices of degree more than $k-2$. Then $G$ contains each tree of order $k$.
\end{conjecture}

We discuss the history and state of the art in detail  in the first paper~\cite{cite:LKS-cut0} of our series. Our main result, which we prove in the present paper, is the 
approximate solution of the Loebl--Koml\'os--S\'os Conjecture, and reads as follows.

\begin{theorem}[Main result]\label{thm:main}
For every $\alpha>0$ there exists  a number $k_0$ such that for any
$k>k_0$ we have the following. Each $n$-vertex graph $G$ with at least
$(\frac12+\alpha)n$ vertices of degree at least $(1+\alpha)k$ contains each tree $T$ of
order $k$.
\end{theorem}

In the first paper~\cite{cite:LKS-cut0} of the series we exposed a decomposition technique (\emph{the sparse decomposition}), and in the second paper~\cite{cite:LKS-cut1}, we found a  rough combinatorial structure in the host graph $G$.  In~\cite{cite:LKS-cut2}, the third paper of the series, we refined this structure, and obtained one of ten possible~\emph{configurations}, at least one of which appears in any graph satisfying the hypotheses of Theorem~\ref{thm:main}. These configurations will be reintroduced in Section~\ref{FromEarlierPapersOfTheSeries}. All the configurations are built up from basic elements which are inherited from the sparse decomposition.
 
In the present paper, we will embed the tree $T$ in the host graph $G$ using the preprocessing from~\cite{cite:LKS-cut2}. Let us give a short outline of this procedure.
First, we cut the tree into smaller subtrees, connected by few vertices. This will be done in Section~\ref{ssec:cut}. 

We then develop techniques to embed the smaller subtrees  in different building blocks of the configurations.
Then, for each of the configurations, we show how to combine the embedding techniques for smaller trees to embed the whole tree $T$. All of this will be done in Section~\ref{sec:embed}. We mention that Section~\ref{ssec:embeddingOverview} contains a 5-page overview of the embedding procedures, with all the relevant ideas.

Finally, in Section~\ref{sec:proof}, we prove Theorem~\ref{thm:main}, with the help of the main results from the earlier papers~\cite{cite:LKS-cut0, cite:LKS-cut1, cite:LKS-cut2}.

\section{Notation and preliminaries}\label{sec:not3}

\subsection{General notation}
The set $\{1,2,\ldots, n\}$ of the first $n$ positive integers is
denoted by \index{mathsymbols}{*@$[n]$}$[n]$. We frequently employ indexing by many indices. We write
superscript indices in parentheses (such as $a^{(3)}$), as
opposed to notation of powers (such as $a^3$).
We sometimes use subscript to refer to
parameters appearing in a fact/lemma/theorem. For example
$\alpha_\PARAMETERPASSING{T}{thm:main}$ refers to the parameter $\alpha$ from Theorem~\ref{thm:main}.
We omit rounding symbols when this does not lead to confusion.

We use lower case Greek letters
to denote small positive constants. The exception is the letter~$\phi$ which is
reserved for embedding of a tree $T$ in a graph $G$, $\phi:V(T)\rightarrow
V(G)$. The capital Greek letters are used for large constants. 

We write \index{mathsymbols}{*VG@$V(G)$}$V(G)$ and \index{mathsymbols}{*EG@$E(G)$}$E(G)$ for the vertex set and edge set of a graph $G$, respectively. Further, \index{mathsymbols}{*VG@$v(G)$}$v(G)=|V(G)|$ is the order of $G$, and \index{mathsymbols}{*EG@$e(G)$}$e(G)=|E(G)|$ is its number of edges. If $X,Y\subset V(G)$ are two not necessarily disjoint sets of vertices we write \index{mathsymbols}{*EX@$e(X)$}$e(X)$ for the number of edges induced by $X$, and \index{mathsymbols}{*EXY@$e(X,Y)$}$e(X,Y)$ for the number of ordered pairs $(x,y)\in X\times Y$ for which $xy\in E(G)$. In particular, note that $2e(X)=e(X,X)$.

\index{mathsymbols}{*DEG@$\deg$}\index{mathsymbols}{*MINDEG@$\mindeg$}\index{mathsymbols}{*MAXDEG@$\maxdeg$}
For a graph $G$, a vertex $v\in V(G)$ and a set $U\subset V(G)$, we write
$\deg(v)$ and $\deg(v,U)$ for the degree of $v$, and for the number of
neighbours of $v$ in $U$, respectively. We write $\mindeg(G)$ for the minimum
degree of $G$, $\mindeg(U):=\min\{\deg(u)\::\: u\in U\}$, and
$\mindeg(V_1,V_2)=\min\{\deg(u,V_2)\::\:u\in V_1\}$ for two sets $V_1,V_2\subset
V(G)$. Similar notation is used for the maximum degree, denoted by $\maxdeg(G)$.
The neighbourhood of a vertex $v$ is denoted by
\index{mathsymbols}{*N@$\neighbour(v)$}$\neighbour(v)$. We set $\neighbour(U):=\bigcup_{u\in
U}\neighbour(u)$. The symbol $-$ is used
for two graph operations: if $U\subset V(G)$ is a vertex
set then $G-U$ is the subgraph of $G$ induced by the set
$V(G)\setminus U$. If $H\subset G$ is a subgraph of $G$ then the graph
$G-H$ is defined on the vertex set $V(G)$ and corresponds
to deletion of edges of $H$ from $G$.

\subsection{Regular pairs}

In this section we introduce the notion of regular pairs which is central for
Szemer\'edi's regularity lemma. We also list some simple properties of regular pairs that will be useful in our embedding process.

Given a graph $H$ and two disjoint
sets $U,W\subset V(H)$ the
\index{general}{density}\index{mathsymbols}{*D@$\density(U,W)$}\emph{density of the pair $(U,W)$} is defined as
$$\density(U,W):=\frac{e(U,W)}{|U||W|}\;.$$
Similarly, for a bipartite graph $G$ with colour
classes $U$, $W$ we talk about its \index{general}{bipartite
density}\emph{bipartite density}\index{mathsymbols}{*D@$\density(G)$} $\density(G)=\frac{e(G)}{|U||W|}$.
For a given $\varepsilon>0$, a pair $(U,W)$ of disjoint
sets $U,W\subset V(H)$ 
is called an \index{general}{regular pair}\emph{$\epsilon$-regular
pair} if $|\density(U,W)-\density(U',W')|<\epsilon$ for every
$U'\subset U$, $W'\subset W$ with $|U'|\ge \epsilon |U|$, $|W'|\ge
\epsilon |W|$. If the pair $(U,W)$ is not $\epsilon$-regular,
then we call it \index{general}{irregular}\emph{$\epsilon$-irregular}. A stronger notion than
regularity is that of super-regularity which we recall now. A pair $(A,B)$ is
\index{general}{super-regular pair}\emph{$(\epsilon,\gamma)$-super-regular} if it is
$\epsilon$-regular,  $\mindeg(A,B)\ge\gamma |B|$, and
$\mindeg(B,A)\ge \gamma |A|$. Note that then $(A,B)$ has bipartite density at least $\gamma$.

The following two well-known properties of
regular pairs will be useful.
\begin{fact}\label{fact:BigSubpairsInRegularPairs}
Suppose that $(U,W)$ is an $\varepsilon$-regular pair of density
$d$. Let $U'\subset U$ and $W'\subset W$ be sets of vertices with $|U'|\ge
\alpha|U|$ and $|W'|\ge \alpha|W|$, where $\alpha>\epsilon$.
Then the pair $(U',W')$ is a $2\varepsilon/\alpha$-regular pair of density at least
$d-\varepsilon$.
\end{fact}
\begin{fact}\label{fact:manyTypicalVertices}
	Suppose that $(U,W)$ is an $\varepsilon$-regular pair of density
	$d$. Then all but at most $\epsilon|U|$ vertices $v\in U$ satisfy
	$\deg(v,W)\ge (d-\epsilon)|W|$.
\end{fact}

\subsection{LKS graphs}\index{mathsymbols}{*LKSgraphs@$\LKSgraphs{n}{k}{\eta}$}
\label{subsection:LKSgraphs}

Write
\index{mathsymbols}{*LKSgraphs@$\LKSgraphs{n}{k}{\eta}$}$\LKSgraphs{n}{k}{\alpha}$
for the class of all $n$-vertex graphs with at least
$(\frac12+\alpha)n$ vertices of degrees at least
$(1+\alpha)k$. Write \index{mathsymbols}{*Trees@$\treeclass{m}$}$\treeclass{m}$ for the class of all trees on $m$ vertices. With this notation, Conjecture~\ref{conj:LKS} states that every graph in $\LKSgraphs{n}{k}{0}$ contains every tree from $\treeclass{k+1}$.

Define \index{mathsymbols}{*LKSmingraphs@$\LKSmingraphs{n}{k}{\eta}$}
 $\LKSmingraphs{n}{k}{\eta}$ as the set 
of all graphs $G\in \LKSgraphs{n}{k}{\eta}$ that are  edge-minimal in $\LKSgraphs{n}{k}{\eta}$. Write \index{mathsymbols}{*S@$\smallvertices{\eta}{k}{G}$}$\smallvertices{\eta}{k}{G}$ for the set of all vertices of $G$ that have degree less than $(1+\eta) k$, and set \index{mathsymbols}{*L@$\largevertices{\eta}{k}{G}$}$\largevertices{\eta}{k}{G}:=V(G)\setminus \smallvertices{\eta}{k}{G}$.

\begin{definition}\label{def:LKSsmall}
Let \index{mathsymbols}{*LKSsmallgraphs@$\LKSsmallgraphs{n}{k}{\eta}$}$\LKSsmallgraphs{n}{k}{\eta}$ be the class of those graphs $G\in\LKSgraphs{n}{k}{\eta}$ for which we have the following three properties:
\begin{enumerate}
   \item All the neighbours of every vertex $v\in V(G)$ with $\deg(v)>\lceil(1+2\eta)k\rceil$ have degrees at most $\lceil(1+2\eta)k\rceil$.\label{def:LKSsmallA}
   \item All the neighbours of every vertex of $\smallvertices{\eta}{k}{G}$
    have degree exactly $\lceil(1+\eta)k\rceil$. \label{def:LKSsmallB}
   \item We have $e(G)\le kn$.\label{def:LKSsmallC}
\end{enumerate}
\end{definition}

\section{Trees}\label{ssec:cut}
In this section we will show how to partition any given tree into small subtrees, connected by only a few vertices; this is what we call an {\it $\ell$-fine partition}. This notion is essential for our proof of
Theorem~\ref{thm:main},  as we can embed these small subtrees one at a time.

Similar but simpler tree-cutting procedures were used earlier for  the dense case of the
Loebl--Koml\'os--S\'os Conjecture~\cite{AKS95,HlaPig:LKSdenseExact,PS07+,Z07+}. There, the small trees were embedded in regular pairs of a
regularity lemma decomposition of the host graph $G$. Since here, we use the sparse decomposition instead,  we had to take more care when cutting the tree. (In particular, features~\eqref{2seeds},~\eqref{short},~\eqref{ellfine:separatedinternalshrubs} of Definition~\ref{ellfine} are needed for the more complex setting here.)

\smallskip
If $T$ is a tree and $r\in V(T)$, then the pair $(T,r)$ is a \index{general}{rooted
tree}\emph{rooted tree} with root $r$. We write 
\index{mathsymbols}{*Veven@$\Veven(T,r)$}\index{mathsymbols}{*Vodd@$\Vodd(T,r)$}$\Vodd(T,r)\subset V(T)$ for the set of
vertices of $T$ of odd distance from $r$. $\Veven(T,r)$ is defined analogously.
Note that $r\in\Veven(T,r)$. The distance between two vertices $v_1$ and $v_2$ in a tree is denoted by
\index{mathsymbols}{*DIST@$\dist(v_1,v_2)$}$\dist(v_1,v_2)$.

We start with a simple well-known fact about the number of leaves in a tree. For completeness we include a proof.
\begin{fact}\label{fact:treeshavemanyleaves3vertices}
Let $T$ be a tree with $\ell$ vertices of degree at
least three. Then $T$ has at least $\ell+2$ leaves.
\end{fact}
\begin{proof}
Let $D_1$ be the set of leaves, $D_2$ the set of vertices
of degree two and $D_3$ be the set of vertices of degree
of at least three. Then
$$2(|D_1|+|D_2|+|D_3|)-2=2v(T)-2=2e(T)=\sum_{v\in
V(T)}\deg(v)\ge |D_1|+2|D_2|+3|D_3|\;,$$
and the statement follows.
\end{proof}

Let $T$ be a tree rooted at $r$, inducing the partial order $\preceq$\index{mathsymbols}{*@$\preceq$} on $V(T)$ (with $r$ as the minimal element).
 If $a\preceq b$ and $ab\in E(T)$ then we say that $b$ is a \index{general}{child}{\em child of} $a$ and $a$ is the \index{general}{parent}{\em parent of} $b$.
\index{mathsymbols}{*Ch@$\children(v)$}$\children(a)$ denotes the set of children of $a$,
and the parent of a vertex $b\not=r$ is denoted
\index{mathsymbols}{*Par@$\parent(v)$}$\parent(b)$. For a set $U\subset V(T)$ write
\index{mathsymbols}{*Par@$\parent(U)$}$\parent(U):=\{\parent(u):u\in U\setminus\{r\}\}\setminus U$ and \index{mathsymbols}{*Ch@$\children(U)$}$\children(U):=\bigcup_{u\in
U}\children(u)\setminus U$.

The next simple fact has already appeared in~\cite{Z07+,HlaPig:LKSdenseExact} (and most likely in some more classic texts as well). Nevertheless, for completeness we give a proof here.
\begin{fact}\label{fact:treeshavemanyleaves}
Let $T$ be a tree with colour-classes $X$ and $Y$, and
$v(T)\ge 2$. Then the set $X$ contains at least
$|X|-|Y|+1$ leaves of $T$. 
\end{fact}
\begin{proof} Root $T$
at an arbitrary vertex $r\in Y$.
Let $I$ be the set of internal vertices of $T$ that belong to $X$. Each $v\in I$ has at least one immediate successor in the tree order induced by $r$.
These successors are distinct for distinct $v\in I$ and all lie in $Y\setminus\{r\}$. Thus
$|I|\le |Y|-1$. The claim follows.
\end{proof}

We say that a
tree $T'\subset T$ is \index{general}{induced tree}{\em induced} by
a vertex $x\in V(T)$ if $V(T')$ is the up-closure of $x$ in $V(T)$, i.e., $V(T')=\{v\in V(T)\: :\: x \preceq
v\}$. We then write \index{mathsymbols}{*T@$T(r,\uparrow x)$}$T'=T(r,\uparrow x)$, or $T'=T(\uparrow x)$, if the root
is obvious from the context and call $T'$ an \index{general}{end subtree}{\em end
subtree}. Subtrees of $T$ that are not end subtrees are called  \index{general}{internal
subtree}{\em internal subtrees}.

Let $T$ be a tree
rooted at $r$ and let $T'\subset T$ be a subtree with $r\not \in V(T')$. The
\index{general}{seed}{\em seed of
$T'$} is the $\preceq$-maximal  vertex $x\in V(T)\setminus V(T')$ for which
$x\preceq v$ for all $v\in V(T')$. We write \index{mathsymbols}{*Seed@$\seed$} $\seed(T')=x$. A~\emph{fruit}\index{general}{fruit} in a rooted tree $(T,r)$ is any vertex $u\in V(T)$ whose distance from $r$ is even and at least four.

\smallskip

We can now state the most important definition of this section, that of a fine partition of a tree. The idea behind this definition is that it will be easier to embed the tree if we do it piecewise. So we partition the tree~$T$ into small subtrees ($\shrubA\cup\shrubB$ in~\eqref{decompose} below) of bounded size (see~\eqref{small}), and a few cut-vertices (sets $W_A$ and $W_B$ in~\eqref{decompose} below). These cut-vertices lie between the subtrees. The partition of the cut-vertices into $W_A$ and $W_B$ is inherited from the bipartition of $T$ (see~\eqref{parity}). The partition $\shrubA$ and $\shrubB$ is given by the position (in $W_A$ or in $W_B$) of the cut-vertex (i.e., seed) of the small subtree (see~\eqref{nice} and~\eqref{cut:precede}).

It is of crucial importance that there are not too many seeds (cf.~\eqref{few}), as they will have to be embedded in special sets. Namely, the set that will accommodate $W_A$ needs to be well connected both to the set reserved for $W_B$, and to the area of the graph considered for embedding the subtrees from $\shrubA$. Another intuitively desirable property is \eqref{Bend}, as the internal subtrees will be more difficult to embed than the end subtrees. This is because they are adjacent to two seeds from $W_A\cup W_B$ and after embedding (a part) of the internal subtree, we need to come back to the sets reserved for $W_A\cup W_B$ to embed the second seed.

\begin{figure}[ht]
	\centering 
	\includegraphics{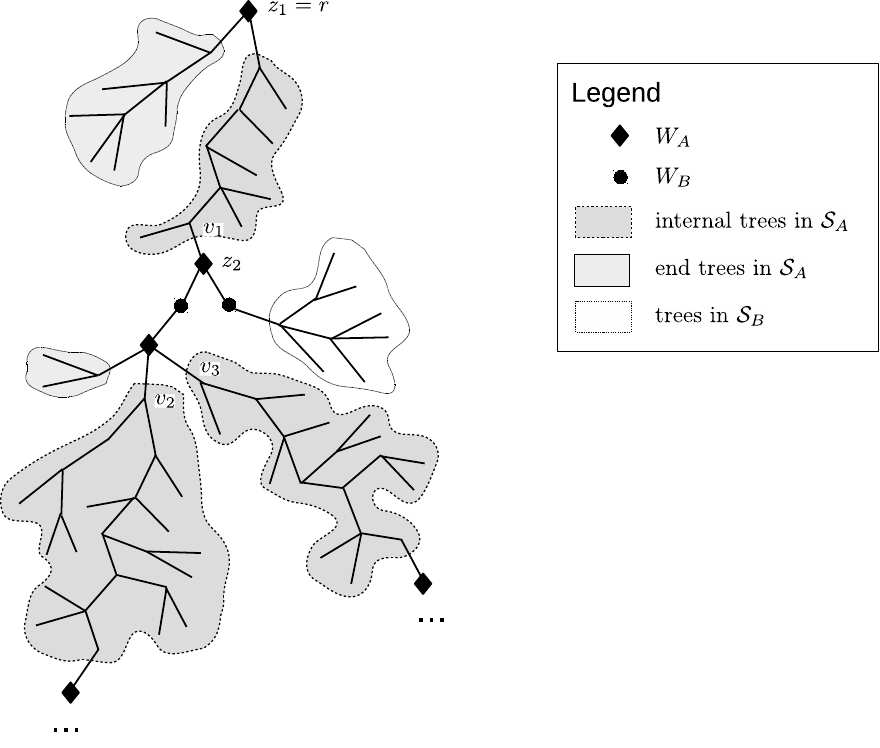}
	\caption{A part of an $\ell$-fine partition of a tree. Some of the properties from Definition~\ref{ellfine} are illustrated. The parities obey \eqref{parity} and \eqref{nice}. The distance between $z_1$ and $z_2$ is at least 6 as required in \eqref{short}. The distance between $v_1$ and $v_2$ is more than 2 as required in~\eqref{ellfine:separatedinternalshrubs} (since the corresponding subtrees precede one another). On the other hand,  \eqref{ellfine:separatedinternalshrubs} does not require the distance between $v_2$ and $v_3$ to be more than 2.}
	\label{fig:treedef}
\end{figure}

\begin{definition}[\bf $\ell$-fine partition]\label{ellfine}
Let $T\in \treeclass{k}$ be a tree rooted at~$r$. An
\index{general}{fine partition}{\em $\ell$-fine partition of~$T$} is a
quadruple $(W_A,W_B, \shrubA,
\shrubB)$, where $W_A,W_B\subseteq V(T)$ and $\shrubA$ and $\shrubB$ are families of subtrees of~$T$ such that
\begin{enumerate}[(a)]
\item  the three sets $W_A$, $W_B$ and $\{V(T^*)\}_{T^*\in\shrubA\cup
\shrubB}$ partition $V(T)$ (in particular, the trees in $T^*\in\shrubA\cup
\shrubB$ are pairwise vertex disjoint),\label{decompose}
\item $r\in W_A\cup W_B$,\label{root}
\item $\max\{|W_A|,|W_B|\}\leq 336k/{\ell}$,\label{few}
\item for $w_1,w_2\in W_A\cup W_B$ the distance $\dist(w_1,w_2)$ is odd if and only if one of them lies in $W_A$ and the other one in $W_B$,\label{parity}
\item $v(T^*)\leq \ell$ for every tree $T^*\in \shrubA\cup
\shrubB$,\label{small}
\item $V(T^*)\cap \neighbour(W_B)=\emptyset$ for every $T^*\in
\shrubA$ and $V(T^*)\cap \neighbour(W_A)=\emptyset$ for every $T^*\in
\shrubB$,\label{nice}
\item each tree of $\shrubA\cup\shrubB$ has its seeds in
$W_A\cup W_B$,\label{cut:precede}
\item  $|\neighbour(V(T^*))\cap (W_A\cup W_B)|\le 2$ for each $T^*\in\shrubA\cup \shrubB$,\label{2seeds}
\item if $\neighbour(V(T^*))\cap (W_A\cup W_B)$ contains two distinct vertices $z_1$ and $z_2$ for some  $T^*\in\shrubA\cup
\shrubB$, then $\dist_T(z_1,z_2)\ge 6$,\label{short}
\item if $T_1,T_2\in\shrubA\cup\shrubB$ are two internal subtrees of $T$ such that $v_1\in T_1$ precedes $v_2\in T_2$ 
then $\dist_T(v_1,v_2)>2$,\label{ellfine:separatedinternalshrubs}
\item $\shrubB$ does not contain any internal tree of $T$, and\label{Bend}
\item $\sum_{\substack{T^*\in
\shrubA,\ T^*\textrm{end subtree of $T$}}}v(T^*)\ge\sum_{T^*\in \shrubB}v(T^*)\;\mbox{.}$
\label{Bsmall}
\end{enumerate}
\end{definition}
An example is given in Figure~\ref{fig:treedef}.

\begin{remark}\label{rem:fruits}
Suppose that $(W_A,W_B, \shrubA,\shrubB)$ is an $\ell$-fine partition of a tree $(T,r)$, and suppose that $T^*\in\shrubA\cup\shrubB$ is such that
$|V(T^*)\cap \neighbour(W_A\cup W_B)|=2$. Let us root $T^*$ at the neighbour $r_1$ of its seed, and
let $r_2$ be the other vertex of $V(T^*)\cap \neighbour(W_A\cup W_B)$. Then
\eqref{parity},~\eqref{nice}, and~\eqref{short} imply that $r_2$ is a fruit in
$(T^*,r_1)$.
\end{remark}

The following is the main lemma of this section.

\begin{lemma}\label{lem:TreePartition}
Let $T\in \treeclass{k}$ be a tree rooted at $r$ and let $\ell\in
[k]$. Then $T$ has an $\ell$-fine partition.
\end{lemma}

\begin{proof}
First we shall use an inductive construction to get candidates for $W_A$, $W_B$,
$\shrubA$ and $\shrubB$, which we shall modify later on, so that they satisfy
all the  conditions required by Definition~\ref{ellfine}.

Set $T_0:=T$. Now, inductively for $i\ge 1$ choose a $\preceq$-maximal vertex
$x_i\in V(T_{i-1})$ with the property that $v(T_{i-1}(\uparrow x_i))>\ell$. We set $T_i:=T_{i-1}-(V(T_{i-1}(\uparrow x_i))\setminus
\{x_i\})$. If, say at step $i=i_\textrm{end}$, no such
$x_{i}$ exists, then $v(T_{i-1})\le \ell$. In that case,
set $x_i:=r$, set $W_1:=\{x_i\}_{i=1}^{i_\textrm{end}}$ and terminate. 
The fact that $v(T_{i-1}- V(T_i))\geq \ell$ for each $i<i_\textrm{end}$ implies
that
\begin{equation}\label{X}
|W_1|-1=i_\textrm{end}-1\le k/\ell\;.
\end{equation}

 Let $\mathcal{C}$  be the set of all
components of the forest $T-W_1$.
Observe that by the choice of the $x_i$ each
$T^*\in\mathcal C$ has order at most $\ell$. 

Let $A$ and $B$ be the colour classes of $T$ such that $r\in A$. Now, choosing
$W_A$ as $W_1\cap A$ and $W_B$ as $W_1\cap B$ and dividing $\mathcal C$
adequately into sets $\shrubA$ and $\shrubB$ would yield a quadruple that
satisfies conditions~\eqref{decompose}, \eqref{root}, \eqref{few},
\eqref{parity}, \eqref{small} and~\eqref{cut:precede}. To ensure the remaining properties,  we shall refine our tree partition by adding
more vertices to $W_1$, thus making the trees in $\shrubA\cup\shrubB$ smaller. In doing so, we have to be careful not to end up
violating~\eqref{few}. We shall enlarge the set of cut vertices in several
steps, accomplishing sequentially, in this order, also properties~\eqref{2seeds}, \eqref{ellfine:separatedinternalshrubs},
\eqref{nice},
\eqref{short}, and in the last step at the same time~\eqref{Bend}
and~\eqref{Bsmall}. It would be easy to check that during these steps none of the
previously established properties is lost, so we will not explicitly check them, except for~\eqref{few}.

For condition~\eqref{2seeds}, first define $T'$ as the subtree of $T$ that contains all vertices of $W_1$ and all vertices that lie on paths in $T$ which have both endvertices in $W_1$. Now, if a subtree $T^*\in\mathcal C$ does not already satisfy~\eqref{2seeds} for $W_1$, then $V(T^*)\cap V(T')$ must contain some  vertices of degree at least three. We will add the set $Y(T^*)$ of all these vertices to $W_1$. Formally, let $Y$ be the union of the sets $Y(T^*)$ over all $T^*\in\mathcal C$, and set $W_2:=W_1\cup Y$. Then the components of $T-W_2$ satisfy~\eqref{2seeds}.

Let us bound the size of the set $W_2$. For each $T^*\in\mathcal C$, note that by Fact~\ref{fact:treeshavemanyleaves3vertices} for $T^*\cap T'$, we know that $|Y(T^*)|$ is at most the number of leaves of $T^*\cap T'$ (minus two). On the other hand, each leaf of $T^*\cap T'$ has a child in $W_1$ (in $T$). As these children are distinct for different trees $T^*\in\mathcal C$, we find that $ |Y|\leq |W_1|$ and thus
\begin{equation}\label{Y}
 |W_2|\leq 2|W_1|\;.
\end{equation}

Next, for condition~\eqref{ellfine:separatedinternalshrubs}, observe that by setting $W_3:=W_2
\cup \parent_T(W_2)$ the components of $T-W_3$ fulfill~\eqref{ellfine:separatedinternalshrubs}. We have 
\begin{equation}\label{eq:YY}
|W_3|\le 2|W_2|\overset{\eqref{Y}}\le 4|W_1|\;.
\end{equation} 

In order to ensure condition~\eqref{nice}, let $R^*$ be the set of the roots ($\preceq$-minimal vertices) of those components $T^*$ of  $T-W_3$ that contain neighbours of both colour classes of $T$. Setting $W_4:=W_3\cup R^*$ we see that~\eqref{nice} is satisfied for $W_4$. Furthermore, as for each vertex in $R^*$ there is a distinct member of $W_3$ above it in the order on $T$, we obtain that
\begin{equation}\label{W3}
|W_4|\leq 2|W_3|\overset{\eqref{eq:YY}}\leq 8|W_1|.
\end{equation}

Next, we shall aim for a stronger version of property~\eqref{short}, namely,
\begin{enumerate}
\item[(\ref{short}')] if $\neighbour(V(T^*))\cap (W_A\cup
W_B)=\{z_1,z_2\}$ with $z_1\neq z_2$ for some  $T^* \in\shrubA\cup
\shrubB$, then $\dist_T(z_1,z_2)\ge 8$.\label{shortBetter}
\end{enumerate}
The reason for requiring this strengthening is that later we might introduce additional cut vertices which would
``shorten $T^*$ by two''.

Consider a component $T^*$ of $ T-W_4$ which is an
internal tree of $T$. If $W_4$ contains two distinct neighbours $z_1$ and $z_2$ of $T^*$ such that
$\dist_{T}(z_1,z_2)< 8$, then we call $T^*$ \emph{short}. Observe that there are at most
$|W_4|$ short trees, because each of these trees has a unique vertex from $W_4$
above it. Let $Z(T^*)\subset V(T^*)$ be the
vertices on the path from $z_1$ to $z_2$ (excluding the end vertices). Then
$|Z(T^*)|\le 7$. Letting $Z$ be the union of the sets $Z(T^*)$ over all short trees in
$T-W_4$, and set $W_5:=W_4\cup Z$, we obtain that
\begin{equation}\label{W4}
|W_5|\leq |W_4|+7|W_4|\overset{\eqref{W3}}\leq 64|W_1|\overset{\eqref{X}}\leq 64 k/\ell+1.
\end{equation}
We still need to ensure~\eqref{Bend} and~\eqref{Bsmall}. To this end, consider 
the set $\mathcal C'$ of all components of $T-W_5$. Set $\mathcal C'_A:=\{T^*\in \mathcal C':\seed(T^*)\in A\}$ and set $\mathcal C'_B:=\mathcal C'\setminus \mathcal C'_A$. We
assume that 
\begin{equation}\label{eq:AssumeEndTrees}
\sum_{T^*\in \mathcal C'_A\ : \ T^*\textrm{ end tree of $T$}} v(T^*)\ge \sum_{T^*\in \mathcal C'_B\ : \ T^*\textrm{ end tree of $T$}}
v(T^*)\;,
\end{equation}
 as otherwise we can simply swap $A$ and $B$.
Now, for each $T^*\in \mathcal C'_B$ that is not an end subtree of $T$, set $X(T^*):=V(T^*)\cap\neighbour_{T}(W_5)$. Let $X$ be the union of all such sets $X(T^*)$. Observe that 
\begin{equation}\label{XW5}
|X|\le 2|W_5\cap B|\le 2|W_5|.
\end{equation} 
For $W:=W_5 \cup X$,  all internal trees of $T-W$ have their seeds in $A$. This will guarantee~\eqref{Bend}, and, together with~\eqref{eq:AssumeEndTrees}, also~\eqref{Bsmall}.

 Finally, set $W_A:=W\cap A$ and $W_B:=W\cap B$, and
let $\shrubA$ and $\shrubB$ be the sets of those components of $T-W$ that have
their seeds in $W_A$ and $W_B$, respectively. By construction,
$(W_A,W_B,\shrubA,\shrubB)$ has all the properties of an $\ell$-fine partition. In particular, for~\eqref{few}, we find with~\eqref{W4} and~\eqref{XW5} that $|W|\leq |W_5|+2|W_5\cap B|\leq 336 k/\ell$.
\end{proof}

For an $\ell$-fine partition $(W_A,W_B,\shrubA,\shrubB)$ of a rooted tree $(T,r)$, the
trees $T^*\in \shrubA\cup\shrubB$ are called \index{general}{shrub}{\em shrubs}. An
\emph{end shrub} is a shrub which is an end subtree. An \emph{internal shrub} is a shrub which is an internal subtree.  A \emph{\kknnaaggNOSPACE}\index{general}{\kknnaaggNOSPACE} is a component of the forest $T[W_A\cup W_B]$. Suppose that $T^*\in \shrubA$ is an internal shrub, and $r^*$ is its $\preceq_r$-minimal vertex. Then $T^*-r^*$ contains a unique component with a vertex from $\neighbour_T(W_A)$. We call this component \index{general}{subshrub}\index{general}{principal subshrub}\emph{principal subshrub}, and the other components \index{general}{peripheral subshrub}\emph{peripheral subshrubs}.

\begin{remark}\label{rem:internalVSend}
\begin{enumerate}[(i)]
\item 
In our proof of Theorem~\ref{thm:main}, we shall apply
Lemma~\ref{lem:TreePartition} to a tree $T_\PARAMETERPASSING{T}{thm:main}\in\treeclass{k}$. The number
$\ell_\PARAMETERPASSING{L}{lem:TreePartition}$ will  be linear in $k$, and
thus~\eqref{few} of Definition~\ref{ellfine} tells us that the size of the sets $W_A$ and
$W_B$ is bounded by an absolute constant (depending on $\alpha_\PARAMETERPASSING{T}{thm:main}$ only).
\item\label{it:fewinternaltrees}
Each internal tree in $\shrubA$ of an $\ell$-fine partition has a unique vertex from $W_A$ above it. Thus with $\ell_\PARAMETERPASSING{L}{lem:TreePartition}$ as above also the number of internal trees in $\shrubA$ is bounded by an absolute constant. This need not be the case for the number of end trees. For instance, if $(T_\PARAMETERPASSING{T}{thm:main},r)$ is a star with $k-1$ leaves and rooted at its centre $r$ then $W_A=\{r\}$ while the $k-1$ leaves of $T_\PARAMETERPASSING{T}{thm:main}$ form the end shrubs in $\shrubA$.
\end{enumerate} 
\end{remark}

\begin{definition}[\bf ordered skeleton]\index{general}{ordered
skeleton} We say that the sequence $\big(X_0,X_1,\ldots, X_m\big)$
is an \emph{ordered skeleton} of the $\ell$-fine partition
$(W_A,W_B,\shrubA,\shrubB)$ of a rooted tree $(T,r)$ if 
\begin{itemize}
 \item $X_0$ is a \kknnaagg and contains~$r$, and all other $X_i$ are either \kknnaaggss or shrubs, 
 \item $V(\bigcup_{i\le m}X_i)=V(T)$, and
 \item for each $i=1,\ldots,m$, the subgraph formed by
  $X_0\cup X_1\cup\ldots\cup X_i$  is connected in $T$.
  \end{itemize}
\end{definition}
Directly from Definition~\ref{ellfine} we get:
\begin{lemma}\label{lem:orderedskeleton}
Any $\ell$-fine
partition of any rooted tree has  an ordered skeleton. 
\end{lemma}

\section{Necessary facts and notation from~\cite{cite:LKS-cut0,cite:LKS-cut1,cite:LKS-cut2}}\label{sec:fromearlier}

\subsection{Sparse decomposition}
We now shift our focus from preprocessing the tree to the host graph. This is where we build on results from the earlier papers in the series.
We first recall the notion of dense spots and related concepts introduced in~\cite{cite:LKS-cut0},~\cite{cite:LKS-cut1}, and~\cite{cite:LKS-cut2}.

\begin{definition}[\bf \index{general}{dense spot}$(m,\gamma)$-dense spot,
\index{general}{nowhere-dense}$(m,\gamma)$-nowhere-dense]\label{def:densespot}
Suppose that $m\in\NN$ and $\gamma>0$.
An \emph{$(m,\gamma)$-dense spot} in a graph $G$ is a non-empty bipartite sub\-graph  $D=(U,W;F)$ of  $G$ with
$\density(D)>\gamma$ and $\mindeg (D)>m$. We call $G$
\emph{$(m,\gamma)$-nowhere-dense} if it does not contain any $(m,\gamma)$-dense spot.
\end{definition}

\begin{definition}[\bf $(m,\gamma)$-dense
cover]\index{general}{dense cover}
Suppose that $m\in\NN$ and $\gamma>0$.
An \emph{$(m,\gamma)$-dense cover} of a
graph $G$ is a family $\DenseSpots$ of edge-disjoint
$(m,\gamma)$-dense
spots such that $E(G)=\bigcup_{D\in\DenseSpots}E(D)$.
\end{definition}

The proofs of the following facts can be found in~\cite{cite:LKS-cut1}.

\begin{fact}\label{fact:sizedensespot}
Let $(U,W;F)$ be a $(\gamma k,\gamma)$-dense spot in a
graph $G$ of maximum degree at most $\Omega k$. Then
$\max\{|U|,|W|\}\le \frac{\Omega}{\gamma}k.$
\end{fact}

\begin{fact}\label{fact:boundedlymanyspots}
Let $H$ be a graph of maximum degree at most $\Omega k$, let $v\in V(H)$, and let $\DenseSpots$ be a family of edge-disjoint $(\gamma k,\gamma)$-dense spots. Then fewer than $\frac{\Omega}{\gamma}$ dense spots from $\DenseSpots$ contain $v$.
\end{fact}

%

%

In the following definition,
note that a subset of a $(\Lambda,\epsilon,\gamma,k)$-avoiding set is also $(\Lambda,\epsilon,\gamma,k)$-avoiding.

\begin{definition}[\bf
\index{general}{avoiding (set)}$(\Lambda,\epsilon,\gamma,k)$-avoiding set]\label{def:avoiding}
Suppose that $k\in\NN$, $\epsilon,\gamma>0$ and $\Lambda>0$. Suppose that
$G$ is a graph and $\DenseSpots$ is a family of dense spots in $G$. A set
$\smallatoms\subset \bigcup_{D\in\DenseSpots} V(D)$ is \emph{$(\Lambda,\epsilon,\gamma,k)$-avoiding} with
respect to $\DenseSpots$ if for every $U\subset V(G)$ with $|U|\le \Lambda k$ the following holds for all but at most $\epsilon k$ vertices $v\in\smallatoms$. There is a dense spot $D\in\DenseSpots$ with $|U\cap V(D)|\le \gamma^2 k$ that contains $v$.
\end{definition}

In the next two definitions, we expose the most important tool in the proof of our main result (Theorem~\ref{thm:main}): the \emph{sparse decomposition}. It generalises the notion of equitable partition from Szemer\'edi's regularity lemma. This is explained in~\cite[Section~\ref{p0.ssec:sparsedecompofdensegraphs}]{cite:LKS-cut0}. The first step to this end is defining the bounded decomposition.

\begin{definition}[\index{general}{bounded decomposition}{\bf
$(k,\Lambda,\gamma,\epsilon,\nu,\rho)$-bounded decomposition}]\label{bclassdef}
Suppose that $k\in\NN$ and $\epsilon,\gamma,\nu,\rho>0$ and $\Lambda>0$. 
Let $\mathcal V=\{V_1, V_2,\ldots, V_s\}$ be a partition of the vertex set of a graph $G$. We say that $( \clusters,\DenseSpots, \Gblack, \Gexp,
\smallatoms )$ is a {\em $(k,\Lambda,\gamma,\epsilon,\nu,\rho)$-bounded
decomposition} of $G$ with respect to $\mathcal V$ if the following properties
are satisfied:
\begin{enumerate}
\item\label{defBC:nowheredense}
$\Gexp$ is  a $(\gamma k,\gamma)$-nowhere-dense subgraph of $G$ with $\mindeg(\Gexp)>\rho k$.
\item\label{defBC:clusters} The elements of $\clusters$ are pairwise disjoint subsets of 
$ V(G)$.
\item\label{defBC:RL} $\Gblack$ is a subgraph of $G-\Gexp$ on the vertex set $\bigcup \clusters$. For each edge
 $xy\in E(\Gblack)$ there are distinct $C_x\ni x$ and $C_y\ni y$ from $\clusters$,
and  $G[C_x,C_y]=\Gblack[C_x,C_y]$. Furthermore, 
$G[C_x,C_y]$ forms an $\epsilon$-regular pair of  density at least $\gamma^2$.
\item We have $\nu k\le |C|=|C'|\le \epsilon k$ for all
$C,C'\in\clusters$.\label{Csize}
\item\label{defBC:densepairs}  $\DenseSpots$ is a family of edge-disjoint $(\gamma
k,\gamma)$-dense spots  in $G-\Gexp$.  For
each $D=(U,W;F)\in\DenseSpots$ all the edges of $G[U,W]$ are covered
by $\DenseSpots$ (but not necessarily by $D$).
\item\label{defBC:dveapul} If  $\Gblack$
contains at least one edge between $C_1,C_2\in\clusters$, then there exists a dense
spot $D=(U,W;F)\in\DenseSpots$ such that $C_1\subset U$ and $C_2\subset
W$.
\item\label{defBC:prepartition}
For
all $C\in\clusters$ there is $V\in\mathcal V$ so that either $C\subseteq V\cap V(\Gexp)$ or $C\subseteq V\setminus V(\Gexp)$.
For
all $C\in\clusters$ and $D=(U,W; F)\in\DenseSpots$ we have $C\cap U,C\cap W\in\{\emptyset, C\}$.
\item\label{defBC:avoiding}
$\smallatoms$ is a $(\Lambda,\epsilon,\gamma,k)$-avoiding subset  of
$V(G)\setminus \bigcup \clusters$ with respect to dense spots $\DenseSpots$.
\end{enumerate}

\smallskip
We say that the bounded decomposition $(\clusters,\DenseSpots, \Gblack, \Gexp,
\smallatoms )$ {\em respects the avoiding threshold~$b$}\index{general}{avoiding threshold} if for each $C\in \clusters$ we either have $\maxdeg_G(C,\smallatoms)\le b$, or $\mindeg_G(C,\smallatoms)> b$.
\end{definition}

The members of $\clusters$ are called \index{general}{cluster}{\it clusters}. Define the
\index{general}{cluster graph}{\it cluster graph} \index{mathsymbols}{*Gblack@$\BGblack$}  $\BGblack$ as the graph
on the vertex set $\clusters$ that has an edge $C_1C_2$
for each pair $(C_1,C_2)$ which has density at least $\gamma^2$ in the graph
$\Gblack$. 
%
%

\begin{definition}[\bf \index{general}{sparse
decomposition}$(k,\Omega^{**},\Omega^*,\Lambda,\gamma,\epsilon,\nu,\rho)$-sparse decomposition]\label{sparseclassdef}
Suppose that $k\in\NN$ and $\epsilon,\gamma,\nu,\rho>0$ and $\Lambda,\Omega^*,\Omega^{**}>0$. 
Let $\mathcal V=\{V_1, V_2,\ldots, V_s\}$ be a partition of the vertex set of a graph $G$. We say that 
$\class=(\HugeVertices, \clusters,\DenseSpots, \Gblack, \Gexp, \smallatoms )$
is a 
{\em $(k,\Omega^{**},\Omega^*,\Lambda,\gamma,\epsilon,\nu,\rho)$-sparse decomposition} of $G$
with respect to $V_1, V_2,\ldots, V_s$ if the following holds.
\begin{enumerate}
\item\label{def:classgap} $\HugeVertices\subset V(G)$,
$\mindeg_G(\HugeVertices)\ge\Omega^{**}k$,
$\maxdeg_H(V(G)\setminus \HugeVertices)\le\Omega^{*}k$, where $H$ is spanned by the edges of $\bigcup\DenseSpots$, $\Gexp$, and
edges incident with $\HugeVertices$,
\item \label{def:spaclahastobeboucla} $( \clusters,\DenseSpots, \Gblack,\Gexp,\smallatoms)$ is a 
$(k,\Lambda,\gamma,\epsilon,\nu,\rho)$-bounded decomposition of
$G-\HugeVertices$ with respect to $V_1\setminus \HugeVertices, V_2\setminus \HugeVertices,\ldots, V_s\setminus \HugeVertices$.
\end{enumerate}
\end{definition}

If the parameters do not matter, we call $\class$ simply a {\em sparse
decomposition}, and similarly we speak about a {\em bounded decomposition}.
We define the graph \index{mathsymbols}{*GD@$\GD$}$\GD$ as the union (both edge-wise, and vertex-wise) of all dense spots~$\DenseSpots$.

\begin{fact}[{\cite[Fact~\ref{p0.fact:clustersSeenByAvertex}]{cite:LKS-cut0}}]\label{fact:clustersSeenByAvertex}
Let  $\class=(\HugeVertices, \clusters,\DenseSpots, \Gblack, \Gexp,\smallatoms )$ be a 
$(k,\Omega^{**},\Omega^*,\Lambda,\gamma,\epsilon,\nu,\rho)$-sparse
decomposition of a graph  $G$. Let $x\in V(G)\setminus \HugeVertices$. Assume that $\clusters\not=\emptyset$, and let $\clustersize$ be the size of each of the members of~$\clusters$. Then there are fewer than
$$\frac{2(\Omega^*)^2k}{\gamma^2 \clustersize}\le\frac{2(\Omega^*)^2}{\gamma^2\nu}$$ clusters
$C\in\clusters$ with $\deg_{\GD}(x,C)>0$.
\end{fact}

 \begin{definition}[\bf \index{general}{captured edges}captured edges]\label{capturededgesdef}
In the situation of Definition~\ref{sparseclassdef}, we refer to the edges in
$ E(\Gblack)\cup E(\Gexp)\cup
E_G(\HugeVertices,V(G))\cup E_{\GD}(\smallatoms,\smallatoms\cup \bigcup \clusters)$
as \index{general}{captured edges}{\em captured} by the sparse decomposition. 
Denote by
\index{mathsymbols}{*Gclass@$\Gcapt$}$\Gcapt$ the spanning subgraph of $G$ whose edges are the captured edges of the sparse decomposition.
Likewise, the captured edges of a bounded decomposition
$(\clusters,\DenseSpots, \Gblack,\Gexp,\smallatoms )$ of a graph $G$ are those
in $E(\Gblack)\cup E(\Gexp)\cup E_{\GD}(\smallatoms,\smallatoms\cup\bigcup\clusters)$.
\end{definition}

The last definition we need is the notion of a \semiregular matching.

\begin{definition}[\bf $(\epsilon,d,\ell)$-\semiregular matching]\label{def:semiregular}
Suppose that $\ell\in\NN$ and  $d,\epsilon>0$.
 A collection
$\mathcal N$ of pairs $(A,B)$ with $A,B\subset V(H)$ is called an
\index{general}{regularized@\semiregular matching}\emph{$(\epsilon,d,\ell)$-\semiregular matching} of a
graph $H$ if \begin{enumerate}[(i)] \item $|A|=|B|\ge \ell$ for each
$(A,B)\in\mathcal N$, \item $(A,B)$ induces in $H$ an $\epsilon$-regular pair of
density at least $d$, for each $(A,B)\in\mathcal N$, and \item all involved sets
$A$ and $B$ are  pairwise disjoint.
\end{enumerate}

Sometimes, when the parameters do not matter we simply write  \emph{\semiregular matching}.

We say that a \semiregular matching $\mathcal N$
\index{general}{absorb}\emph{absorbes} a \semiregular matching $\mathcal M$ if for every $(S,T)\in\mathcal M$ there exists $(X,Y)\in\mathcal N$ such that $S\subset X$
and $T\subset Y$. In the same way, we say that a family of dense spots $\mathcal D$
\index{general}{absorb}\emph{absorbes} a \semiregular matching $\mathcal M$ if for every $(S,T)\in\mathcal M$ there exists $(U,W;F)\in\mathcal D$ such that $S\subset U$
and $T\subset W$.
\end{definition}

\begin{fact}[{\cite[Fact~\ref{p1.fact:boundMatchingClusters}]{cite:LKS-cut1}}]\label{fact:boundMatchingClusters}
Suppose that $\mathcal M$ is an
$(\epsilon,d,\ell)$-\semiregular
matching in a graph $H$. Then $|C|\le
\frac{\maxdeg{(H)}}{d}$ for each $C\in
\mathcal V(\mathcal M)$.
\end{fact}

\subsection{Shadows }

We recall the notion of a shadow given in~\cite{cite:LKS-cut2}. Given a graph $H$, a set
$U\subset V(H)$, and a number $\ell$ we define inductively
\index{mathsymbols}{*SHADOW@$\shadow(U,\ell)$}\index{general}{shadow}
\begin{align*}
& \shadow^{(0)}_H(U,\ell):=U\text{, and }\\
& \shadow^{(i)}_H(U,\ell):=\{ v\in V(H)\::\:
\deg_H(v,\shadow^{(i-1)}_H(U,\ell))>\ell \} \text{ for }i\geq 1.
\end{align*}
\index{general}{exponent of shadow}We call the index $i$ the \emph{exponent} of the shadow. We abbreviate $\shadow^{(1)}_H(U,\ell)$ as $\shadow_H(U,\ell)$. Further, the
graph $H$ is omitted from the subscript if it is clear from the context.
Note that the shadow of a set $U$ might intersect $U$.

The proofs of the following facts can be found in~\cite{cite:LKS-cut2}. 

\begin{fact}\label{fact:shadowbound}
Suppose $H$ is a graph with $\maxdeg(H)\le \Omega k$. Then for each
$\alpha>0, i\in\{0,1,\ldots\}$, and each set $U\subset V(H)$, we have
$$|\shadow^{(i)}(U,\alpha k)|\le \left(\frac{\Omega}{\alpha}\right)^i|U|\;.$$
\end{fact}

\begin{fact}\label{fact:shadowboundEXPANDER}
Let $\alpha,\gamma,Q>0$ be three numbers such that $1\le Q\le
\frac{\alpha}{16\gamma}$. Suppose that $H$ is a $(\gamma k,\gamma)$-nowhere-dense graph, and let $U\subset V(H)$ with $|U|\le Qk$. Then we have $$|\shadow(U,\alpha
k)|\le \frac{16Q^2\gamma}{\alpha}k.$$
\end{fact}

\section{Configurations}\label{FromEarlierPapersOfTheSeries}

\subsection{Common settings}
Recall the definitions of $\smallvertices{\eta}{k}{G}$ and $\largevertices{\eta}{k}{G}$ given in Section~\ref{subsection:LKSgraphs}.
We repeat some common settings that already appeared in~\cite{cite:LKS-cut2} and are outputs of~\cite[Lemma~\ref{p1.prop:LKSstruct}]{cite:LKS-cut1}. The reader can find explanations in~\cite[Section~\ref{p1.ssec:motivation}]{cite:LKS-cut1} why the set $\XA$ (defined again in~\eqref{eq:XAXBXCMik}) has excellent properties for accommodating cut-vertices of $T_\PARAMETERPASSING{T}{thm:main}$, and the set $\XB$ has ``half-that-excellent properties'' for accommodating cut-vertices. In particular, the formula defining $\XB$ suggests that we cannot make use of the set $\smallvertices{\eta}{k}{G}\setminus
(V(\Gexp)\cup\smallatoms\cup V(\mathcal M_A\cup\mathcal M_B)$ for the purpose of embedding shrubs neighbouring the  cut-vertices embedded in $\XB$. In~\cite[Setting~\ref{p2.commonsetting}]{cite:LKS-cut2} we gave some  motivation behind the definition of the sets $V_+, L_\#, \Vgood, \YA, \YB, \WantiC, \gPatoms, \gP,\gP_1,  \gP_2, \gP_3 $, and $ \mathcal F$ in Setting~\ref{commonsetting}, below. 

\begin{setting}\label{commonsetting}
We assume that the constants $\Lambda,\Omega^*,\Omega^{**},k_0$ and $\alphaD,\gamma,\epsilon,\epsilon',\eta,\pi,\rho, \tau, d$ satisfy
\begin{align}
\label{eq:KONST}
\begin{split}
 \eta\gg\frac1{\Omega^*}\gg \frac1{\Omega^{**}}\gg\rho\gg\gamma\gg
d \ge\frac1{\Lambda}\ge\epsilon\ge
\pi\ge  \alphaD
\ge \epsilon'\ge
\nu\gg \tau \gg \frac{1}{k_0}>0\;,
\end{split}
\end{align}
and that $k\ge k_0$.
By writing $c>a_1\gg
a_2\gg \ldots \gg a_\ell>0$ we mean that there exist suitable non-decreasing functions
$f_i:(0,c)^i\rightarrow (0,c)$ ($i=1,\ldots,\ell-1$) such that for each $i\in
[\ell-1]$ we have $a_{i+1}<f_{i}(a_1,\ldots,a_i)$. A suitable choice of these functions in~\eqref{eq:KONST} is explicitly given in Section~\ref{sec:proof}.

\medskip

Suppose that $G\in\LKSsmallgraphs{n}{k}{\eta}$ is given together with its
$(k,\Omega^{**},\Omega^*,\Lambda,\gamma,\epsilon',\nu,\rho)$-sparse
decomposition \index{mathsymbols}{*@$\class$} $$\class=(\HugeVertices, \clusters,\DenseSpots, \Gblack,
\Gexp,\smallatoms )\;, $$
with respect to the partition
$\{\smallvertices{\eta}{k}{G},\largevertices{\eta}{k}{G}\}$, and with respect to the avoiding threshold $\frac{\rho k}{100\Omega^*}$. We write 
\begin{equation}
 \index{mathsymbols}{*VA@$\largeintoatoms$}\index{mathsymbols}{*VA@$\clustersintoatoms$}
\largeintoatoms:=\shadow_{\Gcapt-\HugeVertices}(\smallatoms,\frac{\rho k}{100\Omega^*})\quad\mbox{and}\quad\clustersintoatoms:=\{C\in\clusters\::\:C\subset \largeintoatoms\}\;.
\label{eq:deflargeintoatoms}
\end{equation}
The graph \index{mathsymbols}{*Gblack@$\BGblack$}$\BGblack$ is the corresponding cluster graph. Let $\clustersize$
\index{mathsymbols}{*C@$\clustersize$}
be the size of an arbitrary cluster in $\clusters$.\footnote{The number $\clustersize$ is not defined when $\clusters=\emptyset$. However in that case $\clustersize$ is never actually used.}
 Let \index{mathsymbols}{*G@$\Gcapt$}$\Gcapt$ be the spanning subgraph of $G$ formed by the edges captured by the sparse decomposition~$\class$. There are two $(\epsilon,d,\pi \clustersize)$-\semiregular matchings $\mathcal M_A$
and $\mathcal M_B$ in $\GD$, with the following
properties. Following~\cite[\eqref{p1.def:XAXBXC}]{cite:LKS-cut1} we write
\begin{align}
\begin{split}\label{eq:XAXBXCMik}
\XA&:=\largevertices{\eta}{k}{G}\setminus V(\mathcal M_B)\;,\\
\XB&:=\left\{v\in V(\mathcal M_B)\cap \largevertices{\eta}{k}{G}\::\:\widehat{\deg}(v)<(1+\eta)\frac
k2\right\}\;,\\
\XC&:=\largevertices{\eta}{k}{G}\setminus(\XA\cup\XB)\;,
\end{split}
\end{align}
where 
$\widehat{\deg}(v)$ on the second
line is defined by
\begin{equation*}
\widehat{\deg}(v):=\deg_G\big(v,\smallvertices{\eta}{k}{G}\setminus
(V(\Gexp)\cup\smallatoms\cup V(\mathcal M_A\cup\mathcal M_B)\big)\;.
\end{equation*}
Then we have
\begin{enumerate}[(1)]
\item\label{commonsetting1}
$V(\mathcal M_A)\cap V(\mathcal
M_B)=\emptyset$,
\item\label{commonsetting1apul}
$V_1(\mathcal \M_B)\subset S^0$, where
\begin{equation}\label{eq:defS0}
S^0:=\smallvertices{\eta}{k}{G}\setminus
(V(\Gexp)\cup\smallatoms)\;,
\end{equation}
\item\label{commonsetting2} for each $(X,Y)\in \M_A\cup\M_B$, there is a dense
spot $(U,W; F)\in \DenseSpots$ with $X\subset U$ and $Y\subset W$, and further,
either $X\subset \smallvertices{\eta}{k}{G}$ or $X\subset
\largevertices{\eta}{k}{G}$, and $Y\subset \smallvertices{\eta}{k}{G}$ or
$Y\subset \largevertices{\eta}{k}{G}$,
\item\label{commonsetting3}
for each $X_1\in\V_1(\M_A\cup\M_B)$ there exists a cluster $C_1\in \clusters$ such that $X_1\subset C_1$, and for each $X_2\in\V_2(\M_A\cup\M_B)$ there exists $C_2\in \clusters\cup\{\largevertices{\eta}{k}{G}\cap \smallatoms\}$ such that $X_2\subset C_2$,
\item\label{commonsettingMgood} each pair of the \semiregular matching
$\Mgood:=\{(X_1,X_2)\in\M_A\::\: X_1\cup X_2\subset \XA\}$ corresponds to an
edge in $\BGblack$,
\item\label{commonsettingXAS0}$e_{\Gcapt}\big(\XA,S^0\setminus V(\mathcal
M_A)\big)\le \gamma kn$,
\item\label{commonsetting4} $e_{\Gblack}(V(G)\setminus V(\M_A\cup \M_B))\le
\gamma^2kn$,
\item\label{commonsettingNicDoNAtom} for the \semiregular matching \index{mathsymbols}{*Natom@$\NAtom$}
$\NAtom:=\{(X,Y)\in\M_A\cup\M_B\::\: (X\cup Y)\cap\smallatoms\not=\emptyset\}$ we have $e_{\Gblack}\big(V(G)\setminus
V(\M_A\cup \M_B),V(\NAtom)\big)\le \gamma^2 kn$,
\item \label{commonsetting:numbercaptured}$|E(G)\setminus E(\Gcapt)|\le 2\rho
kn$, 
\item\label{commonsetting:DenseCaptured}$|E(\GD)\setminus (E(\Gblack)\cup
E_G[\smallatoms,\smallatoms\cup\bigcup\clusters])|\le \frac 54\gamma kn$.
\end{enumerate}

We write
\begin{align}
\label{eq:defV+}\index{mathsymbols}{*V_+@$V_+$}
V_+&:=V(G)\setminus 
(S^0\setminus V(\mathcal M_A\cup\mathcal
M_B)) \\  \label{defV+eq}
& = \largevertices{\eta}{k}{G}\cup
V(\Gexp)\cup\smallatoms\cup V(\mathcal M_A\cup\mathcal
M_B)\;,\\
\label{eq:defLsharp}\index{mathsymbols}{*L@$L_\#$}
L_\#&:=\largevertices{\eta}{k}{G}\setminus\largevertices{\frac9{10}\eta}{k}{\Gcapt}\;\mbox{,
and}\\ 
\label{eq:defVgood}\index{mathsymbols}{*Vgood@$\Vgood$} 
\Vgood&:=V_+\setminus (\HugeVertices\cup L_\#)\;,\\
\label{eq:defYA}
\YA&:=
 \shadow_{\Gcapt}\left(V_+\setminus L_\#, (1+\frac\eta{10})k\right) \setminus \shadow_{G-\Gcapt}\left(V(G),\frac\eta{100} k\right)\;,
 \index{mathsymbols}{*YA@$\YA$}\\
\label{eq:defYB} 
\YB&:=
 \shadow_{\Gcapt}\left(V_+\setminus L_\#, (1+\frac\eta{10})\frac k2\right) \setminus
 \shadow_{G-\Gcapt}\left(V(G),\frac\eta{100} k\right)\;,\index{mathsymbols}{*YB@$\YB$}\\
\WantiC&:=(\XA\cup\XB)\cap \shadow_G\left(\HugeVertices, \frac{\eta}{100}
k \right) \index{mathsymbols}{*V@$\WantiC$}\label{eq:defWantiC}\;,
\\ \nonumber
\gPatoms&:=\shadow_{\Gblack}(V(\NAtom),\gamma k)\setminus V(\M_A\cup\M_B)\;,
\index{mathsymbols}{*JE@$\gPatoms$}\\
\nonumber
\gP_1&:=\shadow_{\Gblack}(V(G)\setminus V(\M_A\cup\M_B),\gamma k)\setminus V(\M_A\cup\M_B)\;,
\index{mathsymbols}{*J1@$\gP_1$}
\\ \nonumber
\gP&:=(\XA\setminus \YA)\cup((\XA\cup \XB)\setminus \YB)\cup\WantiC
\cup
L_\sharp \cup \gP_1\\
\nonumber
&~~~~~\cup\shadow_{\GD\cup\Gcapt}(\WantiC\cup
L_\sharp\cup\gPatoms\cup\gP_1,\frac{\eta^2 k}{10^5})\;,
\index{mathsymbols}{*J@$\gP$}
\\
\nonumber
\gP_2&:=\XA\cap \shadow_{\Gcapt}(S^0\setminus
V(\M_A),\sqrt\gamma k)\;,
\index{mathsymbols}{*J2@$\gP_2$}\\
\nonumber
\gP_3&:=\XA\cap\shadow_{\Gcapt}(\XA, \eta^3k/10^3)\;,
\index{mathsymbols}{*J3@$\gP_3$}\\
\label{def:Fcover}
 \mathcal F&:=\{C\in \V(\M_A):C\subset \XA\}\cup\V_1(\M_B)\;.\index{mathsymbols}{*F@$\mathcal F$}
\end{align}
\end{setting}


For the embedding procedure to run smoothly, the vertex set is split into several classes the sizes of which have given ratios.
It will be important that most vertices have their degrees split according to these ratios. Lemma~\ref{lem:randomSplit} allows us to do so. 
The motivation behind Lemma~\ref{lem:randomSplit} and Definition~\ref{def:proportionalsplitting}, below,  is explained more in details at the beginning of~\cite[Section~\ref{p2.ssec:RandomSplittins}]{cite:LKS-cut2}.

\begin{lemma}\label{lem:randomSplit} For each $p\in\NN$ and $a>0$ there exists $k_0>0$ such that
for each $k>k_0$ we have the following.

Suppose $G$ is a graph of order $n\ge k_0$ and $\maxdeg(G)\le\Omega^*k$ with its 
$(k,\Lambda,\gamma,\epsilon,k^{-0.05},\rho)$-bounded
decomposition $( \clusters,\DenseSpots, \Gblack, \Gexp, \smallatoms )$. As
usual, we write $\Gcapt$ for the subgraph captured by $( \clusters,\DenseSpots,
\Gblack, \Gexp, \smallatoms )$, and $\GD$ for the spanning subgraph of $G$
consisting of the edges in $\DenseSpots$. Let $\M$ be an
$(\epsilon,d,k^{0.95})$-\semiregular matching in $G$, and $\VXV_1,\ldots, \VXV_p$
be subsets of $V(G)$. Suppose that $\Omega^*\ge 1$ and $\Omega^*/\gamma<k^{0.1}$.

Suppose that $\aXa_1,\ldots,\aXa_p\in\{0\}\cup[a,1]$ are reals with $\sum \aXa_i\le 1$.
Then there exists a partition $\AXA_1\cup \ldots\cup \AXA_p=V(G)$, and sets $\bar
V\subset V(G)$, $\bar\V\subset \V(\M)$, and $\bar\clusters\subset\clusters$ with the following properties.
\begin{enumerate}[(1)]
  \item\label{It:H1} $|\bar V|\le \exp(-k^{0.1})n$,
  $|\bigcup\bar\V|\le \exp(-k^{0.1})n$,
  $|\bigcup\bar\clusters|<\exp(-k^{0.1})n$.
  \item\label{It:H2} For each $i\in [p]$ and each $C\in
  \clusters\setminus\bar\clusters$ we have $|C\cap \AXA_i|\ge
  \aXa_i|\AXA_i|-k^{0.9}$.
  \item\label{It:H3} For each $i\in [p]$ and each $C\in \V(\M)\setminus\bar\V$
  we have $|C\cap \AXA_i|\ge \aXa_i|\AXA_i|-k^{0.9}$.
  \item\label{It:H4} For each $i\in [p]$, $D=(U,W; F)\in\DenseSpots$
   and 
   $\mindeg_D (U\setminus \bar V,W\cap \AXA_i)\ge \aXa_i\gamma k-k^{0.9}$.
  \item\label{It:HproportionalSizes} For each $i,j\in[p]$ we have $|\AXA_i\cap
  \VXV_j|\ge \aXa_i|\VXV_j|-n^{0.9}$.
  \item\label{It:H5} For each $i\in [p]$ each $J\subset[p]$ and each $v\in V(G)\setminus \bar V$ we have $$\deg_H(v,\AXA_i\cap \VXV_J)\ge \aXa_i\deg_H(v,\VXV_J)-2^{-p}k^{0.9}\;,$$ for each of the graphs $H\in\{G,\Gcapt,\Gexp,\GD,\Gcapt\cup\GD\}$,
  where 
$\VXV_J:=\big(\bigcap_{j\in J}\VXV_j\big)\sm \big(\bigcup_{j\in [p]\setminus J}
\VXV_j\big)$.
  \item\label{It:H6} For each $i,i',j,j'\in [p]$ ($j\neq j'$), we have
\begin{align*}
e_H(\AXA_i\cap \VXV_{j},\AXA_{i'}\cap \VXV_{j'})&\ge
  \aXa_i\aXa_{i'} e_H(\VXV_j,\VXV_{j'})-k^{0.6}n^{0.6}\;,\\
e_H(\AXA_i\cap \VXV_{j},\AXA_{i'}\cap \VXV_{j})&\ge
  \aXa_i\aXa_{i'} e(H[\VXV_j])-k^{0.6}n^{0.6}\qquad\mbox{if $i\neq i'$, and}\\
  e(H[\AXA_i\cap \VXV_{j}])&\ge
  \aXa_i^2 e(H[\VXV_j])-k^{0.6}n^{0.6}\;,
\end{align*}  
for each of the graphs
  $H\in\{G,\Gcapt,\Gexp,\GD,\Gcapt\cup\GD\}$.
  \item\label{It:H7} For each $i\in[p]$ if $\aXa_i=0$ then
  $\AXA_i=\emptyset$.
\end{enumerate}
\end{lemma}

\begin{definition}[\bf Proportional
splitting]\label{def:proportionalsplitting}\index{general}{proportional splitting}
Let $\proporce{0},\proporce{1},\proporce{2}>0$ be three positive reals with
$\sum_i\proporce{i}\le 1$. Under Setting~\ref{commonsetting}, suppose that
$\colouringpartition$ is a partition of $V(G)\setminus\HugeVertices$ which satisfies assertions of Lemma~\ref{lem:randomSplit} with parameter
$p_\PARAMETERPASSING{L}{lem:randomSplit}:=10$ for graph
$G^*_\PARAMETERPASSING{L}{lem:randomSplit}:=(\Gcapt-\HugeVertices)\cup\GD$ (here,
by the union, we mean union of the edges), bounded decomposition $( \clusters,\DenseSpots,
\Gblack, \Gexp, \smallatoms )$, matching
$\M_\PARAMETERPASSING{L}{lem:randomSplit}:=\M_A\cup\M_B$, sets $\VXV_1:=\Vgood,
\VXV_2:=\XA\setminus(\HugeVertices\cup \gP)$, $\VXV_3:=\XB\setminus \gP$,
$\VXV_4:=V(\Gexp)$, $\VXV_5:=\smallatoms$, $\VXV_6:=\largeintoatoms$, $\VXV_7:=\gPatoms$, $\VXV_8:=\largevertices{\eta}{k}{G}$, $\VXV_9:=L_\sharp$, $\VXV_{10}:=\WantiC$ and reals
$\aXa_1:=\proporce{0},\aXa_2:=\proporce{1}$, $\aXa_3:=\proporce{2}$,
$\aXa_4:=\ldots\aXa_{10}=0$. Note that by
Lemma~\ref{lem:randomSplit}\eqref{It:H7} we have that $\colouringpartition$ is
a partition of $V(G)\setminus \HugeVertices$. We call $\colouringpartition$
\emph{proportional $(\proporce{0}:\proporce{1}:\proporce{2})$ splitting}.

We refer to properties of the proportional
$(\proporce{0}:\proporce{1}:\proporce{2})$  splitting $\colouringpartition$
using the numbering of Lemma~\ref{lem:randomSplit}; for example,
``Definition~\ref{def:proportionalsplitting}\eqref{It:HproportionalSizes}''
tells us among other things that $|(\XA\setminus\gP)\cap
\colouringp{0}|\ge\proporce{0}|\XA\setminus(\gP\cup\HugeVertices)|-n^{0.9}$.
\end{definition}
\begin{setting}\label{settingsplitting}
Under Setting~\ref{commonsetting}, suppose that we are given a
proportional 
\index{mathsymbols}{*Pa@$\proporce{i}$}
$(\proporce{0}:\proporce{1}:\proporce{2})$
splitting \index{mathsymbols}{*A@$\colouringp{i}$}
$\colouringpartition$ of $V(G)\setminus\HugeVertices$. We assume
that 
\begin{equation}\label{eq:proporcevelke}
\proporce{0},\proporce{1},\proporce{2}\ge\frac{\eta}{100}\;.
\end{equation}
Let\index{mathsymbols}{*V@$\exceptVertSplit$}\index{mathsymbols}{*V@$\exceptSemSplit$}\index{mathsymbols}{*V@$\exceptClustSplit$}
$\exceptVertSplit,\exceptSemSplit,\exceptClustSplit$ be the exceptional sets as in Definition~\ref{def:proportionalsplitting}\eqref{It:H1}.

We write \index{mathsymbols}{*F@$\shadowsplit$}
\begin{equation}\label{def:shadowsplit}
\shadowsplit:=\shadow_{\GD}\left(\bigcup 
\exceptSemSplit\cup\bigcup \exceptSemSplit^*\cup \bigcup \exceptClustSplit,\frac{\eta^2k}{10^{10}}\right)\;,
\end{equation} where \index{mathsymbols}{*V@$\exceptSemSplit^*$}$\exceptSemSplit^*$ are the partners of $\exceptSemSplit$ in $\M_A\cup\M_B$.

We have 
\begin{equation}\label{eq:boundShadowsplit}
|\shadowsplit|\le \epsilon n\;.
\end{equation}

For an arbitrary
set $U\subset V(G)$ and for $i\in\{0,1,2\}$ we write \index{mathsymbols}{*U@$U\colouringpI{i}$}$U\colouringpI{i}$ for
the set $U\cap\colouringp{i}$. 

For each $(X,Y)\in\M_A\cup\M_B$ such that $X,Y\notin\exceptSemSplit$
we write $(X,Y)\colouringpI{i}$ for an arbitrary fixed pair $(X'\subset
X,Y'\subset Y)$ with the property that
$|X'|=|Y'|=\min\{|X\colouringpI{i}|,|Y\colouringpI{i}|\}$. We extend this notion
of restriction to an arbitrary \semiregular matching $\mathcal N\subset
\M_A\cup\M_B$ as follows. We set\index{mathsymbols}{*N@$\mathcal N\colouringpI{i}$}
$$\mathcal
N\colouringpI{i}:=\big\{(X,Y)\colouringpI{i}\::\:(X,Y)\in\mathcal
N \text{ with } X,Y\notin\exceptSemSplit\big\}\;.$$
\end{setting}
In~\cite{cite:LKS-cut2} it was shown that the above setting yields the following property.
\begin{lemma}[{\cite[Lemma~\ref{p2.lem:RestrictionSemiregularMatching}]{cite:LKS-cut2}}]\label{lem:RestrictionSemiregularMatching}
Assume Setting~\ref{settingsplitting}. Then for each $i\in\{0,1,2\}$, and for
each $\mathcal N\subset\M_A\cup\M_B$ we have that $\mathcal N\colouringpI{i}$ is
a $(\frac{400\epsilon}{\eta},\frac
d2,\frac{\eta\pi}{200}\clustersize)$-\semiregular matching satisfying
\begin{equation}\label{eq:restrictedmatchlarge}
|V(\mathcal N\colouringpI{i})|\ge
\proporce{i}|V(\mathcal N)|-2k^{-0.05}n\;.
\end{equation}
Moreover for all $v\not\in \shadowsplit$ and for all $i=0,1,2$ we have 
$\deg_{\GD}(v, V(\mathcal N)\colouringpI{i}\setminus V(\mathcal
N\colouringpI{i}))\le \frac {\eta^2k}{10^5}$. 
\end{lemma}

\subsection{The ten configurations}\label{ssec:10conf}
Here, we recall the configurations introduced in~\cite[Section~\ref{p2.ssec:TypesConf}]{cite:LKS-cut2}. Recall also that saying that ``we have Configuration X'', ``the graph is in Configuration X'', or  ``Configuration X occurs'' is the same.

\bigskip We start  by giving the definition of Configuration~$\mathbf{(\diamond1)}$. This is a very easy configuration in which a modification of the greedy tree-embedding strategy works.
\begin{definition}[\bf Configuration~$\mathbf{(\diamond1)}$]\index{mathsymbols}{**1@$\mathbf{(\diamond1)}$}
We say that a graph $G$ is in \emph{Configuration~$\mathbf{(\diamond1)}$} if
there exists a non-empty bipartite graph $H\subset G$ with $\mindeg_G(V(H))\ge k$ and $\mindeg(H)\ge k/2$.
\end{definition}

\bigskip We now introduce the configurations $\mathbf{(\diamond2)}$--$\mathbf{(\diamond5)}$ which make use of the set $\HugeVertices$. These configurations build on Preconfiguration $\mathbf{(\clubsuit)}$. 

\begin{definition}[\bf Preconfiguration 
$\mathbf{(\clubsuit)}$]\index{mathsymbols}{***@$\mathbf{(\clubsuit)}$}
\label{def:PreClub}
Suppose that we are in
Setting~\ref{commonsetting}. We say that the graph $G$ is in
\emph{Preconfiguration~$\mathbf{(\clubsuit)}(\Omega^\star)$} if the following
conditions are satisfied.
$G$ contains non-empty sets $L''\subset L'\subset
\largevertices{\frac9{10}\eta}{k}{\Gcapt}\setminus\HugeVertices$, and a non-empty set $\HugeVertices'\subseteq\HugeVertices$ such that
\begin{align}\label{eq:clubsuitCOND1}
\maxdeg_{\Gcapt} (L',\HugeVertices\setminus \HugeVertices')&<\frac{\eta
k}{100} \;\mbox{,}\\ 
\label{eq:clubsuitCOND2}
\mindeg_{\Gcapt}(\HugeVertices',\BUG{L'})&\ge \Omega^\star k\;\mbox{, and}\\
\label{eq:clubsuitCOND3}
\maxdeg_{\Gcapt}(L'',\largevertices{\frac9{10}\eta}{k}{\Gcapt}\setminus(\HugeVertices\cup
L'))&\le\frac{\eta k}{100}\;.
\end{align}
\end{definition}

\begin{definition}[\bf Configuration
$\mathbf{(\diamond2)}$]\index{mathsymbols}{**2@$\mathbf{(\diamond2)}$}Suppose that we are in Setting~\ref{commonsetting}. We say that the graph $G$ is in
\emph{Configuration $\mathbf{(\diamond2)}(\Omega^\star,
\tilde\Omega,\beta)$} if the following
conditions are satisfied.

 The triple $L'',L',\HugeVertices' $ witnesses preconfiguration
$\mathbf{(\clubsuit)}(\Omega^\star)$ in $G$. There exist a
non-empty set $\HugeVertices''\subset \HugeVertices'$, a set $V_1\subset V(\Gexp)\cap\YB\cap L''$, and a set $V_2\subset V(\Gexp)$ with the following properties.
\begin{align*}
\mindeg_{\Gcapt}(\HugeVertices'',V_1)&\ge\tilde\Omega k\;\\
\mindeg_{\Gcapt}(V_1,\HugeVertices'')&\ge \beta k\;,\\
\mindeg_{\Gexp}(V_1,V_2)&\ge \beta k\;,\\
\mindeg_{\Gexp}(V_2,V_1)&\ge \beta k\;.
\end{align*}
\end{definition}

\begin{definition}[\bf Configuration
$\mathbf{(\diamond3)}$]\index{mathsymbols}{**3@$\mathbf{(\diamond3)}$}
\label{def:CONF3}
Suppose that we are in
Setting~\ref{commonsetting}. We say that the graph $G$ is in
\emph{Configuration $\mathbf{(\diamond3)}(\Omega^\star,
\tilde\Omega,\zeta,\delta)$} if the following
conditions are satisfied.

 The triple $L'',L',\HugeVertices'$ witnesses preconfiguration
$\mathbf{(\clubsuit)}(\Omega^\star)$ in $G$. There exist a
non-empty set $\HugeVertices''\subset \HugeVertices'$, a set $V_1\subset \smallatoms\cap \YB\cap L''$, and a set $V_2\subset V(G)\setminus \HugeVertices$ such that the
following properties are satisfied.
\begin{align}
\nonumber
\mindeg_{\Gcapt}(\HugeVertices'',V_1)&\ge \tilde \Omega k\;,\\
\nonumber
\mindeg_{\Gcapt}(V_1,\HugeVertices'')&\ge \delta k\;,\\
\label{eq:WHtc}
\maxdeg_{\GD}(V_1, V(G)\setminus
(V_2\cup \HugeVertices))&\le \zeta k\;,\\ 
\label{confi3theothercondi}
\mindeg_{\GD}(V_2,V_1)&\ge \delta k\;.
\end{align}
\end{definition}

\begin{definition}[\bf Configuration
$\mathbf{(\diamond4)}$]\index{mathsymbols}{**4@$\mathbf{(\diamond4)}$}
\label{def:CONF4}
Suppose that we are in
Setting~\ref{commonsetting}. We say that the graph~$G$ is in
\emph{Configuration
$\mathbf{(\diamond4)}(\Omega^\star, \tilde\Omega,\zeta,\delta)$} if the
following conditions are satisfied.

 The triple $L'',L',\HugeVertices'$ witnesses preconfiguration
$\mathbf{(\clubsuit)}(\Omega^\star)$ in $G$. There exist a
non-empty set $\HugeVertices''\subset \HugeVertices'$, sets $V_1\subset \YB\cap L''$,
$\smallatoms'\subset \smallatoms$, and $V_2\subset V(G)\setminus \HugeVertices$ with the following properties
\begin{align}
\mindeg_{\Gcapt}(\HugeVertices'',V_1)&\ge \tilde\Omega k\;,\notag \\
\mindeg_{\Gcapt}(V_1,\HugeVertices'')&\ge \delta k\;,\notag \\
\mindeg_{\Gcapt\cup\GD}(V_1,\smallatoms')&\ge \delta k\;,\label{confi4:3} \\
\mindeg_{\Gcapt\cup\GD}(\smallatoms',V_1)&\ge \delta k\;,\label{confi4:4} \\
\mindeg_{\Gcapt\cup\GD}(V_2,\smallatoms')&\ge \delta k\;,\label{confi4othercondi} \\
\maxdeg_{\Gcapt\cup\GD}(\smallatoms',V(G)\setminus
(\HugeVertices\cup V_2))&\le \zeta k\;. \label{confi4lastcondi}
\end{align}
\end{definition}

\begin{definition}[\bf Configuration
$\mathbf{(\diamond5)}$]\index{mathsymbols}{**5@$\mathbf{(\diamond5)}$}
\label{def:CONF5}
Suppose that we are in
Setting~\ref{commonsetting}. We say that the graph $G$ is in
\emph{Configuration
$\mathbf{(\diamond5)}(\Omega^\star,\tilde\Omega,\delta,\zeta,\tilde\pi)$} if the
following conditions are satisfied.

 The triple $L'',L',\HugeVertices'$ witnesses preconfiguration
 $\mathbf{(\clubsuit)}(\Omega^\star)$ in $G$. There exists a non-empty set $\HugeVertices''\subset \HugeVertices'$, and a set $V_1\subset (\YB\cap
L''\cap \bigcup\clusters)\setminus V(\Gexp)$ such that the following conditions are fulfilled.
\begin{align}
\mindeg_{\Gcapt}(\HugeVertices'',V_1)&\ge \tilde\Omega k\;, \\
\mindeg_{\Gcapt}(V_1,\HugeVertices'')&\ge \delta k\;,
\label{eq:diamond5P2}\\
\mindeg_{\Gblack}(V_1)&\ge \zeta k\;.\label{confi5last}
\end{align}
Further, we have 
\begin{equation}\label{eq:diamond5P4}
\mbox{$C\cap V_1=\emptyset$ or $|C\cap V_1|\ge \tilde\pi|C|$}
\end{equation} for every
$C\in\clusters$.
\end{definition}

In remains to introduce
configurations~$\mathbf{(\diamond6)}$--$\mathbf{(\diamond10)}$. In these
configurations the set $\HugeVertices$ is not utilized. All these
configurations make use of Setting~\ref{settingsplitting}, i.e., the set
$V(G)\setminus \HugeVertices$ is partitioned into three sets
$\colouringp{0},\colouringp{1}$ and $\colouringp{2}$. The purpose of $\colouringp{0},\colouringp{1}$ and $\colouringp{2}$
is to embed the \kknnaaggssNOSPACE, the internal shrubs, and the
end shrubs of $T_\PARAMETERPASSING{T}{thm:main}$, respectively. Thus the
parameters $\proporce{0}, \proporce{1}$ and $\proporce{2}$ are chosen proportionally to
the sizes of these respective parts of $T_\PARAMETERPASSING{T}{thm:main}$.

We first introduce four preconfigurations $\mathbf{(\heartsuit 1)}$,
$\mathbf{(\heartsuit 2)}$, $\mathbf{(exp)}$ and $\mathbf{(reg)}$ which are
building bricks for
configurations~$\mathbf{(\diamond6)}$--$\mathbf{(\diamond9)}$.
The preconfigurations $\mathbf{(\heartsuit 1)}$ and $\mathbf{(\heartsuit 2)}$ will
be used for embedding end shrubs of a fine partition of the tree
$T_\PARAMETERPASSING{T}{thm:main}$, and preconfigurations $\mathbf{(exp)}$ and
$\mathbf{(reg)}$ will be used for embedding its \kknnaaggssNOSPACE.

 An \emph{$\M$-cover}\index{general}{cover}\index{mathsymbols}{*COVER@$\M$-cover} of a
\semiregular matching $\M$ is a family $\mathcal F\subset \V(\M)$ with the property that at least one of
the elements $S_1$ and $S_2$ is a member of $\mathcal F$, for each $(S_1,S_2)\in
\M$.

\begin{definition}[\bf Preconfiguration
$\mathbf{(\heartsuit 1)}$]\index{mathsymbols}{**1@$\mathbf{(\heartsuit1)}$}
\label{def:heart1}
 Suppose that we are in
Setting~\ref{commonsetting} and Setting~\ref{settingsplitting}. We say that the
graph $G$ is in \emph{Preconfiguration
$\mathbf{(\heartsuit 1)}(\gamma',h)$} of $V(G)$ if there are two
non-empty sets $V_0,V_1\subset
\colouringp{0}\setminus \left(\shadowsplit\cup\shadow_{\GD}(\WantiC, \frac{\eta^2 k}{10^5})\right)$ with the following
properties.
\begin{align}
\label{COND:P1:3}
\mindeg_{\Gcapt}\left(V_0,\Vgood\colouringpI{2}\right)&\ge h/2 \;\mbox{, and}\\
\label{COND:P1:4}
\mindeg_{\Gcapt}\left(V_1,\Vgood\colouringpI{2}\right)&\ge h \;.
\end{align} Further, there is an
$(\M_A\cup\M_B)$-cover $\mathcal F$ such that
\begin{equation}\label{COND:P1:5}
\maxdeg_{\Gcapt}\left(V_1,\bigcup\mathcal F\right)\le \gamma' k\;.
\end{equation}
\end{definition}

\begin{definition}[\bf Preconfiguration
$\mathbf{(\heartsuit 2)}$]\index{mathsymbols}{**2@$\mathbf{(\heartsuit2)}$} Suppose that we
are in Setting~\ref{commonsetting} and Setting~\ref{settingsplitting}. We say that the
graph $G$ is in \emph{Preconfiguration
$\mathbf{(\heartsuit 2)}(h)$} of $V(G)$ if there are two
non-empty sets $V_0,V_1\subset \colouringp{0}\setminus \left(\shadowsplit\cup \shadow_{\GD}(\WantiC, \frac {\eta^2 k}{10^5})\right)$ with the
following properties.
\begin{align}
\begin{split}\label{COND:P2:4}
\mindeg_{\Gcapt}\left(V_0\cup V_1,\Vgood\colouringpI{2}\right)&\ge h.
\end{split}
\end{align}
\end{definition}

\begin{definition}[\bf Preconfiguration
$\mathbf{(exp)}$]\index{mathsymbols}{**exp@$\mathbf{(exp)}$} 
\label{def:exp8}
Suppose that we
are in Setting~\ref{commonsetting} and Setting~\ref{settingsplitting}. We say that the
graph $G$ is in \emph{Preconfiguration
$\mathbf{\mathbf{(exp)}}(\beta)$} if there are two
non-empty sets $V_0,V_1\subset \colouringp{0}$ with the following properties.
\begin{align}
\label{COND:exp:1}
\mindeg_{\Gexp}(V_0,V_1)&\ge \beta k\;,\\
\label{COND:exp:2}
\mindeg_{\Gexp}(V_1,V_0)&\ge \beta k\;.
\end{align}
\end{definition}

\begin{definition}[\bf Preconfiguration
$\mathbf{(reg)}$]\index{mathsymbols}{**reg@$\mathbf{(reg)}$}
\label{def:reg}
 Suppose that we
are in Setting~\ref{commonsetting} and Setting~\ref{settingsplitting}. We say that the graph $G$ is in \emph{Preconfiguration
$\mathbf{\mathbf{(reg)}}(\tilde \epsilon, d', \mu)$}  if there are two  non-empty sets $V_0,V_1\subset \colouringp{0}$
and a non-empty family of vertex-disjoint $(\tilde\epsilon,d')$-super-regular pairs $\{(Q_0^{(j)},Q_1^{(j)}\}_{j\in\mathcal Y}$ (with respect to the edge set $E(G)$) with $V_0:=\bigcup Q_0^{(j)}$ and $V_1:=\bigcup Q_1^{(j)}$ such that
\begin{align}
\label{COND:reg:0}
\min\left\{|Q_0^{(j)}|,|Q_1^{(j)}|\right\}&\ge\mu k\;.
\end{align}
\end{definition}

\begin{definition}[\bf Configuration
$\mathbf{(\diamond6)}$]\index{mathsymbols}{**6@$\mathbf{(\diamond6)}$}
\label{def:CONF6}
Suppose that we are in
Settings~\ref{commonsetting} and~\ref{settingsplitting}. We say that the graph $G$
is in \emph{Configuration
$\mathbf{(\diamond6)}(\delta, \tilde \epsilon,d',\mu, \gamma', h_2)$} if the
following conditions are met.

The vertex sets $V_0,V_1$ 
witness Preconfiguration
 $\mathbf{(reg)}(\tilde \epsilon,d',\mu)$ or
 Preconfiguration~$\mathbf{(exp)}(\delta)$ and either Preconfiguration~$\mathbf{(\heartsuit1)}(\gamma',h_2)$ or
Preconfiguration~$\mathbf{(\heartsuit2)}(h_2)$. There exist non-empty sets $V_2,V_3\subset \colouringp{1}$ such that
 \begin{align}\label{COND:D6:1}
\mindeg_{G}(V_1,V_2)&\ge \delta k\;,\\ 
\label{COND:D6:2}
\mindeg_{G}(V_2,V_1)&\ge \delta k\;,\\ 
\label{COND:D6:3}
\mindeg_{\Gexp}(V_2,V_3)&\ge \delta k \;,\mbox{and}\\
\label{COND:D6:4}
\mindeg_{\Gexp}(V_3,V_2)&\ge \delta k\;.
\end{align}
\end{definition}

\begin{definition}[\bf Configuration
$\mathbf{(\diamond7)}$]\index{mathsymbols}{**7@$\mathbf{(\diamond7)}$}
\label{def:CONF7}
Suppose that we are in
Settings~\ref{commonsetting} and~\ref{settingsplitting}. We say that the graph $G$
is in \emph{Configuration
$\mathbf{(\diamond7)}(\delta, \rho', \tilde \epsilon, d',\mu, \gamma', h_2)$} if
the following conditions are satisfied.

The sets $V_0,V_1$  witness Preconfiguration
 $\mathbf{(reg)}( \tilde \epsilon, d',\mu)$ and either Preconfiguration~$\mathbf{(\heartsuit
 1)}(\gamma', h_2)$ or Preconfiguration~$\mathbf{(\heartsuit 2)}(h_2)$. There
 exist non-empty sets $V_2\subset \smallatoms\colouringpI{1}\setminus \exceptVertSplit$ and $V_3\subset \colouringp{1}$ such that
 \begin{align}
\label{COND:D7:1}
\mindeg_{G}(V_1,V_2)&\ge \delta k\;,\\
\label{COND:D7:2}
\mindeg_{G}(V_2,V_1)&\ge \delta k\;,\\
\label{COND:D7:3}
\maxdeg_{\GD}(V_2,\colouringp{1}\setminus V_3)&< \rho' k \;\mbox{and}\\
\label{COND:D7:4}
\mindeg_{\GD}(V_3,V_2)&\ge \delta k\;.
\end{align}
\end{definition}

\begin{definition}[\bf Configuration
$\mathbf{(\diamond8)}$]\index{mathsymbols}{**8@$\mathbf{(\diamond8)}$}
\label{def:CONF8}
Suppose that we are in
Settings~\ref{commonsetting} and~\ref{settingsplitting}. We say that the graph $G$
is in \emph{Configuration
$\mathbf{(\diamond8)}(\delta,\rho',\epsilon_1,\epsilon_2, d_1,d_2,\mu_1,\mu_2, h_1,h_2)$}
 if the following conditions are met.

The vertex sets $V_0,V_1$  witness Preconfiguration
 $\mathbf{(reg)}(\epsilon_2, d_2,\mu_2)$ and Preconfiguration~$\mathbf{(\heartsuit 2)}(h_2)$.
 There exist non-empty sets $V_2\subset \colouringp{0}$, $V_3,V_4\subset \colouringp{1}$, with $V_3\subset\smallatoms\setminus \exceptVertSplit$, and an $(\epsilon_1, d_1, \mu_1 k)$-\semiregular matching $\mathcal N$ absorbed by $(\M_A\cup\M_B)\setminus \NAtom$, with $V(\mathcal N)\subset \colouringp{1}\setminus V_3$ such that
\begin{align}
\label{COND:D8:1}
\mindeg_{G}(V_1,V_2)&\ge \delta k\;,\\
\label{COND:D8:2}
\mindeg_{G}(V_2,V_1)&\ge \delta k\;,\\
\label{COND:D8:3}
\mindeg_{\Gcapt}(V_2,V_3)&\ge \delta k\;,\\
\label{COND:D8:4}
\mindeg_{\Gcapt}(V_3,V_2)&\ge \delta k\;,\\
\label{COND:D8:5}
\maxdeg_{\GD}(V_3,\colouringp{1}\setminus V_4)&< \rho' k \;,\\
\label{COND:D8:6}
\mindeg_{\GD}(V_4,V_3)&\ge \delta k\;\mbox{, and}\\
\label{COND:D8:7}
\deg_{\GD}(v,V_3)+\deg_{\Gblack}(v,V(\mathcal N))&\ge h_1\;\mbox{for each $v\in V_2$.}
\end{align} 
\end{definition}
\begin{definition}[\bf Configuration
$\mathbf{(\diamond9)}$]\index{mathsymbols}{**9@$\mathbf{(\diamond9)}$}
\label{def:CONF9}
Suppose that we are in
Settings~\ref{commonsetting}, and~\ref{settingsplitting}. We say that the graph $G$
is in \emph{Configuration
$\mathbf{(\diamond9)}(\delta, \gamma', h_1, h_2, \epsilon_1, d_1,
\mu_1,\epsilon_2, d_2,\mu_2)$} if the following conditions are satisfied.

The sets $V_0,V_1$ together with the $(\M_A\cup\M_B)$-cover $\mathcal F'$
witness Preconfiguration~$\mathbf{(\heartsuit1)}(\gamma',h_2)$. 
 There exists an $(\epsilon_1, d_1, \mu_1 k)$-\semiregular matching $\mathcal N$ absorbed by $\M_A\cup\M_B$, with $V(\mathcal N)\subset \colouringp{1}$.
Further, there is a family $\{(Q_0^{(j)},Q_1^{(j)})\}_{j\in\mathcal Y}$ as in Preconfiguration~$\mathbf{(reg)}(\epsilon_2,d_2,\mu_2)$. There is a set $V_2\subseteq
 V(\mathcal N)\setminus \bigcup \mathcal F'\subset \bigcup\clusters$ with the
 following properties:
\begin{align}
\label{conf:D9-XtoV}
\mindeg_{\GD}\left(V_1, V_2\right)\ge h_1\;,&\\
\label{conf:D9-VtoX}\mindeg_{\GD}\left(V_2,V_1\right)\ge
\delta k\;.
\end{align}
\end{definition}

Our last configuration, Configuration~$\mathbf{(\diamond10)}$, will lead to an embedding very similar to the one
in the dense case (as treated in~\cite{PS07+}; this will be explained in detail in Subsection~\ref{ssec:EmbedOverview10}). In order to be able to formalize the configuration we need a preliminary definition. We shall
generalize the standard concept of a regularity graph (in the context of regular
partitions and Szemer\'edi's regularity lemma) to graphs with clusters whose sizes are only bounded from below.

\begin{definition}[\bf{$( \epsilon,d,\ell_1,\ell_2)$-regularized
graph}]\index{general}{regularized graph}\label{def:regularizedGraph} Let $G$ be
a graph, and let $\mathcal V$ be an $\ell_1$-ensemble that partitions $V(G)$.
Suppose that $G[X]$ is empty for each $X\in \mathcal V$ and suppose $G[X,Y]$ is
$\epsilon$-regular and of density either $0$ or at least $d$ for each $X,Y\in
\mathcal V$. Further suppose that for all $X\in \V$ it holds that 
$|\bigcup\neighbour_G(X)|\le \ell_2$.
Then we say that $(G,\mathcal V)$ is an \emph{$(\eps, d,\ell_1,
\ell_2)$-regularized graph}.

A \semiregular matching $\M$ of $G$ is \index{general}{consistent matching}\emph{consistent} with $(G,\mathcal V)$ if $\V(\M)\subset \V$.
\end{definition}

\begin{definition}[\bf Configuration
$\mathbf{(\diamond10)}(\tilde\eps,d',\ell_1, \ell_2,
\eta')$]\index{mathsymbols}{**10@$\mathbf{(\diamond10)}$}
\label{def:CONF10}
Assume Setting~\ref{commonsetting}. The graph $G$ contains an $(
\tilde\epsilon, d', \ell_1, \ell_2)$-regularized graph $(\tilde G,\V)$  and there
is a $( \tilde\epsilon,  d',\ell_1 )$-\semiregular matching $\M$ consistent with
$(\tilde G,\V)$.
There are a family $\LargeTen\subset \V$ and distinct clusters $A, B\in\V$  with
\begin{enumerate}[(a)]
\item\label{diamond10cond1} $E(\tilde G[A,B])\neq \emptyset$, 
\item\label{diamond10cond2} $\deg_{\tilde G}\big(v,V(\M)\cup \bigcup \LargeTen\big)\ge (1+\eta')k$ for all but at most $\tilde\epsilon |A|$
vertices $v\in A$ and for all but at most $\tilde\epsilon|B|$ vertices $v\in B$, and
\item\label{diamond10cond3}
 for each $X\in\LargeTen$ we have $\deg_{\tilde G}(v)\ge (1+\eta')k$ for all but at most $\tilde\epsilon|X|$ vertices $v\in X$.
 \end{enumerate}
\end{definition}

\section{Embedding trees}\label{sec:embed}
In this section we provide an embedding of a tree $T_\PARAMETERPASSING{T}{thm:main}\in\treeclass{k}$ in
the setting of the configurations introduced in
Subsection~\ref{ssec:10conf}.  In Section~\ref{ssec:embeddingOverview} we
first give a fairly detailed overview of the embedding techniques used. In
Section~\ref{ssec:Duplicate} we introduce a class of stochastic processes which
will be used for some embeddings. Section~\ref{ssec:EmbeddingShrubs} contains a
number of lemmas about embedding small trees, and use them  for embedding \kknnaaggss and shrubs of a given fine partition of $T_\PARAMETERPASSING{T}{thm:main}$. Embedding the entire tree $T_\PARAMETERPASSING{T}{thm:main}$ is then handled in the final Section~\ref{sec:MainEmbedding}. There we have to distinguish between particular configurations. The configurations are grouped into three categories (Section~\ref{sssec:EmbedDiamon0Diamond1}, Section~\ref{sssec:EmbedMoreComplex}, and Section~\ref{sssec:OrderedSkeleton}) corresponding to the similarities between the configurations.

\subsection{Overview of the embedding procedures}\label{ssec:embeddingOverview}
We outlined the high-level embedding strategy in based on the previous work in the dense setting (c.f.~\cite{PS07+}) in \cite[{Section~\ref{p1.ssec:motivation}}]{cite:LKS-cut1}. In this section we however have already a finer structure given by one of the configurations.

Recall that we are working under Setting~\ref{commonsetting}. 
Given a host graph $G_\PARAMETERPASSING{T}{thm:main}$ with one of the Configurations $\mathbf{(\diamond2)}$--$\mathbf{(\diamond10)}$, we have to embed in it a given tree $T=T_\PARAMETERPASSING{T}{thm:main}\in\treeclass{k}$, which comes with its  
$(\tau k)$-fine partition $(W_A,W_B,\shrubA,\shrubB)$. The $\tau k$-fine partition of $T$ will make it possible to combine embeddings of smaller parts of $T$ into one embedding of the whole tree. This means that we will first develop tools for embedding singular shrubs and
\kknnaaggss of the $(\tau k)$-fine partition in various basic building bricks of the
configurations: the avoiding set $\smallatoms$, the expander $\Gexp$,  regular pairs, and  vertices of huge degree
$\HugeVertices$. Second, we will combine these basic techniques to embed the
entire tree $T$. Here, the order in which different parts of $T$ are embedded
is important. Also, it will be crucial at some points to reserve places for parts of the tree
which will be embedded only later. 

In the following subsections, we sketch our embedding techniques. We group them into five categories comprising  related configurations\footnote{Configuration $\mathbf{(\diamond1)}$ is trivial (see
Section~\ref{sssec:EmbedDiamon0Diamond1}) and needs no draft.}: Configurations
$\mathbf{(\diamond2)}$--$\mathbf{(\diamond5)}$, Configurations $\mathbf{(\diamond6)}$--$\mathbf{(\diamond7)}$,  Configuration~$\mathbf{(\diamond8)}$, Configuration~$\mathbf{(\diamond9)}$, and Configuration~$\mathbf{(\diamond10)}$, treated in
Sections~\ref{ssec:EmbedOverview25},~\ref{ssec:EmbedOverview67}, \ref{ssec:EmbedOverview8}, \ref{ssec:EmbedOverview9}, \ref{ssec:EmbedOverview10},
respectively. 

To illustrate our embedding techniques in more detail, and how they combine, we chose to explain the embedding procedure for Configuration~$\mathbf{(\diamond7)}$~$\mathbf{(exp)}$~$\mathbf{(\heartsuit 1)}$  even more in details. This is done in Section~\ref{sssec:7exp1}. Not all the techniques are used in $\mathbf{(\diamond7)}$~$\mathbf{(exp)}$~$\mathbf{(\heartsuit 1)}$; in particular that configuration does not deal with huge degree vertices (as we do in Section~\ref{ssec:EmbedOverview25}) and does not make use of $\Gblack$. Yet, at least in this configuration, it may be a useful intermediate step between the description in Section~\ref{ssec:EmbedOverview67} and the full proof in Lemma~\ref{lem:embed:total68}.

\subsubsection{Embedding overview for Configurations
$\mathbf{(\diamond2)}$--$\mathbf{(\diamond5)}$}\label{ssec:EmbedOverview25}
In each of the Configurations
$\mathbf{(\diamond2)}$--$\mathbf{(\diamond5)}$ we have sets $\HugeVertices'',\HugeVertices',L'', L'$ and $V_1$. Further, we have some additional sets ($V_2$ and/or $\smallatoms'$) depending on the particular configuration.

A common embedding scheme for Configurations $\mathbf{(\diamond2)}$--$\mathbf{(\diamond5)}$ is illustrated in Figure~\ref{fig:DIAMOND25overview}. 
\begin{figure}[ht]
\centering 
\includegraphics{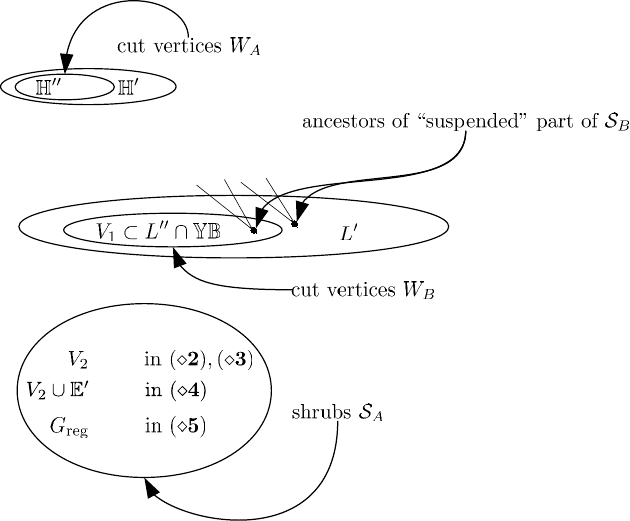}
\caption[Embedding overview for Configurations~$\mathbf{(\diamond2)}$--$\mathbf{(\diamond5)}$]{An overview of embedding 
of a tree $T\in\treeclass{k}$ given with its fine partition $(W_A,W_B,\shrubA,\shrubB)$ using
Configurations~$\mathbf{(\diamond2)}$--$\mathbf{(\diamond5)}$. The \kknnaaggss are
embedded between $\HugeVertices''$ and $V_1$, all the shrubs $\shrubA$ are embedded in sets specific to particular configurations so that the vertices neighbouring the seeds~$W_A$ are embedded in $V_1$. Parts of the shrubs $\shrubB$ are embedded directly (using various embedding techniques), while the rest is ``suspended'', i.e., the ancestors of the unembedded remainders are embedded on vertices which have large degrees in $\HugeVertices'$. The embedding of $\shrubB$ is then finalized in the last stage.}
\label{fig:DIAMOND25overview}
\end{figure}
There are two stages of the embedding procedure: the \kknnaaggssNOSPACE, the shrubs $\shrubA$ and some parts of the shrubs $\shrubB$ are embedded in Stage~1, and then in Stage~2 the remainders of $\shrubB$ are embedded. Recall that $\shrubA$ contains both internal and end shrubs while $\shrubB$ contains exclusively end shrubs (Definition~\ref{ellfine}~\eqref{Bend}). We note that here the shrubs $\shrubB$ are further subdivided and some parts of them are embedded in Stage~1 and some in Stage~2.

\begin{itemize}
\item In Stage~1, the \kknnaaggss of $T$ are embedded in $\HugeVertices''$ and $V_1$ so that $W_A$ is mapped to $\HugeVertices''$ and $W_B$ is mapped to $V_1$.
\item In Stage~1, the internal and end shrubs of $\shrubA$ are embedded using the sets $V_1,V_2$ and $\smallatoms'$ which are specific to the particular Configurations~$\mathbf{(\diamond2)}$--$\mathbf{(\diamond5)}$. The vertices of $\shrubA$ neighbouring the seeds~$W_A$ are always embedded in $V_1$. Parts of the shrubs $\shrubB$ are embedded while the ancestors of the unembedded remainders are embedded on vertices which have large degrees in $\HugeVertices'$. 
\item In Stage~2, the embedding of $\shrubB$ is finalized. The remainders of $\shrubB$ are embedded starting with embedding their roots in $\HugeVertices'$.
\end{itemize}
A hierarchy of the embedding lemmas used to resolve Configurations
$\mathbf{(\diamond2)}$--$\mathbf{(\diamond5)}$ is given in Table~\ref{tab:Conf25}.
\begin{table}
\centering
\begin{tabular}{ccccc}
\hline
\multicolumn{5}{|c|}{Main embedding lemma: Lemma~\ref{lem:conf2-5}}\\
\hline
\multicolumn{1}{c}{$\Uparrow$}& &\multicolumn{1}{c}{$\Uparrow$}& &\multicolumn{1}{c}{$\Uparrow$}\\
\cline{1-1}\cline{3-3}\cline{5-5}
\multicolumn{1}{|c|}{Shrubs $\shrubA$}& &\multicolumn{1}{|c|}{Shrubs $\shrubB$ (Stage 1): Lemma~\ref{lem:blueShrubSuspend}}& &\multicolumn{1}{|c|}{Shrubs $\shrubB$ (Stage 2): Lemma~\ref{lem:embedC'endshrub}}\\
\cline{3-3}\cline{5-5}
\multicolumn{1}{|l|}{$\mathbf{(\diamond2)}$: Lemma~\ref{lem:embed:greyFOREST}}& & & & \\
\multicolumn{1}{|l|}{$\mathbf{(\diamond3)}$: Lemma~\ref{lem:HE3}}& & & & \\
\multicolumn{1}{|l|}{$\mathbf{(\diamond4)}$: Lemma~\ref{lem:HE4}}& & & & \\
\multicolumn{1}{|l|}{$\mathbf{(\diamond5)}$: regularity}& & & & \\
\cline{1-1}
\end{tabular}
\caption[Embedding lemmas for Configurations $\mathbf{(\diamond2)}$--$\mathbf{(\diamond5)}$]{Embedding lemmas employed for Configurations $\mathbf{(\diamond2)}$--$\mathbf{(\diamond5)}$.}
\label{tab:Conf25}
\end{table}

\subsubsection{Embedding overview for Configurations
$\mathbf{(\diamond6)}$--$\mathbf{(\diamond7)}$}\label{ssec:EmbedOverview67}
Suppose Setting~\ref{commonsetting} and~\ref{settingsplitting} (see Remark~\ref{rem:h1h2} below for a comment on the constants $\proporce{0},\proporce{1},\proporce{2}$). Recall that we have in each of these configurations sets $V_0\cup V_1\subset \colouringp{0}$, sets $V_2\cup V_3\subset \colouringp{1}$ and $\Vgood\colouringpI{2}$.

A common embedding scheme for Configurations $\mathbf{(\diamond6)}$--$\mathbf{(\diamond7)}$ is illustrated in Figure~\ref{fig:DIAMOND67overview}. 
\begin{figure}[ht]
\centering 
\includegraphics{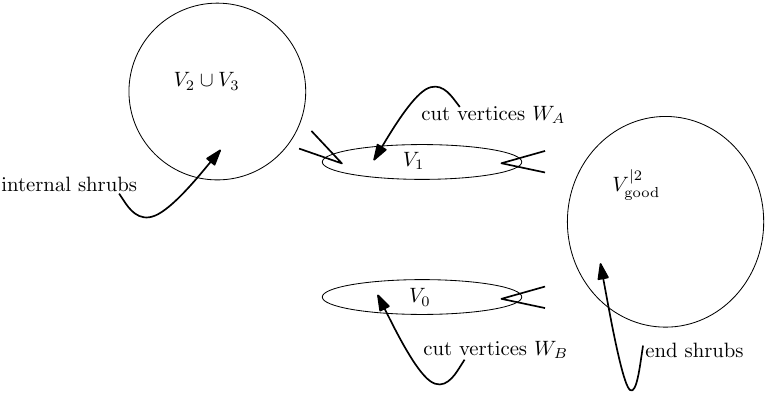}
\caption[Embedding overview for Configurations~$\mathbf{(\diamond6)}$--$\mathbf{(\diamond7)}$]{An overview of embedding a fine partition $(W_A,W_B,\shrubA,\shrubB)$
of a tree $T\in\treeclass{k}$ using
Configurations~$\mathbf{(\diamond6)}$--$\mathbf{(\diamond7)}$. The \kknnaaggss are
embedded between $V_0$ and $V_1$, the internal shrubs are embedded in
$V_2\cup V_3$, and the end shrubs are embedded
using $\Vgood\colouringpI{2}$.}
\label{fig:DIAMOND67overview}
\end{figure}
The embedding has three parts.
\begin{itemize}
\item The \kknnaaggss of $T$ are embedded between $V_0$ and $V_1$ so that $W_A$ is
mapped to $V_1$ and $W_B$ is mapped to $V_0$ using either the
Preconfiguration~$\mathbf{(exp)}$ or $\mathbf{(reg)}$. Thus the seeds~$W_A\cup W_B$ are mapped to $\colouringp{0}$.
\item The internal shrubs  of $T$ are embedded in $V_2\cup V_3$, always putting 
neighbours of $W_A$ into $V_2$. Note that the internal shrubs are therefore
embedded in $\colouringp{1}$, and thus there is no interference with embedding
the \kknnaaggssNOSPACE. We need to understand why a mere degree of $\delta k$ (from $V_1$ to $V_2$, ensured by~\eqref{COND:D6:1}
and~\eqref{COND:D7:1}, with $\delta\ll 1$) is sufficient for embedding internal
shrubs of potentially big total order, that is, how to ensure that already
embedded internal trees do not cause a blockage later. Here the
expansion\footnote{This expansion is given by the presence of $\Gexp$ in
Configurations~$\mathbf{(\diamond6)}$
(cf.~\eqref{COND:D6:3}--\eqref{COND:D6:4}), and by the presence of the avoiding
set $\smallatoms$ in Configurations~$\mathbf{(\diamond7)}$ ($V_2\subset
\smallatoms\colouringpI{1}\setminus\exceptVertSplit$).} ruling between the $V_2$
and $V_3$ comes into play. This property (together with other properties of
Preconfigurations~$\mathbf{(exp)}$ and $\mathbf{(reg)}$) will allow that, once
finished embedding an internal tree, the follow-up \kknnaagg can be embedded in a
place (in $V_1$) which sees very little of the previously embedded internal shrubs.
 
This is the only part of the embedding process which makes use of the specifics of Configurations~$\mathbf{(\diamond6)}$ and $\mathbf{(\diamond7)}$. For this reason we will be able to follow the same embedding scheme as presented here also for Configuration~$\mathbf{(\diamond8)}$, the only difference being the embedding of the internal shrubs (see Section~\ref{ssec:EmbedOverview8}).
\item The end shrubs are embedded in the yet unoccupied part of $G$.
For this we use the properties of Preconfigurations~$\mathbf{(\heartsuit1)}$ or $\mathbf{(\heartsuit2)}$.
The end shrubs are embedded using (but not entirely into) the designated vertex
set $\Vgood\colouringpI{2}$. 
\end{itemize}
The above embedding scheme is divided into two main steps: first the \kknnaaggss and the
internal trees are embedded (see Lemma~\ref{lem:embed:skeleton67}), and this
partial embedding is then extended to end shrubs (see Lemmas~\ref{lem:embed:heart1} and~\ref{lem:embed:heart2}). A more detailed hierarchy of the embedding lemmas used is given in Table~\ref{tab:Conf69}.
\begin{table}
\centering
\begin{tabular}{ccccc}
\hline
\multicolumn{5}{|c|}{Main embedding lemma: Lemma~\ref{lem:embed:total68}}\\
\hline
\multicolumn{3}{c}{$\Uparrow$}& &\multicolumn{1}{c}{$\Uparrow$}\\
\cline{1-3}\cline{5-5}
\multicolumn{3}{|c|}{Internal part}&\multicolumn{1}{|c|}{ }&\multicolumn{1}{|c|}{End shrubs}\\
\multicolumn{3}{|c|}{$\mathbf{(\diamond6)}$, $\mathbf{(\diamond7)}$: Lemma~\ref{lem:embed:skeleton67}}
& &\multicolumn{1}{|l|}{$\mathbf{(\heartsuit1)}$: Lemma~\ref{lem:embed:heart1}}\\
\multicolumn{3}{|c|}{$\mathbf{(\diamond8)}$: Lemma~\ref{lem:embed:skeleton8}}
& &\multicolumn{1}{|l|}{$\mathbf{(\heartsuit2)}$: Lemma~\ref{lem:embed:heart2}}\\
\cline{1-3}
\cline{5-5}
\multicolumn{1}{c}{$\Uparrow$}&\multicolumn{1}{c}{ }&\multicolumn{1}{c}{$\Uparrow$}& & \\
\cline{1-1}\cline{3-3}
\multicolumn{1}{|c|}{\KknnaaggssNOSPACE}&\multicolumn{1}{c}{ }&\multicolumn{1}{|c|}{Internal shrubs}& & \\
\multicolumn{1}{|l|}{$\mathbf{(exp)}$: Lemma~\ref{lem:embed:greyFOREST}}&\multicolumn{1}{c}{ }&\multicolumn{1}{|c|}{$\mathbf{(\diamond6)}$: Lemma~\ref{lem:embedStoch:DIAMOND6}}& &\\
\multicolumn{1}{|l|}{$\mathbf{(reg)}$: Lemma~\ref{lem:embed:superregular}}&\multicolumn{1}{c}{ }&\multicolumn{1}{|c|}{$\mathbf{(\diamond7)}$: Lemma~\ref{lem:embedStoch:DIAMOND7}}& &\\
\cline{1-1}
& & \multicolumn{1}{|c|}{$\mathbf{(\diamond8)}$: Lemmas~\ref{lem:embedStoch:DIAMOND7},~\ref{lem:embed:BALANCED},~\ref{lem:embed:regular}}& & \\
\cline{3-3}
\end{tabular}
\caption[Embedding lemmas for Configurations $\mathbf{(\diamond6)}$--$\mathbf{(\diamond8)}$]{Embedding lemmas employed for Configurations $\mathbf{(\diamond6)}$--$\mathbf{(\diamond8)}$ when embedding a tree $T\in\treeclass{k}$ with a given fine partition.}
\label{tab:Conf69}
\end{table}
\begin{remark}\label{rem:h1h2}
In Configuration~$\mathbf{(\diamond6)}$,
the number $\proporce{1}$ will be approximately the proportion of the total order of the internal shrubs of a given fine partition $(W_A,W_B,\shrubA,\shrubB)$ of $T$ while $\proporce{2}$ will be approximately the proportion of the total order of the end shrubs. The number $\proporce{0}$ is just a small constant. 

These numbers -- scaled up by $k$ -- determine the parameter $h_1\approx \proporce{1}k$ (in Configurations~$\mathbf{(\diamond8)}$ and $\mathbf{(\diamond9)}$) and $h_2\approx\proporce{2}k$ (in Configurations~$\mathbf{(\diamond6)}$--$\mathbf{(\diamond9)}$). The properties of these configurations will then allow to embed all the internal shrubs and end shrubs. Note that the parameter $h_1$ does not appear in Configurations~$\mathbf{(\diamond6)}$ and $\mathbf{(\diamond7)}$. This suggests that the total order of the internal shrubs is not at all important in 
Configurations~$\mathbf{(\diamond6)}$--$\mathbf{(\diamond7)}$. Indeed, we would succeed even embedding a tree with internal shrubs of total order say $100k$.\footnote{Configuration~$\mathbf{(\diamond8)}$ has this property only in part. We would succeed even embedding a tree with principal subshrubs of total order say $100k$ provided that the  total order of peripheral subshrubs is somewhat smaller than  $h_1$.}

In view of this it might be tempting to think that the end shrubs in $\shrubA$ could also be embedded using the same technique as the internal shrubs into the sets $V_2\cup V_3$ provided by these configurations (cf.\ Figure~\ref{fig:DIAMOND67overview}). This is however not the case. Indeed, the minimum degree conditions~\eqref{COND:D6:1},~\eqref{COND:D7:1}, and~\eqref{COND:D8:1} allow embedding only a small number of shrubs from a single cut-vertex $x\in W_A$ while there may be many end shrubs attached to $x$; cf.\ Remark~\ref{rem:internalVSend}\eqref{it:fewinternaltrees}.
\end{remark}

\subsubsection{Detailed overview of the embedding process for Configuration~$\mathbf{(\diamond7)}$~$\mathbf{(exp)}$~$\mathbf{(\heartsuit 1)}$}
\label{sssec:7exp1}
The purpose of this section is to further detail the embedding described in Section~\ref{ssec:EmbedOverview67} in the case of Configuration $\mathbf{(\diamond7)}$~$\mathbf{(exp)}$~$\mathbf{(\heartsuit 1)}$. We decided to choose this particular subconfiguration since the corresponding embedding exhibits many new features that come with the sparse decomposition.

We assume the same setting as in Section~\ref{ssec:EmbedOverview67} (in particular, recall Remark~\ref{rem:h1h2}).

The embedding process will first deal with \kknnaaggss and internal shrubs of $T$. Only after having embedded all those we turn our attention to end shrubs. We remind the reader that the sets $V_0\cup V_1$, $V_2\cup V_3$ and $\Vgood\colouringpI{2}$ are disjoint and thus the embedding into these respective parts do not interfere with each other.

\begin{table}
	\centering
\begin{tabular}{l|l}
defining formula for sets $F_i$ & formula for sets $F_i$ using purely the sets $U_j$\\
\hline		
$F_1\subset \shadow\left(U_2,\Theta(k)\right)$ & 
$F_1\subset \shadow\left(U_2,\Theta(k)\right)$ 
\\
$F_2\subset \shadow\left(U_1\cup F_1,\Theta(k)\right)$ &
$F_2\subset \shadow\left(U_1,\Theta(k)\right)\cup \shadow^{(2)}\left(U_2,\Theta(k)\right)$ 
\\
$F_3\subset \shadow\left(U_2\cup F_2,\Theta(k)\right)$ & $F_3\subset \shadow\left(U_2,\Theta(k)\right)\cup
\shadow^{(2)}\left(U_1,\Theta(k)\right)\cup \shadow^{(3)}\left(U_2,\Theta(k)\right)$ 
\\
\end{tabular}
	\caption{Hierarchy of shadows defining the sets $F_i$ as used in Section~\ref{sssec:7exp1}. Some of these shadows are with respect to the graph $\GD$ and some with respect to the graph $\Gcapt-\HugeVertices$.}
	\label{tab:shadows}
\end{table}
For the purpose of this overview, the sets $U_i$, $i=0,1,2,3$, will refer to the set of vertices in $V_i$ already used by the embedding at the very moment of the embedding procedure we are presently dealing with. Apart from the sets~$U_i$ of \emph{used} vertices, we define also sets of \emph{forbidden} vertices $F_i\subseteq V_i$, for $i=1,2,3$, which contain vertices whose use could possibly lead to a situation where we would be stuck with no possibility to extend the given partial embedding. More precisely, the set $F_i$ will consist of those vertices of $V_i$ that send $\Theta(k)$ (where the hidden constant in $\Theta(k)$ is much smaller than~1) edges to one of the sets $U_j$, and/or to one of the sets~$F_j$. So, $F_i$ can be expressed using shadows. More precisely, we set $F_1=V_1\cap \shadow_{\Gcapt-\HugeVertices}\left(U_2,\Theta(k)\right)$, $F_2= V_2\cap \shadow_{\Gcapt-\HugeVertices}\left(U_1\cup F_1,\Theta(k)\right)$, and $F_3=V_3 \cap \shadow_{\GD}\left(U_2\cup F_2,\Theta(k)\right)$. These definitions are shown in Table~\ref{tab:shadows}. It can be seen from Table~\ref{tab:shadows} that each set $F_i$ can be expressed purely in terms of the sets $U_j$ using shadows of exponent at most 3. Note that $\sum_i|U_i|\le k$.
As we do not use the set~$\HugeVertices$ of large degree vertices, the sizes of the sets $F_i$ will be at most linear in~$k$. Indeed, 
\begin{equation}\label{eq:Flinear}
|F_i|\le \left|\bigcup_{s=1}^3\shadow^{(s)}_{\GD\cup(\Gcapt-\HugeVertices)}\left(\bigcup_j U_j,\Theta(k)\right)\right|\eqBy{F\ref{fact:shadowbound}}O(k)\;.
\end{equation}
This is crucial in order to use the properties of the expanding graph~$\Gexp$ and the avoiding set~$\smallatoms$.

\begin{figure}[t]
	\centering 
	\includegraphics{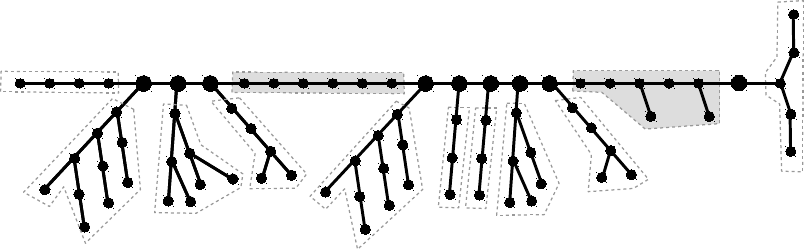}
	\caption{An example of a path-like tree. Cut vertices of its fine partition drawn bigger. Shrubs are drawn by a dashed line. Internal shrubs are drawn on gray background.}
	\label{fig:pathliketrees}
\end{figure}
In order not to clutter this overview with too many technical details, we chose to explain the embedding procedure on a rather simple type of trees: ``path-like'' trees. With the term {\it path-like trees} we mean trees having the property that the deletion of the external shrubs and the contraction of the internal shrubs, with respect to their fine-partition, leads to a path.
 See Figure~\ref{fig:pathliketrees}. Our motivation for working with path-like trees in this overview is that if the tree $T$ is more complex we face the complication of parallel branching of the embedding procedure. (This complication is handled by using the stochastic process  $\Duplicate$ as outlined in~\cite[Section~\ref{p0.sssec:whyGexp}]{cite:LKS-cut0}.) Note that, however, the family of path-like trees is general enough that it contains trees with any given ratio of internal shrubs and end shrubs.

 At every step of the embedding procedure, we will avoid the sets $U_i$ and $F_i$, making an exception for the roots  of internal shrubs, which may be mapped even to $F_2\setminus U_2$. It will be clear from the following, why we need   this exception and why we can afford it. Note that at the beginning of our embedding procedure, the sets $U_i, F_i$ are all empty, and thus trivial to avoid.
 
As outlined  in Table~\ref{tab:Conf69},we shall make use of the setting of $\mathbf{(exp)}$ to embed \kknnaaggssNOSPACE, of $\mathbf{(\heartsuit 1)}$ to embed the end shrubs, and the specifics of Configuration~$\mathbf{(\diamond7)}$ will be used for embedding the internal shrubs.

\paragraph{Embedding the first \kknnaaggNOSPACE}We start by embedding the \kknnaagg containing a fixed root $R$ of $T$ mapping~$W_A$ to~$V_1$ and mapping~$W_B$ to~$V_0$. The \kknnaaggss are only of size $O(1)$ by Definition~\ref{ellfine}\eqref{few}. So, the mere minimum degree conditions \eqref{COND:exp:1} and~\eqref{COND:exp:2} are sufficient for  embedding the \kknnaagg while avoiding the sets $U_0$ and $U_1$. In addition, we wish to avoid  the set $F_1$. 
While embedding the first \kknnaagg, we have not embedded any internal shrub yet. Therefore, initially, the set~$U_2$ is empty, and so is the set $F_1$.

\begin{figure}[ht]
	\centering 
	\includegraphics{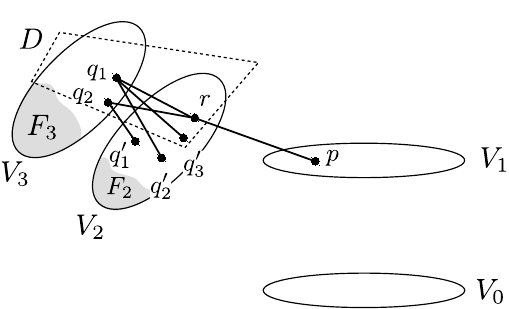}
	\caption{Embedding an internal shrub in Section~\ref{sssec:7exp1}. A suitable dense spot $D$ shown dashed.}
	\label{fig:detail7int}
\end{figure}
The rest of the embedding combines three techniques: embedding internal shrubs, embedding \kknnaaggssNOSPACE, and embedding end shrubs.

\paragraph{Embedding an internal shrub}Assume that we are at a given time of the embedding process when we have just finished embedding some \kknnaagg and are about to embed the next internal shrub~$T^*$. A picture corresponding to the description below is given in Figure~\ref{fig:detail7int}. As the predecessor~$p$ of the root~$r$ of the shrub is mapped to $V_1$, \eqref{COND:D7:1} tells us that the image of~$p$ has a substantial degree into $V_2$. Since $p$ was mapped outside of~$F_1$, the image of~$p$ has a substantial degree to $V_2\setminus U_2$. The set $V_2\setminus U_2$ has the avoiding property (see Definition~\ref{def:avoiding}), and therefore only very few candidates should not be used for the accomodating $r$, as they are \emph{exceptional}  with respect to the set~$U_3\cup F_3$ (which is of size $O(k)$ by~\eqref{eq:Flinear}). Therefore, we can map $r$ to some non-exceptional vertex in $V_2\setminus U_2$. In order to embed the children $q_1,\ldots,q_\ell$ of $r$ we shall use the property of the avoiding set, i.e., we use that there is a dense spot~$D$ containing the image of~$r$ such that 
\begin{equation}
\label{eq:makeuseavoid}
|D\cap (U_3\cup F_3)|\le\gamma^2k\;.
\end{equation}
As the image of~$r$ has substantial minimal degree in~$D\cap \colouringp{1}$ by Setting~\ref{settingsplitting}\eqref{It:H4} and only a very small portion of it goes outside of~$V_3$ (by~\eqref{COND:D7:3}) or to $U_3\cup F_3$ (by~\eqref{eq:makeuseavoid}), we can map $q_1,\ldots,q_\ell$ to $V_3\setminus (U_3\cup F_3)$ (recall that $\ell\le \tau k$, and~$\tau$ is the smallest constant in our hierarchy). 

The minimum degree condition~\eqref{COND:D7:4} together with the fact that the children $q_1,\ldots,q_\ell$ were embedded outside of $F_3$ this will ensure that we can map the grandchildren $q'_1,\ldots,q'_{\ell'}$ of $r$ to $V_2$ while avoiding the set $U_2\cup F_2$.

As we have seen above, it is enough to avoid~$U_2$ and the set of exceptional vertices in~$V_2$ (in the sense of the avoiding set) to be able to further extend the embedding of the internal shrub, by finding (possibly different) dense spots $D_1', \ldots, D_{\ell'}'$ containing $q_1', \ldots, q_{\ell}'$, respectively, such that $|D_i'\cap (U_3\cup F-3)|\le \gamma^2k$. We repeat this process until the embedding of~$T^*$ is finished.

\medskip

 The idea behind defining the set~$F_2$ is to prevent getting stuck when we need to map the next seed from $W_A$ to the set~$V_1\setminus (U_1\cup F_1)$, as here again we have no structural information on~$V_2$ or between~$V_2$ and~$V_1$ we can exploit (the avoiding property is useful only to go from~$V_2$ to~$V_3$, as it can be combined with the negligible loss of degree outside~$V_3$).

Before we turn our attention to further parts of the embedding process let us meditate on the reason for allowing the embedding of the root $r$ of $T^*$ in~$F_2$ and why we can afford such an exception. 
If we had to avoid the set~$F_2$ for the embedding of the roots of the internal shrubs, we would need to include a shadow of $F_2$ in~$F_1$. On the other hand, the set~$F_2$ includes a shadow of~$F_1$, so this would create a loop in the definitions. We can afford this exception for the following reason.
For any vertex mapped to $V_2\setminus F_2$, we can ensure that if it has a child belonging to~$W_A$, then this child can be  mapped to $V_1\setminus (U_1\cup F_1)$. This however is not guaranteed for the roots of the internal shrubs. Therefore, it is important that no root of internal shrub has a child belonging to~$W_A$. This is the reason behind Property~\eqref{short} of Definition~\ref{ellfine}.\footnote{Actually, a slightly weaker condition would be sufficient here.  Configuration~$\mathbf{(\diamond8)}$, however, is more complex, justifying the necessity of the stronger condition given in Property~\eqref{short} of Definition~\ref{ellfine}.}

\paragraph{Embedding further \kknnaaggssNOSPACE} Recall that the first \kknnaagg has already been embedded. We shall now explain how we make use of the expanding property of~$\Gexp$ from Preconfiguration~$\mathbf{(exp)}$ to embed any further \kknnaagg $X$. First, note that the first vertex of $X$ we are about to embed can be mapped to a suitable $w\in V_1\setminus (U_1\cup F_1)$, as its predecessor $q$ (which was a part of a previous internal shrub) does not belong to~$F_2$.\footnote{The only \kknnaagg without a predecessor contains the root $R$, and we have explained how to embed it.} Hence, let us assume that any vertex $x\in W_A$ is mapped to some vertex~$w\in V_1\setminus (U_1\cup F_1)$.  We want  to pick a prospective candidate among the neighbours of~$w$ to which we shall  map any given child of~$x$. The only properties required from this candidate is to be unused and to have substantial degree to $V_1\setminus (U_1\cup F_1)$. Only a tiny fraction of the neighbours of~$w$ lie in $U_0$, as the size of $W_A\cup W_B$ is $O(1)$ by Definition~\ref{ellfine}\eqref{few}. By~\eqref{COND:exp:1}, any  vertex in $\neighbour(w)\cap (V_0\setminus U_0)$ is a suitable prospective candidate, except those that send many edges to $U_1\cup F_1$  (in $\Gexp$). However, there are only very few such vertices by Fact~\ref{fact:shadowboundEXPANDER}. Thus, \eqref{COND:exp:1} tells us that we can accomodate $x$.

One could argue that while embedding the \kknnaaggss and internal shrubs, the sets~$U_i$, and thus~$F_i$, do increase dynamically. However, this is not a real problem and can easily be dealt with. Indeed, in every step of our embedding process, we have a substantial number of candidates we can choose from (of the order of magnitude $\delta k$). The size of one \kknnaaggNOSPACE, respectively of one internal tree, is of a much smaller order. Therefore, it is enough to update the sets~$U_i$ and~$F_i$ only at certain times.

\paragraph{Embedding $\shrubB$-shrubs}Once we have embedded all \kknnaaggss and all internal shrubs, we start embedding shrubs that are adjacent to~$W_B$. By Definition~\ref{ellfine}\eqref{Bend} these are end shrubs. As explained in~\cite[{Section~\ref{p1.sssec:roughversusfine}}]{cite:LKS-cut1}, the embedding of the end shrubs is much easier since we do not have to return to $V_0$ and $V_1$ for embedding cut vertices.

Let us note that at this stage of the embedding process no vertex of $\colouringp{2}$, and thus of~$\Vgood\colouringpI{2}$, has been used. The total order of the end shrubs is about~$h_2\approx\proporce{2}k$. Definition~\ref{ellfine}\eqref{Bsmall} tells us that the total order of $\shrubB$ is at most $h_2/2$. Display~\eqref{COND:P1:3} tells us that the degree of the vertices in~$V_0$ to~$\Vgood\colouringpI{2}$ is at least~$h_2/2$. As we suspend the embedding of the end shrubs adjacent to vertices in~$W_A$ until the last stage, there are always enough unused neighbours of vertices from~$V_0$ lying in~$\Vgood\colouringpI{2}$. To extend the embedding from a root to the entire end shrub it corresponds to, we use our basic techniques that build on the avoiding property, on the properties of the nowhere-dense graph,\footnote{The two are explained  in~\cite[Section~\ref{p0.sssec:whyavoiding}]{cite:LKS-cut0}, and in~\cite[Section~\ref{p0.sssec:whyGexp}]{cite:LKS-cut0}.}
or on exploiting regular pairs. The definitions~\eqref{eq:defVgood} and~\eqref{eq:defV+} indeed provide us with a setting in which it is possible to extend the embedding from~$\Vgood$ as explained in~\cite[{Section~\ref{p1.ssec:motivation}}]{cite:LKS-cut1}. The order in which we embed the $\shrubB$-shrubs is important in order to fill the end-clusters of regular pairs of $(\mathcal M_A\cup \mathcal M_B)\colouringpI{2}$ at the same pace as long as possible.

%

\paragraph{Embedding $\shrubA$-end shrubs}
It remains to embed the end shrubs from $\shrubA$. We shall use the same techniques we used for $\shrubB$-shrubs.

By~\eqref{COND:P1:4}, the minimum degree from~$V_1$ to~$\Vgood\colouringpI{2}$ is at least~$h_2$ and the total order of all end shrubs (including those from~$\shrubB$) is slightly less than~$h_2$. Therefore, there are always sufficient unused neighbours of vertices from~$V_1$ in~$\Vgood\colouringpI{2}$. Finally, \eqref{COND:P1:5} means that we do not need to care whether we fill of the end-clusters of regular pairs of $(\mathcal M_A\cup \mathcal M_B)\colouringpI{2}$ in a balanced way.

\subsubsection{Embedding overview for Configuration $\mathbf{(\diamond8)}$}\label{ssec:EmbedOverview8}
Suppose we are in Setting~\ref{commonsetting} and~\ref{settingsplitting}. We are working with sets $V_0$, $V_1$, $\Vgood\colouringpI{2}$, $V_2, V_3$ and $V_4$ and with a \semiregular matching $\mathcal N$ coming from the configuration.

The embedding scheme follows Table~\ref{tab:Conf69}, and is illustrated in Figure~\ref{fig:DIAMOND8}. 
\begin{figure}[ht]
\centering 
\includegraphics{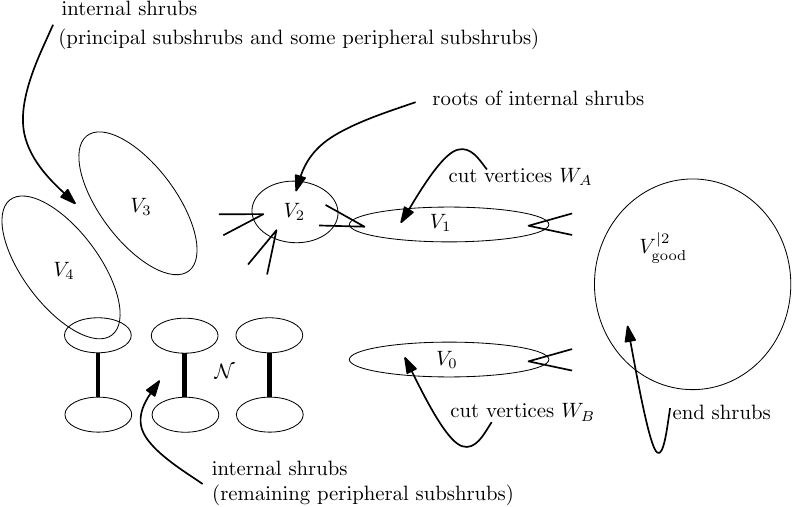}
\caption[Embedding overview for Configuration~$\mathbf{(\diamond8)}$]{An overview of embedding a fine partition $(W_A,W_B,\shrubA,\shrubB)$
of a tree $T\in\treeclass{k}$ using
Configuration~$\mathbf{(\diamond8)}$. The \kknnaaggss are
embedded between $V_0$ and $V_1$. The roots of the internal shrubs are embedded in $V_2$. Some of the subshrubs of the internal shrubs are embedded in $V_3\cup V_4$ and some in $\mathcal N$; principal subshrubs are always embedded in $V_3\cup V_4$. The end shrubs are embedded using the properties of  $\Vgood\colouringpI{2}$.}
\label{fig:DIAMOND8}
\end{figure}
The embedding of the \kknnaaggss and of the external shrubs is done in the same way as in Configurations~$\mathbf{(\diamond6)}$--$\mathbf{(\diamond7)}$. We only describe here the way the internal shrubs are embedded. Their roots are embedded in~$V_2$. From that point we proceed embedding subshrub by subshrub. Some of the subshrubs get embedded between $V_3$ and $V_4$. This pair of sets has the same expansion property as the pair $V_2,V_3$ in Configuration~$\mathbf{(\diamond7)}$. In particular, it allows to avoid the shadow of the already occupied set so that the follow-up \kknnaagg can be embedded in a location almost isolated from the previous images, similarly as described in Section~\ref{ssec:EmbedOverview67}. For this reason we make sure that principal subshrubs get embedded here. The degree condition from $V_2$ to $V_3$ is too weak to ensure that all remaining subshrubs are embedded between $V_3$ and $V_4$. Therefore we might have to embed some subshrubs in $\mathcal N$.  Condition~\eqref{COND:D8:7} --- where $h_1$ is 
approximately the order of the internal shrubs, as in Remark~\ref{rem:h1h2} --- indicates that it should be possible to accommodate all 
the subshrubs.
For technical reasons, the order in which different types of subshrubs are embedded is very important.

\subsubsection{Embedding overview for Configuration $\mathbf{(\diamond9)}$}\label{ssec:EmbedOverview9}
The embedding process in Configuration~$\mathbf{(\diamond9)}$
follows the same scheme as in Configurations
$\mathbf{(\diamond6)}$--$\mathbf{(\diamond8)}$, but the embedding of the
internal shrubs follows the regularity method. Assuming the simplest situation $\mathcal F=\V_2(\mathcal N)$ and $V_2=V_1(\mathcal N)$, we
would have $\mindeg_{\Gblack}(V_1,V_1(\mathcal N))\ge
h_1$ (cf.~\eqref{conf:D9-XtoV}). See Figure~\ref{fig:DIAMOND9} for an illustration.
\begin{figure}[ht]
\centering 
\includegraphics{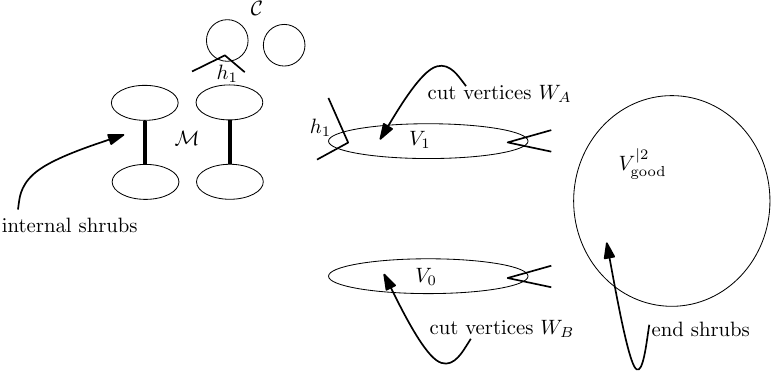}
\caption[Embedding overview for Configuration~$\mathbf{(\diamond9)}$]{An overview of embedding a fine partition $(W_A,W_B,\shrubA,\shrubB)$
of a tree $T\in\treeclass{k}$ using
Configuration~$\mathbf{(\diamond9)}$. The \kknnaaggss are
embedded between $V_0$ and $V_1$, the internal shrubs using the regularity
method in $\mathcal N$ and the end shrubs are embedded
using $\Vgood\colouringpI{2}$.}
\label{fig:DIAMOND9}
\end{figure}
Similarly as above, the \kknnaaggss are embedded between $V_0$ and $V_1$. The internal
shrubs are accommodated using the regularity method in $\mathcal N$, and the end
shrubs are embedded in $\Vgood\colouringpI{2}$ using Preconfiguration~$\mathbf{(\heartsuit1)}$. The embedding lemma for this configuration is given in Lemma~\ref{lem:embed9}.

\subsubsection{Embedding overview for Configuration $\mathbf{(\diamond10)}$}\label{ssec:EmbedOverview10}
Configuration~$\mathbf{(\diamond10)}$ is very closely related to the structure obtained by Piguet and Stein~\cite{PS07+} in their solution of the dense approximate case of Conjecture~\ref{conj:LKS}.\footnote{In~\cite[Section~\ref{p1.ssec:motivation}]{cite:LKS-cut1} we described in quite some detail how our main ``rough structural result'',~\cite[Lemma~\ref{p1.prop:LKSstruct}]{cite:LKS-cut1} relates to and differs from the Piguet--Stein structure. The description in this section, however, goes in a different direction since Configuration~$\mathbf{(\diamond10)}$ is much narrower that the general structure asserted in \cite[Lemma~\ref{p1.prop:LKSstruct}]{cite:LKS-cut1}.}

\begin{theorem}[Piguet--Stein~\cite{PS07+}]\label{thm:PiguetStein}
For any $q>0$ and $\alpha>0$ there exists a number $n_0$ such that for any $n>n_0$ and
$k>qn$ the following holds. 
Each $n$-vertex graph $G$ with at least
$n/2$ vertices of degree at least $(1+\alpha)k$ contains each tree of order $k+1$.
\end{theorem}

Let us describe their proof first. Piguet and Stein prove that when $k>qn$ (for some fixed $q>0$ and $k$ sufficiently large) the cluster graph\footnote{ordinary, in the sense of the classic regularity lemma} $\BGblack$ of a graph $G\in \LKSgraphs{n}{k}{\eta}$ contains the following structure (cf.~\cite[Lemma~8]{PS07+}). There is a set of clusters $\BL\subset \clusters$ such that each cluster in $\BL$ contains only vertices of captured degrees at least $(1+\frac\eta2)k$. There is a matching $M\subset \BGblack$, and an edge $AB$, with $A,B\in\BL$. One of the following conditions is satisfied
\begin{enumerate}
 \item[\textbf{(H1)}] $M$ covers $\neighbour_{\BGblack}(\{A,B\})$, or
 \item[\textbf{(H2)}] $M$ covers $\neighbour_{\BGblack}(A)$, and the vertices in $B$ have captured degrees at least $(1+\frac\eta2)\frac k2$ into $\bigcup(\BL\cup V(M))$. Further, each edge in $M$ has at most one endvertex in $\neighbour_{\BGblack}(A)$.
\end{enumerate}

Piguet and Stein use structures~\textbf{(H1)} and~\textbf{(H2)} to embed any
given tree $T\in\treeclass{k}$ into $G$ using the regularity method; see
Sections~3.6 and~3.7 in~\cite{PS07+}, respectively. Actually, a slight relaxation of~\textbf{(H1)} and~\textbf{(H2)} would be sufficient for the embedding to work, as can be easily seen from their proof: Again, there is a set of clusters $\BL\subset \clusters$ such that each cluster in $\BL$ contains only vertices of captured degrees at least $(1+\frac\eta2)k$, there is a matching $M\subset \BGblack$, and an edge $AB$, $A,B\in\BL$. One of the following conditions is satisfied
\begin{enumerate}
 \item[\textbf{(H1')}] the vertices in $A\cup B$ have captured degrees at least $(1+\frac\eta2)k$ into the vertices of $\bigcup (\BL\cup V(M))$, or
 \item[\textbf{(H2')}] the vertices in $A$ have captured degrees at least $(1+\frac\eta2)k$ into the vertices of $\bigcup V(M)$, and the vertices in $B$ have captured degrees at least $(1+\frac\eta2)\frac k2$ into $\bigcup(\BL\cup V(M))$. Further, each edge in $M$ has at most one endvertex in $\neighbour_{\BGblack}(A)$.
\end{enumerate}
It can be seen that Configuration~$\mathbf{(\diamond10)}$ is a direct
counterpart to~\textbf{(H1')}.\footnote{Observe that some parts of $\BGblack$
are irrelevant in the embedding process of~\cite{PS07+}. The objects $\BGblack$,
$\BL$, and $M$ in the structural result  of~\cite{PS07+} correspond to $(\tilde G,\V)$, $\mathcal L^*$, and $\M$ in Configuration~$\mathbf{(\diamond10)}$.} (The counterpart of~\textbf{(H2')} is contained in Configuration~$\mathbf{(\diamond9)}$ and the similarity is somewhat weaker.)

\medskip
The embedding lemma for Configuration~$\mathbf{(\diamond10)}$ is stated in Lemma~\ref{lem:embed10}.

\subsection{The role of random splitting}\label{ssec:whyrandomsplitting}
The random splitting as introduced in Setting~\ref{settingsplitting} is used in Configurations $\mathbf{(\diamond6)}$--$\mathbf{(\diamond9)}$; the set $\colouringp{0}$ will host the cut-vertices $W_A\cup W_B$, the set $\colouringp{1}$ will host the internal shrubs, and the set $\colouringp{2}$ will (essentially) host the end shrubs of a $(\tau k)$-fine partition of $T_\PARAMETERPASSING{T}{thm:main}$.

The need for introducing the random splitting is dictated by Configurations $\mathbf{(\diamond6)}$--$\mathbf{(\diamond9)}$. To see this, let us try to follow the embedding plan from, for example, Section~\ref{ssec:EmbedOverview67} without the random splitting, i.e., dropping the conditions $\subset \colouringp{0}$, $\subset \colouringp{1}$, $\subset \colouringp{2}$ from Definitions~\ref{def:heart1}--\ref{def:CONF7}. Then the sets $V_2$ and $V_3$ in Figure~\ref{fig:DIAMOND67overview}, which will host the internal shrubs, may interfere with $V_0$ and $V_1$ primarily designated for $W_A$ and $W_B$. In particular, the conditions on degrees between $V_0$ and $V_1$ given by \eqref{COND:exp:1}--\eqref{COND:exp:2} in Definition~\ref{def:exp8}, or given by the super-regularity in Definition~\ref{def:reg} (in which $\beta_{\PARAMETERPASSING{D}{def:exp8}}>0$, or $d'_{\PARAMETERPASSING{D}{def:reg}}\mu_{\PARAMETERPASSING{D}{def:reg}}>0$ are tiny) may be insufficient for embedding greedily all 
the cut-vertices and all the internal shrubs of $T_\PARAMETERPASSING{T}{thm:main}$. It should be noted that this problem occurs even in 
Preconfiguration~$\mathbf{(exp)}$, i.e., the expanding property does not add
enough strength to the minimum degree conditions. \footnote{See~\cite[Section~\ref{p0.sssec:whyGexp}]{cite:LKS-cut0} for details.} 
Restricting $V_0$ and $V_1$ to host only the
cut-vertices (only $O(1/\tau)=o(k)$ of them in total, cf.
Definition~\ref{ellfine}\eqref{few}), resolves the problem.

The above justifies the distinction between the space $\colouringp{0}$ for embedding the cut-vertices and the space $\colouringp{1}\cup\colouringp{2}$ for embedding the shrubs. There are some other approaches which do not need to further split $\colouringp{1}\cup\colouringp{2}$ but doing so seems to be the most convenient.

\subsection{Stochastic process $\Duplicate(\ell)$}\label{ssec:Duplicate}
Let us introduce a class of stochastic processes, which we call
\index{mathsymbols}{*Duplicate@$\Duplicate(\ell)$}$\Duplicate(\ell)$ ($\ell\in\NN$). These are discrete processes $(X_1,Y_1),(X_2,Y_2),\ldots,(X_q,Y_q)\in\{0,1\}^2$ (where $q\in \mathbb N$ is arbitrary) satisfying the following.
\begin{itemize}
  \item For each $i\in[q]$, we have either
  \begin{enumerate}[(a)]
    \item $X_i=Y_i=0$ (deterministically), or
    \item $X_i=Y_i=1$ (deterministically), or
    \item\label{duplC} exactly one of $X_i$ and $Y_i$ is one, and in that case $\probability[X_i=1]=\frac12$.
  \end{enumerate}
  \item If the distribution of $(X_i,Y_i)$ is according to~\eqref{duplC}, then the random choice is made independently of the values $(X_j,Y_j)$ ($j<i$).
  \item We have $\sum_{i=1}^q(X_i+Y_i)\le \ell$.
\end{itemize}

We note that this definition is not deep and its purpose is only to adopt the language we shall use later. The following lemma asserts that the first and second components of a process $\Duplicate(\ell)$ are typically balanced.

\begin{lemma}\label{lem:randomduplicate}
Suppose that $(X_1,Y_1),(X_2,Y_2),\ldots,(X_q,Y_q)$ is a process in
$\Duplicate(\ell)$. Then for any $a>0$ we have
$$\probability\left[\sum_{i=1}^q
X_q-\sum_{i=1}^q Y_q\ge a\right]\le \exp\left(-\frac{a^2}{2\ell}\right)\;.$$
\end{lemma}
\begin{proof}
We shall use the following version of the Chernoff bound for sums of
independent random variables $Z_i$, with distribution
$\probability[Z_i=1]=\probability[Z_i=-1]=\frac12$.
\begin{equation}\label{eq:CHERNOFF}
\probability\left[\sum_{i=1}^n Z_i\ge a\right]\le
\exp\left(-\frac{a^2}{2n}\right)\;.
\end{equation}

Let $J\subset [q]$ be the set of all indices $i$ with $X_i+Y_i=1$. By the definition of $\Duplicate(\ell)$, we have $|J|\le \ell$. 
By~\eqref{eq:CHERNOFF} we have
\begin{align*}
\probability\left[\sum_J
(X_i-Y_i)\ge a\right]\le \exp\left(-\frac{a^2}{2|J|}\right)\le
\exp\left(-\frac{a^2}{2\ell}\right) \;.
\end{align*}
\end{proof}

We shall use the stochastic process $\Duplicate$ to guarantee that certain fixed
vertex sets do not get overfilled during our tree embedding procedure.
$\Duplicate$ is used in Lemmas~\ref{lem:embedStoch:DIAMOND6}
and~\ref{lem:embedStoch:DIAMOND7} through Lemma~\ref{lem:randomshrubembedding}. The way we use $\Duplicate$ was sketched in \cite[Section~\ref{p0.sssec:whyGexp}]{cite:LKS-cut0}.

\subsection{Embedding small trees}\label{ssec:EmbeddingShrubs}
When embedding the tree $T_\PARAMETERPASSING{T}{thm:main}$ in our proof of Theorem~\ref{thm:main} it will be important to control where different bits of $T_\PARAMETERPASSING{T}{thm:main}$ go. This motivates the following notation. Let $X_1,\ldots,X_\ell\subset V(T)$ be arbitrary vertex sets of a tree $T$, and let $V_1,\ldots,V_\ell\subset V(G)$ be arbitrary vertex sets of a graph $G$. Then an embedding $\phi:V(T)\rightarrow V(G)$ of $T$ in $G$ is an \emph{$(X_1\hookrightarrow V_1,\ldots,X_\ell\hookrightarrow V_\ell)$-embedding} \index{mathsymbols}{**embedding@$(X_1\hookrightarrow V_1,\ldots,X_\ell\hookrightarrow V_\ell)$-embedding}\index{general}{**embedding@$(X_1\hookrightarrow V_1,\ldots,X_\ell\hookrightarrow V_\ell)$-embedding} if $\phi(X_i)\subset V_i$ for each $i\in[\ell]$.

We provide several sufficient conditions for embedding a small tree with
additional constraints. 

\HIDDENTEXT{Lemma about spots hidden under SPOTS}

The first lemma deals with embeddings using an avoiding
set.
\begin{lemma}\label{lem:embed:avoidingFOREST}Let $\Lambda,k
\in \NN$ and let $\epsilon,\gamma\in (0,\frac12)$ with $\gamma^2 >\eps$.
Suppose $\smallatoms$ is a $(\Lambda,\epsilon,\gamma,k)$-avoiding set with respect to  a set  $\DenseSpots$ of $(\gamma k,\gamma)$-dense spots in a graph $H$. Suppose that $(T_1,r_1),\ldots,(T_\ell,r_\ell)$ are rooted trees with $|\bigcup_i T_i|\leq \gamma k/2$. Let $U\subset V(H)$ with $|U|\le \Lambda k$, and let $U^*\subseteq  \smallatoms$ with $ |U^*|\ge\eps k+\ell$.
 Then there are mutually disjoint $(r_i\hookrightarrow U^*, V(T_i)\setminus\{r_i\}\hookrightarrow V(H)\setminus U)$-embeddings of the trees $(T_i,r_i)$ in $H$.
\end{lemma}

\begin{proof}
Since $\smallatoms$ is $(\Lambda,\epsilon,\gamma,k)$-avoiding, there
exists a set $Y\subset \smallatoms$ with $|Y|\le\epsilon k$, such that each vertex $v$ in
$\smallatoms\setminus Y$ has degree at least $\gamma k$ into some $(\gamma
k,\gamma)$-dense spot $D\in \DenseSpots$ with $|U\cap V(D)|\le\gamma^2k$. In particular, $U^* \setminus Y$ is large enough so that we can embed all vertices $r_i$ there. We successively extend this embedding to an embedding of $\bigcup_i T_i$, at each step finding a suitable image in $V(D)\setminus U$ for one neighbour of an already embedded vertex $v\in \bigcup_i V(T_i)$. This is possible since the image of $v$ has degree at least $\gamma k - |U\cap V(D)|> \gamma k/2\geq \sum_i v(T_i) $ into $V(D)\setminus U$.
\end{proof}

The next lemma deals with embedding a tree into a nowhere-dense graph, a
prime example of which is the graph $\Gexp$.

\begin{lemma}\label{lem:embed:greyFOREST}
Let $k\in\NN$, let $Q\ge 1$ and let $\gamma,\zeta\in (0,1)$ be such that
 $128Q\gamma\le\zeta^2$. Let $H$ be a
$(\gamma k,\gamma)$-nowhere-dense graph. Let $(T_1,r_1),\ldots,(T_\ell,r_\ell)$ be rooted trees of total order less than $\zeta k/4$. Let $V_1,V_2,U,U^*\subset V(H)$  be
four sets with $U^*\subset V_1$, $|U|<Q k$,
$|U^*|>\frac{32Q^2\gamma}{\zeta} k +\ell$, and
$\mindeg_{H}(V_j,V_{3-j})\ge\zeta k$ for $j=1,2$.   Then there are mutually
disjoint $(r_i\hookrightarrow U^*, \Veven(T_i) \hookrightarrow V_1\setminus U,\Vodd(T_i)\hookrightarrow V_2\setminus U)$-embeddings of the trees
$(T_i,r_i)$ in $H$.
\end{lemma}
\begin{proof}
Set $B:=\shadow_H(U,\zeta k/2)$. By Fact~\ref{fact:shadowboundEXPANDER}, we have
$|B|\le\frac{32Q^2\gamma}{\zeta} k\le\zeta k/4$. In particular, $U^*\setminus B$ is large enough to accommodate the images $\phi (r_i)$ of all vertices $r_i$.

Successively, extend $\phi$, in each step mapping a neighbour
$u$ of some already embedded vertex $v\in \bigcup_i V(T_i)$ to a yet
unused neighbour of $\phi(v)$ in $V_j\setminus (B\cup
U)$, where $j$ is either 1 or 2, depending on the parity of $\dist_T(r,v)$.
 This is possible as
 $\phi(v)$, lying outside $B$,  has at least $\zeta k/2$ neighbours in $V_i\setminus U$. Thus $\phi (v)$ has  at least $\zeta k/4$ neighbours in $V_i\setminus (U\cup B)$, which is more than $\sum_iv(T_i)$. 
\end{proof}

The next three lemmas (Lemma~\ref{lem:embed:regular}--\ref{lem:embed:superregular}) deal with embedding trees in a regular or a super-regular pair. Before stating them, we give an auxiliary lemma that will be used in the proof of Lemma~\ref{lem:fillingCD}.
\begin{lemma}\label{lem:2ratio}
	Let $\{x_i\}_{i=1}^s$, $\{y_i\}_{i=1}^s$ be two families of reals in $[0, K]$, with $\sum_i x_i>0$. Write $X:=\sum_i x_i$, $Y:=\sum_i
	y_i$, and $\gamma:=Y/X$. Then for each $X'\in [0,X]$ there is a set
	$I\subset [s]$ such that 
	\begin{enumerate}[(a)]
		\item $\sum_{i\in I}x_i\le X'\le \sum_{i\in I} x_i+K$, and
		\item $\sum_{i\in I}y_i-K\le \gamma X'\le \sum_{i\in I}y_i+2K$. 
	\end{enumerate}
\end{lemma}
\begin{proof}
	Inductively construct sets $J_\ell\subseteq [s]$ as follows for $\ell =1,\ldots, s$. We start by setting $J_1=\emptyset$. 
	In step $\ell$, if $\gamma \sum_{j\in J_\ell}x_{j}\ge
	\sum_{j\in J_\ell} y_{j}$, then choose $j_\ell\in[s]\setminus J_{\ell}$ such that $\gamma x_{j_\ell}\le y_{j_\ell}$. Otherwise, take $j_\ell\in[s]\setminus J_{\ell}$  with $\gamma x_{j_\ell}> y_{j_\ell}$. The existence of such an index $j_\ell$ follows by averaging. Set $J_{\ell+1}:=J_{\ell}\cup\{j_\ell\}$. Our procedure ensures that for each $\ell$ we have
	\begin{equation}\label{yiyiyi}
	\sum_{j\in J_\ell } y_{j}-K\le \gamma \sum_{j\in J_\ell}x_j \le \sum_{j\in J_\ell } y_{j}+K. 
	\end{equation}
	
	Now for a given $X'$, let $p$ be the largest integer such that
	$\sum_{j\in J_p} x_{j}\le X'$. Setting $I:=J_p$, we clearly have
	(a), while the first inequality in (b) holds because of~\eqref{yiyiyi} (first
	inequality) for $\ell =p$.  For the second inequality in (b), it is enough to focus on the case $p\neq s$, as otherwise $X=X'$ and consequently $\gamma X'=\sum_{i\in I} y_{j}$.
	But then, by the definition of~$p$ and by~\eqref{yiyiyi} (second inequality) for $\ell= p+1$,
	\[
	\gamma X' \leq \gamma \sum_{i\in J_{p+1}}x_i \leq \sum_{i\in J_{p+1}}y_i+K \leq\sum_{i\in I}y_i+2K,
	\]
	as desired.
\end{proof}

\begin{lemma}\label{lem:embed:regular}
 Let $\epsilon>0$
and $\beta>2\epsilon$. Let $(C,D)$ be an $\epsilon$-regular pair in a graph $H$, with $|C|=|D|=:\ell$, and with density $\density(C,D)\geq
3\beta$. Suppose that there are sets $X\subseteq C$, $Y\subseteq D$, and
$X^*\subset X$ satisfying $\min\{|X|,|Y|\}\ge 4\frac{\epsilon}{\beta}\ell$ and
$|X^*|> \frac\beta2\ell$.  Let $(T,r)$ be a rooted tree of order $v(T)\le
\epsilon \ell$.  Then there exists an $(r\hookrightarrow
X^*,\Veven(T)\hookrightarrow X,\Vodd(T)\hookrightarrow Y)$-embedding of $T$ in $H$.
\end{lemma}
\begin{proof} We shall construct an embedding $\phi: V(T)\rightarrow X\cup
	Y$ satisfying the requirements of the lemma.
	Fact~\ref{fact:BigSubpairsInRegularPairs} implies that $(X,Y)$ is $\beta/2$-regular of density greater than $2\beta$. By
	Fact~\ref{fact:manyTypicalVertices}, there are sets
	$X'\subset X$ and $Y'\subset Y$ with $|X'|>(1-\beta/2) |X|$ and $|Y'|>(1-\beta /2) |Y|$ such
	that $\mindeg(X',Y)\ge \frac 32\beta |Y|$, $\mindeg(Y',X)\ge
	\frac 32 \beta |X|$. Then
	\begin{equation}\label{eq:mindegreeenoughembed}
	\begin{split}
	\mindeg(H[X',Y'])\ge \beta \min\{|X|,|Y|\}\ge 2\epsilon
	\ell > v(T).
	\end{split}
	\end{equation}
	Choose any vertex in $X^*\cap X'$ (which is non-empty by the
	above calculations) for $\phi(r)$. By~\eqref{eq:mindegreeenoughembed}  we can greedily extend $\phi$ to an embedding $\phi: V(T)\rightarrow X'\cup Y'$.
\end{proof}

\begin{lemma}\label{lem:fillingCD}
 Let
$\beta,\epsilon>0$ and $\ell \in\NN$ be such that $\epsilon\le \beta^2/8$. Let
$(C,D)$ be an $\epsilon$-regular pair with $|C|=|D|=\ell$ of density
$\density(C,D)\geq 3\beta$ in a graph $H$. 
Let $(T_1,r_1),(T_2,r_2),\ldots ,(T_s,r_s)$ be rooted trees with
$v(T_i)\leq\epsilon\ell$ for all $i\in [s]$. Let $U\subset V(H)$ fulfill $|C\cap U|=|D\cap U|$, and let $X^*\subseteq (C\cup D)\setminus U$ be such that
\begin{equation}\label{eq:conFill}
|X^*|\geq \sum_{i=1}^sv(T_i)
+ 50\beta\ell\;.
\end{equation}
 Then there are mutually
disjoint $(r_i\hookrightarrow X^*, V(T_i) \hookrightarrow (C\cup D)\setminus U)$-embeddings of the trees
$(T_i,r_i)$ in~$H$.
\end{lemma}
\begin{proof}
	Let us write $M:=|X^*\cap C|$ and $m:=|X^*\cap D|$. Without loss of generality, we assume that $M\ge m$. For each $i\in [s]$, let us write $a_i$ and $b_i$ for the
	number of vertices of $T_i$ at even and odd distance from $r_i$, respectively.
	Furthermore, we write $A:=\sum_i a_i$, $B:=\sum_i b_i$, and $\gamma:=B/A$.
%
	In a first step, we shall partition the set $[s]$ into three sets $I_1,I_2$, and $I''$, according to three cases: 
	\begin{itemize}
		\item [(C1)] $m\le 4\beta\ell$,
		\item [(C2)]  $m>4\beta\ell$, and  $2(m-4\beta\ell)\ge A+B$,
		\item [(C3)]  $m>4\beta\ell$, and  $2(m-4\beta\ell)< A+B$.
	\end{itemize}
	
	Once this has been done, we will show how to embed the rooted trees $T_i$, using this partition.	

	In Case~(C1), we set $I_1=I_2:=\emptyset$, and $I'':=[s]$. 
	For Cases~(C2), and~(C3), we first partition~$[s]$ into two sets $I$ and $I''$ and will make use of an auxiliary set $I'$ in order to obtain~$I_1$ and~$I_2$ as follows. 
In Case~(C2), set $I:=[s]$, $I'':=\emptyset$, and $I':=I$.
In Case~(C3), 
	we apply Lemma~\ref{lem:2ratio} with
	input $(x_i)_{i\in[s]}:=(a_i)_{i\in[s]}$, $(y_i)_{i\in[s]}:=(b_i)_{i\in[s]}$,
	$X':=\frac{2A}{A+B}(m-4\beta\ell)$, and the bound $K:=\frac \beta 4\ell$. The  bound $X'\leq X=A$ required in Lemma~\ref{lem:2ratio} follows from the second property of Case~(C3). The lemma yields a set
	$I\subset[s]$ such that 
	\begin{align}
	\label{eq:I'x1}\sum_{I} a_i&\le \frac{2A}{A+B}(m-4\beta\ell)\;,\\
	\label{eq:I'x2}\frac{2A}{A+B}(m-4\beta\ell)&\le  \sum_{I} a_i+\frac \beta 4\ell\;,\\
	\label{eq:I'y1}\sum_{I} b_i-\frac \beta 4\ell&\le  \frac{2B}{A+B}(m-4\beta\ell)\;\mbox{, and}\\
	\label{eq:I'y2}\frac{2B}{A+B}(m-4\beta\ell)&\le \sum_{I} b_i+\frac \beta 2\ell\;.
	\end{align}
	Bound~\eqref{eq:I'x2} can be used to bound
	$\sum_{I''}a_i$ for the complementary set $I'':=[s]\setminus
	I$ as follows.
	\begin{align}
	\nonumber
	\sum_{I''}a_i&\le
	A-\left(\frac{2A}{A+B}(m-4\beta\ell)-\frac \beta 4\ell\right)=A\left(1-\frac{2(m-4\beta\ell)}{A+B}\right)+\frac \beta 4\ell\\ &\le
	(M+m-50\beta \ell)\left(1-\frac{2(m-4\beta\ell)}{M+m-50\beta \ell}\right)+\frac \beta 4\ell\le M-m-40\beta \ell\;,
	\label{eq:I''x}
	\end{align}
	where we employed the bound $A\le A+B\le M+m-50\beta\ell$ from~\eqref{eq:conFill}.
	Likewise, we have from~\eqref{eq:I'y2} that
	\begin{align}
	\label{eq:I''y}
	\sum_{I''}b_i\le  M-m-40\beta \ell\;.
	\end{align}
	The main feature of Lemma~\ref{lem:2ratio} is that the ratio
	$\sum_{I}b_i:\sum_{I}a_i$ is almost exactly $\gamma$. In order to even out
	a small imperfection we may have, let us introduce a dummy pair $(a_0,b_0)$, with
	$0<a_0, b_0\le\beta\ell/2$ such that for $I':=I\cup\{0\}$, we have 
	$$\frac{\sum_{I'}b_i}{\sum_{I'}a_i}=\gamma\;.$$
	The existence of such a pair $(a_0,b_0)$ follows from the properties of
	Lemma~\ref{lem:2ratio}.
	
In Cases~(C2), and~(C3), we apply Lemma~\ref{lem:2ratio}  to further partition the set~$I$.
	More specifically, the input of Lemma~\ref{lem:2ratio} consists of
	$X':=\frac{A}{A+B}(m-4\beta \ell)$, $(x_i)_{i\in I'}:=(a_i)_{i\in I'}$, 
	$(y_i)_{i\in I'}:=(b_i)_{i\in I'}$, and $K:=\frac \beta 2\ell$.  Lemma~\ref{lem:2ratio}  gives
	an index set $J_1\subset I'$. Set $I_1:=J_1\setminus\{0\}\subseteq I$. We have that
	\begin{align}
	\label{eq:I1x1}\sum_{I_1} a_i&\le \frac{A}{A+B}(m-4\beta \ell)\;,\\
	\label{eq:I1x2}\frac{A}{A+B}(m-4\beta \ell)&\le \sum_{I_1} a_i+\beta \ell\;,\\
	\label{eq:I1y1}\sum_{I_1} b_i-\frac \beta 2\ell&\le \frac{B}{A+B}(m-4\beta \ell)\;\mbox{, and}\\
	\label{eq:I1y2}\frac{B}{A+B}(m-4\beta \ell)&\le \sum_{I_1} b_i+\frac {3}{2}\beta \ell\;.
	\end{align}
	Set $I_2:=I\setminus I_1$. From~\eqref{eq:I'x1} and~\eqref{eq:I1x2}
	we have
	\begin{equation}\label{eq:I2x}
	\sum_{I_2} a_i\le
	\frac{2A}{A+B}(m-4\beta\ell)-\left(\frac{A}{A+B}(m-4\beta\ell)-\beta \ell\right)=\frac{A}{A+B}(m-4\beta \ell)+\beta \ell\;.
	\end{equation}
	Similarly,~\eqref{eq:I'y1} and~\eqref{eq:I1y2} give
	\begin{equation}\label{eq:I2y}
	\sum_{I_2} b_i\le \frac{B}{A+B}(m-4\beta \ell)+2\beta\ell\;.
	\end{equation}
	We shall now see how the partition $[s]=I_1\cup I_2\cup I''$ gives us
	instructions to embed the trees $T_1,\ldots,T_s$ one by one. The
	trees $T_i$, $i\in I_1$ are embedded in the bipartite graph $(W,(D\cap X^*)\setminus U)$, where $W$ is an
	arbitrary subset of $C\cap X^*$ of size $m$, with the root $r_i$
	embedded in $W$. The trees $T_i$, $i\in
	I_2$ are embedded in the bipartite graph $(D\cap X^*, W)$, with the root $r_i$
	embedded in $D\cap X^*$. Finally, the trees $T_i$, $i\in I''$ are
	embedded in $((C\cap X^*)\setminus W,D\setminus (X^*\cup U))$, with the root
	embedded in $(C\cap X^*)\setminus W$. We can embed the trees
	$(T_i)_{i\in I_1\cup I_2}$ as described above, by repetitively using 
	Lemma~\ref{lem:embed:regular}, as we have enough space for the embeddings:
	summing up~\eqref{eq:I1x1} and~\eqref{eq:I2y} we have $$\sum_{I_1}
	a_i+\sum_{I_2}b_i\le m-4\beta \ell+2\beta \ell=|W|-2\beta \ell\;,$$
	and similarly from~\eqref{eq:I1y1} and~\eqref{eq:I2x} we have
	$$\sum_{I_1}
	b_i+\sum_{I_2}a_i\le m-4\beta \ell+\frac 32\beta \ell\le |D\cap X^*|-2\beta \ell\;.$$
	Likewise, the trees $(T_i)_{i\in I''}$ can be embedded in~$((C\cap X^*)\setminus
	W,D\setminus (X^*\cup U))$ with the help of Lemma~\ref{lem:embed:regular},
	as~\eqref{eq:I''x} says that $\sum_{I''}a_i\le |(C\cap X^*)\setminus W|-40\beta \ell$, and as~\eqref{eq:I''y} says that $\sum_{I''}b_i\le
	|(C\cap X^*)\setminus W|-40\beta \ell\le|D\setminus (X^*\cup U)|-40\beta \ell$.
\end{proof}

\begin{lemma}\label{lem:embed:superregular}
Let $d>10\epsilon>0$. Suppose that $(A,B)$ forms an $(\epsilon,d)$-super-regular pair with $|A|,|B|\ge \ell$. Let $U_A\subset A$, $U_B\subset B$ be such that $|U_A|\le |A|/2$ and $|U_B|\le d|B|/4$. Let $(T,r)$ be a rooted tree of order at most $d\ell/4$, and let $v\in A\setminus U_A$ be arbitrary. Then there exists an $(r\hookrightarrow v,\Veven(T,r)\hookrightarrow A\setminus U_A, \Vodd(T,r)\hookrightarrow B\setminus U_B)$-embedding of $T$.
\end{lemma}
\begin{proof}[Sketch of the proof]
The lemma is a variant of Lemma~\ref{lem:embed:regular} with only two qualitative differences. Firstly, the assumptions of the lemma are stronger in that we now have super-regularity rather than regularity. Secondly, the assertion of the lemma is stronger in that we can map the root of the tree on a specific vertex $r\hookrightarrow v$, rather  than into a specified set $r_\PARAMETERPASSING{L}{lem:embed:regular}\hookrightarrow X^*_\PARAMETERPASSING{L}{lem:embed:regular}$. The proof scheme of Lemma~\ref{lem:embed:regular} indeed gives this stronger assertion under the current assumptions. To see this, note that in the proof of Lemma~\ref{lem:embed:regular}, it was enough to map $r$ to an arbitrary vertex which had enough degree to the destination set ($B\setminus U_B$, in the present lemma) of its children. In the current setting, any $v\in A\setminus U_A$ can serve as such a vertex as $\deg(v, B\setminus U_B)\ge \deg(v, B)-|U_B|\ge d|B|-d|B|/4=\frac34d|B|$, where the last inequality uses the super-regularity of $(A,B)$.
\end{proof}

Suppose that we have to embed a rooted tree $(T,r)$, and its root was
already mapped on a vertex $\phi(r)$. Suppose that $r$ has degree $\ell_X+\ell_Y$ in a
regular pair $(X,Y)$, where $\ell_X:=\deg(\phi(r), X), \ell_Y:=\deg(\phi(r),
Y)$, with $\ell_X\ge \ell_Y$, say.
The hope is that we can embed $T$ in $(X,Y)$ as long as $v(T)$ is a bit smaller
than $\ell_X+\ell_Y$. For this, the greedy strategy does not work (see
Figure~\ref{fig:embedBalancedUnbalance}) and we need to be somewhat more careful.
\begin{figure}[ht]
\centering 
\includegraphics{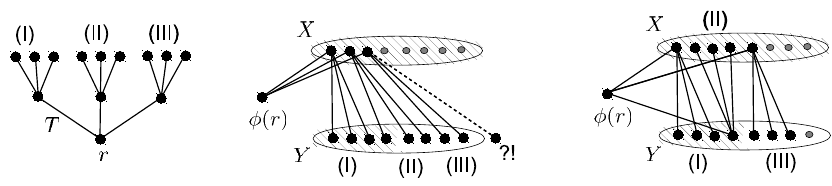}
\caption[Balanced and unbalanced embedding in a regular pair]{An example of a
rooted tree $(T,r)$, depicted on the left. The forest $T-r$ has three components (I), (II), (III) of total
order 12. Say the vertex $r$ is embedded so that for the regular pair $(X,Y)$ we
have $\deg(\phi(r),X)=8$, $\deg(\phi(r),Y)=4$ (neighbourhoods of $\phi(r)$
hatched).

While the greedy strategy does not work (middle), splitting the process into
a balanced and an unbalanced stage (right) does --- here the components~(I) and~(II)
are embedded in the balanced stage and the component~(III) in the unbalanced
stage.}
\label{fig:embedBalancedUnbalance}
\end{figure}
 We split the embedding process into two stages.
In the first stage we choose a subset of the components of $T-r$ of total order 
approximately $2\min\big(\ell_X,\ell_Y\big)=2\ell_Y$. When embedding these, we choose 
orientations of each component in such a way that the image is approximately
balanced with respect to $X$ and $Y$. In the second stage we 
embed the remaining components so that their roots are embedded in $X$.
We refer to the first stage as 
\index{general}{balanced way of embedding}\index{general}{unbalanced way of embedding}\emph{embedding in a balanced way}, and 
to the second stage as  \emph{embedding in an unbalanced way}.

The next lemma says that each regular pair can be filled-up in a balanced way by
trees.
\begin{lemma}\label{lem:embed:BALANCED}
Let $G$ be a graph, $v\in V(G)$ be a vertex, $\M$ be an
$(\epsilon,d,\nu k)$-\semiregular matching in $G$, and $\{f_{CD}\}_{(C,D)\in\M}$ be a
family of integers between $-\tau k$ and $\tau k$. Suppose $(T,r)$ is a rooted tree,
$$v(T)\le \left(1-\frac{4(\epsilon+\frac{\tau}\nu)}{d-2\epsilon}\right)|V(\M)|\;,$$ with the property that each component of $T-r$ has order at most $\tau k$. If $V(\M)\subset \neighbour_G(v)$ then there
exists an $(r\hookrightarrow v,V(T-r)\hookrightarrow V(\M))$-embedding $\phi$ of $T$ such that for each $(C,D)\in\M$ we have $|C\cap
\phi(T)|+f_{CD}=|D\cap \phi(T)|\pm\tau k$.
\end{lemma}

The proof of Lemma~\ref{lem:embed:BALANCED} is standard, and is given for
example in~\cite[Lemma~5.12]{HlaPig:LKSdenseExact}. 

Lemma~\ref{lem:embed:BALANCED} suggests the following definitions. 
The \emph{discrepancy}\index{general}{discrepancy} of a set $X$ with respect to a pair of sets $(C,D)$ is the number $|C\cap X|-|D\cap X|$.
$X$ is \emph{$s$-balanced}\index{general}{balanced set} with respect to a \semiregular matching $\M$ if the discrepancy of $X$ with respect to each $(C,D)\in \M$ is at most $s$ in absolute value.
\begin{lemma}\label{lem:embed:ONESIDE}
Let $G$ be a graph, $v\in V(G)$ be a vertex, $\M$ be an
$(\epsilon,d,\nu k)$-\semiregular matching in $G$ with an $\M$-cover $\mathcal F$, and $U\subset
V(G)$. Suppose $(T,r)$ is a rooted tree with 
$$
v(T)+|U|\le
\deg_G\left(v,V(\M)\setminus \bigcup\mathcal F\right)-\frac{4(\epsilon+\frac{\tau}\nu)}{d-2\epsilon}|V(\M)|\;,$$
such that each component
of $T-r$ has order at most $\tau k$. Then there
exists an $(r\hookrightarrow v,V(T-r)\hookrightarrow V(\M)\setminus U)$-embedding $\phi$ of $T$.
\end{lemma}
The proof of Lemma~\ref{lem:embed:ONESIDE} is again standard and we again omit it.

The following lemma uses a probabilistic technique to embed a shrub while reserving a  set of vertices in the host graph for later use. We wish the reserved set to use about as much space inside certain given sets $P_i$ as the image of our shrub does. (In later applications the sets $P_i$ correspond to neighbourhoods of vertices which are still `active'.) 

Lemma~\ref{lem:randomshrubembedding} will find an immediate application in all the remaining lemmas of this subsection. However it is really necessary only for Lemmas~\ref{lem:embedStoch:DIAMOND6}--\ref{lem:embedStoch:DIAMOND7}, which deal with
embedding shrubs in the presence of one of the Configurations~$\mathbf{(\diamond
6)}$--$\mathbf{(\diamond8)}$. For Lemmas~\ref{lem:HE3} and~\ref{lem:HE4}, which are for Configurations~$\mathbf{(\diamond
3)}$ and $\mathbf{(\diamond4)}$, a simpler auxiliary lemma (without reservations) would suffice.

\begin{lemma}\label{lem:randomshrubembedding}
Let $H$ be a graph, let $X^*,X_1,X_2,
P_1,P_2,\ldots,P_L\subset V(H)$, and let $(T_1,r_1)$, \ldots, 
$(T_\ell,r_\ell)$ be rooted trees, such that $L\le k$,
$|P_j|\le k$ for each $j\in[L]$, and $|X^*|\geq 2\ell$.
Suppose that $\mindeg (X_1\cup X^*,X_{2})\geq 2\sum v(T_i)$ and $\mindeg
(X_2,X_{1})\geq 2\sum v(T_i)$.

Then there exist pairwise disjoint $\left(r_i\hookrightarrow
X^*,\Veven(T_i,r_i)\setminus\{r_i\}\hookrightarrow
X_1,\Vodd(T_i,r_i)\hookrightarrow X_2\right)$-em\-beddings $\phi_i$ of $T_i$ in
$G$ and a set $C\subset (X_1\cup X_2)\sm\bigcup\phi_i(T_i)$ of size $\sum v(T_i)$ such that for each $j\in[L]$ we have
\begin{equation}
  \label{Pjdoesitright}  |P_j\cap \bigcup\phi_i(T_i)|\le |P_j\cap C|+k^{3/4}.
\end{equation}
\end{lemma}

\begin{proof}
Let $m:=\sum v(T_i)$.

We construct pairwise disjoint random $\left(r_i\hookrightarrow
X^*,\Veven(T_i,r_i)\setminus\{r_i\}\hookrightarrow
X_1,\Vodd(T_i,r_i)\hookrightarrow X_2\right)$-embeddings $\phi_i$ and a set
$C\subseteq V(H)\sm\bigcup\phi_i(T_i)$ which satisfies~\eqref{Pjdoesitright}
with positive probability. Then the statement follows.

Enumerate the vertices of $\bigcup T_i$ as $\bigcup
V(T_i)=\{v_1,\ldots,v_{m}\}$ such that $v_i=r_i$ for
$i=1,\ldots,\ell$, and such that for each $j>\ell$ we have that the parent of
$v_j$ lies is the set $\{v_1,\ldots,v_{j-1}\}$.
Pick pairwise disjoint sets $A_1,\ldots,A_\ell\subset X^*$ of size two.
Choose uniformly and independently at random an element $x_j\in A_j$. Denote the other element of $A_j$ as $y_j$. 

Now, successively for $i= \ell+1,\ldots, m$, we shall  define vertices
$x_i$ and $y_i$. Let $r$ denote the root of the tree in which $v_i$ lies, and
let $v_s=\parent(v_i)$ be the parent of~$v_i$.
We shall choose $x_{i},y_j\in X_{j_i}$ where $j_i=\dist (r,v_i) \mod 2 +1$.
In step $i$, proceed as follows. Since $x_s\in X_{j_{s}}$ (or since $x_s\in X^*$), we have 
$$\deg (x_{s},X_{j_i}\sm  \bigcup_{h<i}\{x_h,y_h\})\geq 2.$$ Hence, 
we may take an arbitrary subset $A_i\subseteq (\neighbour(x_{s})\cap
X_{{j_{i}}})\setminus  \bigcup_{h<i}\{x_h,y_h\}$ of size exactly two. As above,
randomly label its elements as $x_i$ and $y_i$ independently of all other
choices.

The choices of the maps $(v_j\mapsto x_j)_{j=1}^{m}$ determine
$\phi_1,\ldots,\phi_\ell$.
Then $C:=\{y_1,\ldots,y_m\}$
has  size exactly $m$ and avoids $\bigcup \phi_i (T_i)$. 

 For each
$j\in[L]$ we set up a stochastic process $\mathfrak
S^{(j)}=\left((X_i^{(j)},Y_i^{(j)})\right)_{i=1}^{m}$, defined by
$X^{(j)}_i=\mathbf{1}_{\{x_i\in P_j\}}$ and $Y^{(j)}_i=\mathbf{1}_{\{y_i\in P_j\}}$. Note that $\mathfrak S^{(j)}\in
\Duplicate(|P_j|)\subset\Duplicate(k)$. Thus, for a fixed $j\in[L]$, by
Lemma~\ref{lem:randomduplicate}, the probability that
$|P_j\cap(\bigcup\phi_i(T_i))|>|P_j\cap C|+k^{3/4}$ is at most
$\exp(-\sqrt{k}/2)$.
Using the union bound over all $j\in [L]$ we get that Property~\ref{stochEmbedDIAMOND6P2} holds with probability at least $$1-L\cdot \exp\left(-\frac{\sqrt{k}}{2}\right)>0\;.$$ This finishes the proof.
\end{proof}

We now get to the first application of Lemma~\ref{lem:randomshrubembedding}.

\begin{lemma}\label{lem:embedStoch:DIAMOND6}
Assume we are in Setting~\ref{commonsetting}. Suppose
that we are given sets $V_2,V_3\subset V(G)$ are such that we have
\begin{equation}\label{mucuruvi}
\mindeg_H(V_2, V_3)\geq \delta k\quad\mbox{and}\quad\mindeg_H(V_3, V_{2})\geq \delta k,
\end{equation}
where  $\delta>300/k$, and $H$ is a $(\gamma k, \gamma)$-nowhere dense subgraph of $G$. Suppose that $U,U^*,
P_1,P_2,\ldots,P_L\subset V(G)$,
and $L\le k$, are such that $|U|\le \frac{ \delta}{24\sqrt\gamma} k$, $U^*\subset V_2$, $|U^*|\ge\frac{\delta}{4}k$, and $|P_j|\le k$ for each
$j\in[L]$.
Let $(T,r)$ be a rooted tree of order at most $\delta k/8$.

Then there exists a $\left(r\hookrightarrow U^*,\Veven(T,r)\setminus\{r\}\hookrightarrow
V_2\setminus U,\Vodd(T,r)\hookrightarrow V_3\setminus U\right)$-embedding $\phi$ of $T$ in
$G$ and a set
 $C\subset (V_2\cup V_3) \sm (U\cup \phi(T))$ of size $v(T)$ such that for each $j\in[L]$ we have
\begin{equation}
  \label{stochEmbedDIAMOND6P2} |P_j\cap \phi(T)|\le |P_j\cap C|+k^{3/4}.
\end{equation}
\end{lemma}

\begin{proof}
Set $B:=\shadow_{\Gexp}(U,\delta k/4)$.
By Fact~\ref{fact:shadowboundEXPANDER}, we have that
$|B|\le 64\frac{\gamma}{\delta}(\frac{ \delta}{24\sqrt\gamma})^2 k\le\frac{\delta}{4}k-2$. In
particular, $X^*:=U^*\setminus B$ has size at least $2$. Set $X_1:=V_2\sm (U\cup B)$ and set $X_2:=V_3\sm (U\cup B)$. Using~\eqref{mucuruvi},
we find that
$$\mindeg_{\Gexp}(X_1,X_2)\ge \delta k - \maxdeg_{\Gexp}(X_1,U) - |B|\ge 
\delta k - \frac{\delta}{4}k - \frac{\delta}{4}k
\ge 2v(T)\;,$$
and similarly, $\mindeg_{\Gexp}(X_2,X_1)\ge 2v(T)$.
We may thus apply Lemma~\ref{lem:randomshrubembedding} to obtain the desired embedding $\phi$ and the set $C$.
\end{proof}

\begin{lemma}\label{lem:embedStoch:DIAMOND7}
Assume Setting~\ref{commonsetting} and Setting~\ref{settingsplitting}. Suppose
that we are given sets $Y_1,Y_2\subset\colouringp{1}\setminus \exceptVertSplit$
with $Y_1\subset\smallatoms$, and
\begin{enumerate}[(i)]
\item $\maxdeg_{\GD}(Y_1,\colouringp{1}\setminus
Y_2)\le \frac{\eta \gamma}{400}$, and \label{luckyluke}
\item $\mindeg_{\GD}(Y_2,Y_1)\ge \delta k$.\label{jollyjumper}
\end{enumerate}

Suppose that $U,U^*, P_1,P_2,\ldots,P_L\subset V(G)$ are sets such that $|U|\le
\frac{\Lambda\delta}{2\Omega^*} k$, $U^*\subset Y_1$,
with $|U^*|\ge\frac{\delta}{4}k$, $|P_j|\le k$ for each $j\in[L]$, and $L\le
k$.
Suppose $(T_1,r_1),\ldots,(T_\ell,r_\ell)$ are rooted trees of total order at
most $\delta k/1000$.
Suppose further that $\delta<\eta\gamma/100$, $\epsilon'<\delta/1000$,  and $k>1000/\delta$.

Then there exist pairwise disjoint $\left(r_i\hookrightarrow
U^*,\Veven(T_i,r_i)\hookrightarrow Y_1\setminus U,\Vodd(T_i,r_i)\hookrightarrow
Y_2\setminus U\right)$-em\-beddings $\phi_i$ of $T_i$ in $G$ and a set $C\subset
V(G-\bigcup \phi_i(T_i))$ of size $\sum v(T_i)$ such that for each $j\in[L]$ we
have that
\begin{equation}
 \label{stochEmbedDIAMOND7P2} |P_j\cap \bigcup\phi_i(T_i)|\le |P_j\cap
 C|+k^{3/4}.
\end{equation}
\end{lemma}
\begin{proof} Set $U':=\shadow_{\GD}(U,\delta k/2)\cup U$. By
Fact~\ref{fact:shadowbound}, we have $|U'|\le \Lambda k$. As $Y_1$ is a
$(\Lambda,\epsilon',\gamma,k)$-avoiding set, by Definition~\ref{def:avoiding}
there exists a set $B\subset Y_1$, $|B|\le \epsilon' k$ such that for all $v\in
Y_1\setminus B$ there exists a dense spot $D_v\in\DenseSpots$ with $v\in V(D_v)$
and $|V(D_v)\cap U'|\le \gamma^2k$. As $Y_1$ is disjoint from
$\exceptVertSplit$, by Definition~\ref{def:proportionalsplitting}\eqref{It:H4}
and by~\eqref{eq:proporcevelke}, we have that $\deg_{D_v}(v,V(D_v)\colouringpI{1})\ge\frac{\eta\gamma}{200} k$. 
By~\eqref{luckyluke}, we have that $\deg_{\GD}(v,V(D_v)\colouringpI{1}\setminus Y_2)<\frac{\eta\gamma
}{400}k$, and hence, $$\deg_{\GD}\big(v,(V(D_v)\colouringpI{1}\cap Y_2)\setminus
U'\big)\ge \frac{\eta\gamma k}{400}-\gamma^2k\ge \frac{\eta\gamma k}{800}\;.$$ Thus,
\begin{equation}\label{gr1}
\mindeg_{\GD}(Y_1\setminus B, Y_2\setminus (U'\cup B))\ge
\frac{\eta\gamma k}{800} -\eps'k\ge 2\sum v(T_i)\;.
\end{equation}
Further, by the definition of $U'$ and by~\eqref{jollyjumper}, we have
\begin{equation}\label{gr2}
\mindeg_{\GD}(Y_2\setminus U',Y_1\setminus (U\cup B))\ge \frac{\delta k}2 -\eps'k \ge
2\sum v(T_i)\;.
\end{equation}

Set $X^*:=U^*\setminus B$, and note that $|X^*|\ge \delta
k/4-\epsilon' k\ge 2\ell$. Set $X_1:=Y_1\sm (U\cup B)$ and $X_2:=Y_2\sm (U'\cup
B)$.
Inequalities~\eqref{gr1} and~\eqref{gr2} guarantee that we may apply
Lemma~\ref{lem:randomshrubembedding} to obtain the desired embeddings $\phi_i$.
\end{proof}

\begin{lemma}\label{lem:HE3}
Assume Setting~\ref{commonsetting}. Suppose that the sets $L',L'',\HugeVertices',\HugeVertices'',V_1,V_2$ witness Configuration~$\mathbf{(\diamond3)}(0,0,\gamma/4,\delta)$.
Suppose that $U,U^*\subset V(G)$ are sets such that $|U|\le  k$,
$U^*\subset V_1$, $|U^*|\ge\frac{\delta}{4}k$.
Suppose $(T,r)$ is a rooted tree of order at most $\delta k/1000$. Suppose further that $\delta\leq\gamma/100$, $\epsilon'<\delta/1000$, and $4\Omega^*/\delta\le \Lambda$.

Then there is an $\left(r\hookrightarrow U^*,\Veven(T,r)\setminus\{r\}\hookrightarrow V_1\setminus
U,\Vodd(T,r)\hookrightarrow V_2\setminus U\right)$-embedding of~$T$ in~$G$.
\end{lemma}

\begin{proof} The proof of this lemma is very similar to the one of Lemma~\ref{lem:embedStoch:DIAMOND7} (in fact, even easier). 
Set $U':=\shadow_{\GD}(U,\delta k/2)\cup U$ and note that $|U'|\le \Lambda k$ by
Fact~\ref{fact:shadowbound}. As $V_1$ is $(\Lambda,\epsilon',\gamma,k)$-avoiding, by Definition~\ref{def:avoiding} there is a set $B\subset V_1$, $|B|\le \epsilon' k$ such that for all $v\in V_1\setminus B$
there exists a dense spot $D_v\in\DenseSpots$ with 
$\deg_{D_v}(v,V(D_v)\sm U')\ge \gamma k/2$. 
By~\eqref{eq:WHtc}, we know that $\deg_{\GD}(v,V(D_v)\setminus
V_2)\leq \gamma  k/4$, and hence, 
$\deg_{\GD}\big(v,(V(D_v)\cap V_2)\setminus U'\big)\ge
\gamma k/{4}$.
Thus,
\begin{equation}\label{gr1grgr}
\mindeg_{\GD}(V_1\setminus B, V_2\setminus U')\ge
\frac{\gamma k}{4} \ge 2v(T)\;.
\end{equation}
Further, by the definition of $U'$ and by~\eqref{confi3theothercondi}, we have
\begin{equation}\label{gr2grgr}
\mindeg_{\GD}(V_2\setminus U',V_1\setminus U)\ge \frac{\delta k}2 \ge 2v(T)\;.
\end{equation}

Set $X^*:=U^*\setminus B$, and note that $|X^*|\ge \delta k/4-\epsilon' k\ge 2$. Set $X_1:=V_1\sm (U\cup B)$ and $X_2:=V_2\sm (U'\cup B)$. Inequalities~\eqref{gr1grgr} and~\eqref{gr2grgr} guarantee that we may apply Lemma~\ref{lem:randomshrubembedding} (with empty sets $P_i$) to obtain the desired embedding $\phi$.
\end{proof}

\begin{lemma}\label{lem:HE4}
Assume Setting~\ref{commonsetting}. Suppose that the sets $L',L'',\HugeVertices',\HugeVertices'',V_1,\smallatoms', V_2$ witness Configuration~$\mathbf{(\diamond4)}(0,0,\gamma/4,\delta)$.
Suppose that $U\subset V(G)$, $U^*\subset V_1$ are sets such
that $|U|\le k$ and $|U^*|\ge\frac{\delta}{4}k$.
Suppose $(T,r)$ is a rooted tree of order at most $\delta k/20$ with a fruit $r'$. Suppose further that $4\epsilon'\le \delta\le \gamma/100$, and $ \Lambda\geq 300(\frac{\Omega^*}{\delta})^3$.

Then there exists an $\left(r\hookrightarrow U^*,r'\hookrightarrow
V_1\setminus U, 
V(T)\setminus\{r,r'\}\hookrightarrow (\smallatoms'\cup V_2)\setminus
 U\right)$-embedding of~$T$ in~$G$.
\end{lemma}
\begin{proof}\
Set 
$$U':=\tilde U\cup \shadow_{\Gcapt-\HugeVertices}(U,\delta k/4)\cup
\shadow^{(2)}_{\Gcapt-\HugeVertices}(\tilde U,\delta k/4)$$ and let $$U'':=\tilde U\cup \shadow_{\GD}(U',\delta
k/2).$$ We use Fact~\ref{fact:shadowbound} to see that $|U'|\le \frac{\delta}{4\Omega^*} \Lambda
 k$ and  $|U''|\le \Lambda
 k$.  We then use Definition~\ref{def:avoiding} and~\eqref{confi4lastcondi} to find a set $B\subset \smallatoms'$ of size at most $\epsilon' k$ such that
\begin{equation}\label{gr81grgr}
\mindeg_{\GD}(\smallatoms'\setminus B, V_2\setminus U'')\ge 2v(T)\;.
\end{equation}

 Using~\eqref{gr81grgr},  and employing~\eqref{confi4:3} and ~\eqref{confi4othercondi}, we see that we may apply Lemma~\ref{lem:randomshrubembedding} with $X^*_\PARAMETERPASSING{L}{lem:randomshrubembedding}:= U^*$,
 $X_{1,\PARAMETERPASSING{L}{lem:randomshrubembedding}}:=\smallatoms'\sm (B\cup U')$ and $X_{2,\PARAMETERPASSING{L}{lem:randomshrubembedding}}:=V_2\sm U''$ (and with empty sets $P_i$) in order to embed the tree $T-T(r,\uparrow r')$ rooted at $r$.
Then embed $T(r,\uparrow r')$, by applying  Lemma~\ref{lem:randomshrubembedding} a second time, using~\eqref{confi4:3} and~\eqref{confi4:4}.
\end{proof}

\subsection{Main embedding lemmas}\label{sec:MainEmbedding}
For this section, we need to introduce the notion of a ghost.
The idea  behind this notion is that once we used a set $U$ for the embedding of our tree, the remainder of the graph cannot be used as before. Namely, if $U$ covers part of a cluster of some matching edge, then we will not be able to fill up the partner cluster using usual regularity embedding techniques.\footnote{An example where this issue arises was given in~\cite[Figure~\ref{p1.fig:why52}]{cite:LKS-cut1}.}
The ghost of $U$ will block the unusable part of the partner cluster, and we will know that we cannot expect to fill it up.

 Given a \semiregular matching $\mathcal N$, we call an involution $\mathfrak{d}:V(\mathcal N)\rightarrow V(\mathcal N)$ with the property that $\mathfrak{d}(S)=T$ for each $(S,T)\in\mathcal N$ a \index{general}{matching involution}\emph{matching involution}.

Assume Setting~\ref{commonsetting} and fix a matching involution $\mathfrak{b}$ for $\mathcal M_A\cup\mathcal M_B$. For any set $U\subset V(G)$ we then define \index{general}{ghost}\index{mathsymbols}{*GHOST@$\ghost(U)$} 
$$\ghost(U):=U\cup \mathfrak{b}\big(U\cap V(\M_A\cup\M_B)\big)\;.$$
Clearly, we have that $|\ghost(U)|\le 2|U|$, and $|\ghost(U)\cap S|=|\ghost(U)\cap T|$ for each $(S,T)\in\M_A\cup\M_B$.

The notion of a ghost extends to other \semiregular matchings. If $\mathcal N$ is a \semiregular matching and $\mathfrak{d}$ is a matching involution for $\mathcal N$ then we write $\ghost_{\mathfrak{d}}(U):=U\cup \mathfrak{d}\big(U\cap V(\mathcal N)\big)$.

\subsubsection{Embedding in Configuration~$\mathbf{(\diamond1)}$}\label{sssec:EmbedDiamon0Diamond1}
This subsection contains an easy observation that each tree of order $k$ is contained in $G$ if the graph $G$ contains Configuration~$\mathbf{(\diamond1)}$.

\begin{lemma}\label{lem:embed:greedy}
Let   $G$ be a graph, and let $A, B\subseteq V(G)$ be such that
$\mindeg(G[A,B])\ge k/2$, and $\mindeg(A)\ge k$. Then
each tree of order $k$ is contained in $G$.
\end{lemma}
\begin{proof}
Let $T\in\treeclass{k}$ have colour classes $X$ and $Y$, with $|X|\ge k/2 \ge
|Y|$.  By Fact~\ref{fact:treeshavemanyleaves}, for the set $W$ of those leaves of $T$ that lie in $X$, we have $|X\setminus W|\le k/2$. We embed
$T-W$ greedily in $G$, mapping $Y$ to $A$ and $X\setminus W$ to $B$. We then embed $W$ using the fact that $\mindeg(A)\ge k$.
\end{proof}

\subsubsection{Embedding in Configurations $\mathbf{(\diamond2)}$--$\mathbf{(\diamond5)}$}\label{sssec:EmbedMoreComplex}
In this section we show how to embed $T_\PARAMETERPASSING{T}{thm:main}$ in the presence of configurations $\mathbf{(\diamond2)}$--$\mathbf{(\diamond5)}$. As outlined in Section~\ref{ssec:EmbedOverview25} our main embedding lemma, Lemma~\ref{lem:conf2-5}, builds on Lemma~\ref{lem:blueShrubSuspend} which handles Stage~1 of the embedding, and Lemma~\ref{lem:embedC'endshrub} which handles Stage~2.

\begin{lemma}\label{lem:embedC'endshrub}
Assume we are in Setting~\ref{commonsetting}.
Suppose $L'',L',\HugeVertices'$ witness Preconfiguration
$\mathbf{(\clubsuit)}(\frac{10^5\Omega^*}{\eta})$.
 Let $(T,r)$ be a rooted tree of order at most $\gamma^2\nu k/6$. 
 Let $U\subset V(G)$ with $|U|+v(T)\le k$, and let $v\in \HugeVertices'\setminus U$. 
 Then there exists an $(r\hookrightarrow v,V(T)\hookrightarrow V(G)\setminus U)$-embedding of $(T,r)$.
\end{lemma}
\begin{proof}
We proceed by induction on the order of $T$. The base case
$v(T)\leq 2$ obviously holds. Let us assume Lemma~\ref{lem:embedC'endshrub} is true for all trees $T'$ with $v(T')<v(T)$.

Let $U_1:=\shadow_{\Gcapt}(U-\HugeVertices,\eta k/200)$, and $U_2:=\bigcup \{C\in \clusters\::\: |C\cap U|\ge \frac12|C|\}$. We have $|U_1|\le \frac{200\Omega^*}{\eta}k$ by Fact~\ref{fact:shadowbound}, and $|U_2|\le 2|U|$. Set 
\begin{align*}
L_{\smallatoms} &:=\BUG{L'}\cap\shadow_{\Gcapt}\left(\smallatoms,\frac{\eta k}{50}\right),\\
L_{\HugeVertices} &:=\BUG{L'}\cap\shadow_{\Gcapt}\left(\HugeVertices, |U\cap\HugeVertices|+\frac{\eta k}{50}\right),\text{ and }\\
L_{\clusters} &:=\BUG{L'}\cap\shadow_{\Gblack}\left(V({\Gblack}),(1+\frac{\eta}{50})k-|U\cap\HugeVertices|\right)\;.
\end{align*}
Observe that $L_{\clusters}\subseteq \bigcup \clusters$ and that since $\BUG{L'}\subseteq 
\largevertices{\frac9{10}\eta}{k}{\Gcapt}\setminus\HugeVertices$, we have
$$\BUG{L'}\subset V(\Gexp)\cup\smallatoms\cup L_{\HugeVertices}\cup L_{\smallatoms}\cup L_{\clusters}\;.$$
As by~\eqref{eq:clubsuitCOND2}, we have $\deg_G(v,\BUG{L'})\ge \frac{10^5\Omega^*k}{\eta}>5(|U\cup U_1 \cup U_2|+v(T)+\eta k)$,
 one of the following five cases must occur.

\medskip
\noindent{\underline{\sl Case I: $\deg_G(v,V(\Gexp)\setminus U)>v(T)+\eta k$.}}
Lemma~\ref{lem:embed:greyFOREST} gives an embedding of the forest $T-r$ (whose components are rooted at neighbours of $r$). The input sets/parameters of Lemma~\ref{lem:embed:greyFOREST} are $Q_\PARAMETERPASSING{L}{lem:embed:greyFOREST}:=1$, 
$\zeta_\PARAMETERPASSING{L}{lem:embed:greyFOREST}:=12\sqrt\gamma$, $U^*_\PARAMETERPASSING{L}{lem:embed:greyFOREST}:=(\neighbour_G(v)\cap V(\Gexp))\setminus U$, $U_\PARAMETERPASSING{L}{lem:embed:greyFOREST}:=U$, and $V_1=V_2:=V(\Gexp)$.

\smallskip
\noindent{\underline{\sl Case II: $\deg_G(v,\smallatoms\setminus U)>v(T)+\eta k$.}} Lemma~\ref{lem:embed:avoidingFOREST} gives an embedding of the forest $T-r$ (whose components are rooted at neighbours of $r$). The input sets/parameters of Lemma~\ref{lem:embed:avoidingFOREST} are $U^*_\PARAMETERPASSING{L}{lem:embed:avoidingFOREST}:=(\neighbour_G(v)\cap \smallatoms)\setminus U$, $U_\PARAMETERPASSING{L}{lem:embed:avoidingFOREST}:=U$ and $\eps_\PARAMETERPASSING{L}{lem:embed:avoidingFOREST}:=\eps'\leq\eta$.
Here, and below, we implicitly assume parameters of the same name to be the same, i.e.~$\gamma_\PARAMETERPASSING{L}{lem:embed:avoidingFOREST}:=\gamma$.

\smallskip
\noindent{\underline{\sl Case III: $\deg_G(v,L_{\smallatoms}\setminus (U\cup U_1))>v(T)+\eta k$.}} We only outline the strategy. Embed the children of~$r$ in $L_{\smallatoms}\setminus (U\cup U_1)$ using a map $\phi:\children_T(r)\rightarrow L_{\smallatoms}\setminus (U\cup U_1)$. By definition of $L_{\smallatoms}$, and $U_1$, we have $\deg_{\Gcapt}(\phi(w),\smallatoms\setminus U)> \frac{\eta k}{100}$ for each $w\in \children_T(r)$. Now, for every $w\in \children_T(r)$ we can proceed as in Case~II to extend this embedding to the rooted tree $\big(T(r,\uparrow w) ,w\big)$. That is, Case~III is ``Case~II with an extra step in the beginning''.


\smallskip
\noindent{\underline{\sl Case IV: $\deg_G(v,L_{\HugeVertices}\setminus U)>v(T)+\eta k$.}} 
We embed the children $\children_T(r)$ of $r$ in distinct vertices of $L_{\HugeVertices}\setminus U$. This is possible by the assumption of Case~IV. 

Now,~\eqref{eq:clubsuitCOND1} implies that $\mindeg_{\Gcapt}(L_{\HugeVertices},\HugeVertices')\ge |U\cap \HugeVertices|+ \frac{\eta k}{100}$. Consequently, $\mindeg_{\Gcapt}(L_{\HugeVertices},\HugeVertices'\setminus U)\ge \frac{\eta k}{100}$. Therefore, for each $w\in\children_T(r)$ embedded in  $L_{\HugeVertices}\setminus U$ we can find an embedding of $\children_T(w)$ in $\HugeVertices'\setminus U$ such that the images of grandchildren of $r$ are disjoint. We fix such an embedding. We can now apply induction. More specifically, for each grandchild $u$ of $r$ we embed the rooted tree $\big(T(r,\uparrow u), u\big)$ using Lemma~\ref{lem:embedC'endshrub} (employing induction) using the updated set $U$, to which the images of the newly embedded vertices were added.

\smallskip
\noindent{\underline{\sl Case V: $\deg_G(v,L_{\clusters}\setminus (U\cup  U_1\cup U_2))\ge v(T)$.}} Let $u_1,\ldots,u_\ell$ be the children of $r$. Let us consider arbitrary distinct neighbours $x_1,\ldots,x_\ell\in L_{\clusters}\setminus (U\cup U_1\cup U_2)$ of $v$. Let $T_i:=T(r,\uparrow u_i)$. We sequentially embed the rooted trees $(T_i,u_i)$, $i=1,\ldots,\ell$, writing $\phi$ for the embedding. In step $i$, consider the set $W_i:=\left(U\cup \bigcup_{j<i} \phi(T_j)\right)\setminus \HugeVertices$. Let $D_i\in\clusters$ be the cluster containing $x_i$. By the definitions of $L_{\clusters}$ and of $U_1$, $$\deg_{\Gblack}(x_i,V(\Gblack)\setminus W_i)\ge \frac{\eta k}{50}-\frac{\eta k}{200}\ge \frac{\eta k}{100}\;.$$ 
Fact~\ref{fact:clustersSeenByAvertex} yields a cluster $C_i\in \clusters$ for which $$\deg_{\Gblack}(x_i,C_i\setminus W_i)\ge \frac{\eta}{100}\cdot \frac{\gamma\clustersize}{2(\Omega^*)^2}>\frac{\gamma^2\clustersize}{2}+v(T)>\frac{12\epsilon'\clustersize}{\gamma^2}+v(T)\;.$$ In particular there 
is at least one edge from $E(\Gblack)$ between $C_i$ and $D_i$, and therefore, $(C_i,D_i)$ forms an $\epsilon'$-regular pair of density at least $\gamma^2$ in $\Gblack$. 
Map $u_i$ to $x_i$ and let $F_1,\ldots,F_m$ be the components of the forest $T_i-u_i$.
We now  sequentially embed the trees $F_j$ in the pair $(D_i,C_i)$ using Lemma~\ref{lem:embed:regular}, with  
$X_\PARAMETERPASSING{L}{lem:embed:regular}:=C_i\setminus (W_i\cup \bigcup_{q<j}\phi(F_q))$, $X^*_\PARAMETERPASSING{L}{lem:embed:regular}:=\neighbour_{\Gblack}(x_i,X_\PARAMETERPASSING{L}{lem:embed:regular})$, $Y_\PARAMETERPASSING{L}{lem:embed:regular}:=D_i\setminus (W_i\cup \{x_i\}\cup \bigcup_{q<j}\phi(F_q))$, $\eps_\PARAMETERPASSING{L}{lem:embed:regular}:=\eps'$,
and $\beta_\PARAMETERPASSING{L}{lem:embed:regular}:=\gamma^2/3$.
\end{proof}

We are now ready for the lemma that will handle Stage~1 in configurations $\mathbf{(\diamond2)}$--$\mathbf{(\diamond5)}$.

\begin{lemma}\label{lem:blueShrubSuspend}
Assume we are in Setting~\ref{commonsetting}, with $L'',L', \HugeVertices'$ witnessing
Preconfiguration $\mathbf{(\clubsuit)}(\Omega^\dagger)$ in $G$.
Let $U\subset V(G)\setminus \HugeVertices$
and let $(T,r)$ be  a rooted tree
with $v(T)\le k/2$ and $|U|+v(T)\le k$.
Suppose that each component of $T-r$ has order at most $\tau k$.
Let $x\in (L''\cap \YB)\setminus\bigcup_{i=0}^2\shadow_{\Gcapt}^{(i)}(\ghost(U),
\eta k/1000 )$.

Then there is a subtree $T'$ of $T$ with $r\in V(T')$
which has an $(r\hookrightarrow x, V(T')\setminus\{r\}\hookrightarrow
V(G)\setminus \HugeVertices)$-embedding $\phi$.
Further, the components of $T-T'$ can be partitioned into two (possibly empty) families $\mathcal C_1$ and $\mathcal C_2$, such that
 the
following two assertions
hold.
\begin{enumerate}[(a)]
 \item 
 \label{eq:boty}
If $\mathcal C_1\neq\emptyset$, then $\mindeg_{\Gcapt}(\phi(\parent (V(\bigcup
\mathcal C_1))),\HugeVertices')> k+\frac{\eta k}{100} -v(T')$, 
\item
\label{eq:botyzwei} 
 $\parent (V(\bigcup \mathcal C_2))\subseteq \{r\}$, and
$\deg_{\Gcapt}(x,\HugeVertices')>\frac k2+\frac{\eta k}{100}
-v(T'\cup\bigcup \mathcal C_1)$.
\end{enumerate}
\end{lemma}

\begin{proof}
Let $\mathcal C$ be the family of all components of $T-r$. We start by defining
$\mathcal C_2$. Then we have to distribute  $T-\bigcup\mathcal C_2$ between
$T'$ and $\mathcal C_1$. First, we find a set $\mathcal C_M\subseteq \mathcal
C\sm\mathcal C_2$ which fits into the matching $\mathcal M_A\cup \mathcal M_B$
(and thus will form a part of $T'$). Then, we consider the remaining components of
$\mathcal C\sm\mathcal C_2$: some of these will be embedded entirely,  of others
we only embed the root, and leave the rest for $\mathcal C_1$. Everything embedded
will become a part of $T'$.

Throughout the proof we write $\shadow$ for $\shadow_{\Gcapt}$.

\medskip

Set $\overline\Vgood:=\Vgood\setminus \shadow(\ghost(U),\frac{\eta k}{1000})$, and choose
 $\tilde{\mathcal C}\subseteq\mathcal C$ such that 
\begin{equation}\label{eq:HNY}
\deg_{\Gcapt}(x, \overline\Vgood)-\frac{\eta k}{30}
\ < \ 
\sum_{S\in \tilde{\mathcal C}} v(S)
\ \leq \
\max\left\{0,\deg_{\Gcapt}\left(x, \overline\Vgood\right)-\frac{\eta k}{40}\right\}.
\end{equation}
Set $\mathcal C_2:=\mathcal C\sm \tilde{\mathcal C}$.
Note that this choice clearly satisfies
the first part of~\eqref{eq:botyzwei}. Let us now verify the second part of~\eqref{eq:botyzwei}. For this, we calculate
\begin{align*}
\deg_{\Gcapt}(x, \HugeVertices') & \geq
\deg_{\Gcapt}(x,V_+\setminus L_\#)-\deg_{\Gcapt}\left(x,\shadow\left(\ghost(U),\frac{\eta k}{1000}\right)\right)
\\
&~~~~-\deg_{\Gcapt}\left(x,V_+\setminus\left(L_\#\cup \shadow\left(\ghost(U),\frac{\eta k}{1000}\right)\cup\HugeVertices\right)\right)\\
&~~~~-\deg_{\Gcapt}(x,\HugeVertices\setminus\HugeVertices')\\
\JUSTIFY{by~\eqref{eq:defYB}, $x\not\in\shadow^{(2)}(\ghost(U),\frac{\eta k}{1000})$,~\eqref{eq:HNY}, \eqref{eq:clubsuitCOND1}}
&\ge \left(\frac k2 +
\frac{\eta k}{20}\right)-\frac{\eta k}{1000}-\left(
\sum_{S\in \tilde{\mathcal C}} v(S)+\frac{\eta
k}{30}\right)-\frac{\eta k}{100}\\
& > \ \frac k2 - \sum_{S\in \tilde{\mathcal C}} v(S) + \frac{\eta k}{20}
\\
&\geq \frac k2
- v\left(T'\cup\bigcup\mathcal C_1\right)+ \frac{\eta k}{100},
\end{align*}
as needed for~\eqref{eq:botyzwei}. 

\smallskip

Now, set 
\begin{equation}\label{eq:trickyM}
\mathcal M :=\big\{(X_1,X_2)\in \mathcal
M_A\cup \mathcal M_B\::\: \deg_{\GD}(x,(X_1\cup X_2)\setminus\smallatoms)>
0\big\}\;.
\end{equation}
\begin{claim}\label{cl:Megdes}
We have $|V(\mathcal M)|\le \frac{4(\Omega^*)^2}{\gamma^2}k$.
\end{claim} 
\begin{proof}[Proof of Claim~\ref{cl:Megdes}]
Indeed, let $(X_1,X_2)\in\M$, i.e.~$(X_1,X_2)\in\M_A\cup\M_B$ with $\deg_{\GD}(x,(X_1\cup
X_2)\setminus\smallatoms)> 0$. Then, using Property~\ref{commonsetting3} of
Setting~\ref{commonsetting}, we see that there exists a cluster $C_{(X_1,X_2)}\in\clusters$
such that $\deg_{\GD}(x,C_{(X_1,X_2)})>0$, and either $X_1\subset
C_{(X_1,X_2)}$ or $X_2\subset C_{(X_1,X_2)}$. In particular, there exists
a dense spot $(A_{(X_1,X_2)},B_{(X_1,X_2)};F_{(X_1,X_2)})\in\DenseSpots$ such
that $x\in A_{(X_1,X_2)}$, and $X_1\subset B_{(X_1,X_2)}$ or $X_2\subset
B_{(X_1,X_2)}$. By Fact~\ref{fact:boundedlymanyspots}, there are at most
$\frac{\Omega^*}{\gamma}$ such dense spots, let $Z$ denote the union of all vertices contained in these spots. 
Fact~\ref{fact:sizedensespot} implies that $|Z|\le \frac{2(\Omega^*)^2}{\gamma^2}k$.
Thus $|V(\M)|\le 2 |V(\M)\cap Z|\le 2 |Z|\le
\frac{4(\Omega^*)^2}{\gamma^2}k$. 
\end{proof}
First we shall embed as
many components from $\tilde{\mathcal C}$ in $\mathcal M$ as possible. To this
end, consider an inclusion-maximal subset $\mathcal C_M$ of $\tilde{\mathcal
C}$ with
\begin{equation}\label{eq:trickySum}
\sum_{S\in\mathcal C_M}v(S)\leq \deg_{\Gcapt}(x,V(\mathcal
M))-\frac{\eta k}{1000}\;.
\end{equation}

We aim to utilize the degree of $x$ to $V(\M)$
to embed $\mathcal C_M$ in $V(\M)$, using the regularity method.
 
\begin{subremark}
This remark (which may as well be skipped at a first reading) is aimed at those readers who are wondering about a seeming inconsistency of the defining
formulas~\eqref{eq:trickyM} for $\M$, and~\eqref{eq:trickySum} for $\mathcal
C_M$. 
That is,~\eqref{eq:trickyM} involves the degree in $\GD$ and
excludes the set $\smallatoms$, while~\eqref{eq:trickySum} involves the degree in
$\Gcapt$. The setting in~\eqref{eq:trickyM} was chosen so that it allows us to
control the size of $\M$ in Claim~\ref{cl:Megdes}, crucially relying on Property~\ref{commonsetting3} of
Setting~\ref{commonsetting}. Such a
control is necessary to make the regularity method work. Indeed,  in each regular
pair there may be a small number of atypical vertices\footnote{The issue of atypicality itself could be avoided by preprocessing each pair $(S,T)$ of $\M_A\cup\M_B$
and making it super-regular. However this is not possible for atypicality with
respect to a given (but unknown in advance) subpair $(S',T')$.}, and we must
avoid these vertices when embedding the components by the regularity method.
 Thus without the control on $|\M|$ it might happen that the degree of $x$  
 is unusable because $x$ sees very small numbers of atypical vertices in an enormous number 
 of sets corresponding to $\M$-vertices. On the other hand, the edges $x$ sends 
 to $\smallatoms$ can be utilized by other techniques in later stages.
Once we have defined $\M$ we want to use the full degree to $V(\M)$ to ensure we
can embed the shrubs as balanced as possible into the $\M$-edges. This is necessary as otherwise part of the degree of $x$ might be unusable for the embedding, e.g.~because it might go to $\M$-vertices whose partners are already full.
\end{subremark}

 For each  $(C,D)\in
\mathcal M$ we choose a family $\mathcal C_{CD}\subseteq \mathcal C_M$ maximal such that 
\begin{equation}
\sum_{S\in \mathcal C_{CD}}v(S)\le
\deg_{\Gcapt}\left(x,(C\cup  D)\setminus \ghost
(U)\right)- \left(\frac\gamma{\Omega^*}\right)^3 |C|\;,\label{eq:hezkamistnost}
\end{equation}
and further, we
require  $\mathcal C_{CD}$ to be disjoint from families
$\mathcal C_{C'D'}$ defined in previous
steps. 
We claim that $\{\mathcal C_{CD}\}_{(C,D)\in\M}$
forms a partition of $\mathcal C_M$, i.e., all the elements of $\mathcal C_M$
are  used. Indeed, otherwise, by the maximality of $\mathcal
C_{CD}$ and since the components of $T-r$ have size at most $\tau k$, we obtain
\begin{align}\label{eq:ForM}
\begin{split}
\sum_{S\in \mathcal C_{CD}}v(S)&\ge \deg_{\Gcapt}(x,(C\cup  D)\setminus \ghost
(U))-\left(\frac\gamma{\Omega^*}\right)^3 |C|-\tau k\\
&\geByRef{eq:KONST} \deg_{\Gcapt}(x,(C\cup  D)\setminus \ghost
(U))-2\left(\frac\gamma{\Omega^*}\right)^3 |C|\;,
\end{split}
\end{align}
for each $(C,D)\in\M$.
Then we have
\begin{align*}
\sum_{S\in \mathcal
C_M}v(S)&>\sum_{(C,D)\in \mathcal
M}\sum_{S\in \mathcal C_{CD}} v(S)\\
\JUSTIFY{by~\eqref{eq:ForM}}&\ge \sum_{(C,D)\in
\mathcal M} \left(\deg_{\Gcapt}(x,(C\cup
D)\setminus \ghost(U))-2\left(\frac\gamma{\Omega^*}\right)^3 |C|\right)\\
\JUSTIFY{by Claim~\ref{cl:Megdes} and
Fact~\ref{fact:boundMatchingClusters}} &\ge \deg_{\Gcapt}\left(x,V(\mathcal
M)\setminus \ghost(U)\right)-2\left(\frac\gamma{\Omega^*}\right)^3 \cdot \frac{2(\Omega^*)^2}{\gamma^2}k\\ 
\JUSTIFY{as $x\not\in\shadow(\ghost(U))$}&\ge
\deg_{\Gcapt}(x,V(\mathcal M))- \frac{\eta k}{1000}\\
\JUSTIFY{by~\eqref{eq:trickySum}}&\ge\sum_{S\in \mathcal C_M}v(S)\;,
\end{align*}
a contradiction.

\def\LfCD{\PARAMETERPASSING{L}{lem:fillingCD}}
We use Lemma~\ref{lem:fillingCD} to
embed the components of $\mathcal C_{CD}$ in $(C\cup D)\setminus \ghost(U)$ with the
following setting: $C_\LfCD:=C$, $D_\LfCD:=D$, $U_\LfCD:=\ghost(U)$, $X^*_\LfCD:=(\neighbour_{\Gcapt}(x)\cap (C\cup D))\setminus U_\LfCD$,
 and $(T_i, r_i)$ are the rooted trees from $\mathcal C_{CD}$
with the roots being the neighbours of $r$. The constants in Lemma~\ref{lem:fillingCD} are
$\epsilon_\LfCD:=\epsilon'/8$, $\beta_\LfCD:=\sqrt{\epsilon '}$, and
$\ell_\LfCD:=|C|\ge \nu\pi k$. The rooted trees in $\mathcal C_{CD}$ are smaller than $\epsilon_\LfCD\ell_\LfCD$ by~\eqref{eq:KONST}. Condition~\eqref{eq:conFill} is satisfied by~\eqref{eq:hezkamistnost}, and since $(\gamma/{\Omega^*})^3\geq 50\sqrt{\eps '}$.

\medskip

It remains to deal with the components of $\tilde{\mathcal C}\setminus\mathcal
C_M$.
In the sequel we shall assume that $\tilde{\mathcal C}\setminus\mathcal
C_M\neq \emptyset$ (otherwise skip this step and go directly to the definition of $T'$ and $\mathcal C_1$, with $p=0$). Thus, by our choice of $\mathcal C_M$, we have
\begin{equation}\label{eq:maxT1'}
\sum_{S\in\mathcal C_M}v(S)\ge
\deg_{\Gcapt}(x,V(\mathcal
M))-\frac{\eta k}{900}\;. \end{equation}

Let $T_1,T_2,\ldots,T_p$ be the trees of $\tilde{\mathcal
C}\setminus\mathcal C_M$ rooted at the vertices $r_i\in\children(r)\cap
V(T_i)$ neighbouring~$r$. We shall sequentially extend our embedding of $\mathcal C_M$ to subtrees $T'_i\subset T_i$. Let $U_i\subset V(G)$ be the union of the images of $\bigcup\mathcal C_M\cup\{r\}$ and of $T'_1,\ldots,T'_i$ under this embedding.

Suppose that we have embedded the trees $T'_1,\ldots,T'_i$ for some
$i=0,1,\ldots,p-1$. We claim that at least one of the following holds.
\begin{enumerate}
  \item[{\bf (V1)}]
  $\deg_{\Gcapt}(x,V(\Gexp)\setminus
  (U\cup U_i))\ge \frac{\eta k}{1000}$,
  \item[{\bf
  (V2)}]$\deg_{\Gcapt}(x,\smallatoms\setminus
  (U\cup U_i))\ge \frac{\eta k}{1000}$, or
  \item[{\bf
  (V3)}]$\deg_{\Gcapt}(x,L'\setminus
  (V(\Gexp)\cup\smallatoms\cup U\cup U_i\cup \shadow(\ghost(U),\frac{\eta k}{1000})))\ge
  \frac{\eta k}{1000}$.
\end{enumerate}
Indeed, suppose that none of {\bf(V1)}--{\bf (V3)} holds. Then, first note that since $U\subseteq \ghost(U)$ and since $x\notin \shadow(\ghost(U),\eta k/1000)$, we have
\begin{equation}\label{whatxsendstoU}
\deg_{\Gcapt}(x,U)\leq \frac{\eta k}{1000}.
\end{equation}

Also,
\begin{equation}\label{MAMBMatoms}
 \deg_{\GD}(x,V(\M_A\cup \M_B)) \leq\deg_{\GD}(x,V(\M)\cup\smallatoms).
\end{equation}

Thus,
\begin{align*}
\deg_{\Gcapt}&\left(x,\Vgood\setminus\shadow(\ghost(U), \frac{\eta k}{1000})\right)\\
\JUSTIFY{by \eqref{whatxsendstoU} and \eqref{MAMBMatoms}, def of $\Vgood$}&\le 
\deg_{\Gcapt}\left(x,\left(V(\M)\cup V(\Gexp)\cup \smallatoms\cup L'\right)\setminus
(U\cup \shadow(\ghost(U), \frac{\eta k}{1000})\right)\\
&~~~~+\deg_{\Gcapt}\big(x,\largevertices{\frac9{10}\eta}{k}{\Gcapt}\setminus(\HugeVertices\cup
L')\big)+\frac{\eta k}{1000}\\ 
\JUSTIFY{by~\eqref{eq:clubsuitCOND3}}&\le
\deg_{\Gcapt}\left(x,\left( V(\Gexp)\cup
\smallatoms\cup L'\right)\setminus
(V(\M)\cup U\cup \shadow(\ghost(U),\frac{\eta k}{1000}))\right)
\\
&~~+
\deg_{\Gcapt}\left(x,V(\M)\right)+\frac{\eta k}{100}+\frac{\eta k}{1000}
\\
\JUSTIFY{by $\neg{\bf(V1)}$, $\neg{\bf(V2)}$, $\neg{\bf(V3)}$,
by~\eqref{eq:maxT1'}} &\le
3\cdot\frac{\eta
k}{1000}+\sum_{j=1}^iv(T'_j)
+
\sum_{S\in\mathcal C_M}v(S)+\frac{\eta k}{900}
+\frac{\eta k}{100}+\frac{\eta k}{1000}
\\
&< \sum_{S\in\tilde{\mathcal
C}}v(S)+\frac{\eta k}{40}\;,
\end{align*}
a contradiction to~\eqref{eq:HNY}.

In cases {\bf(V1)}--{\bf(V2)} we shall embed the entire tree
$T'_{i+1}:=T_{i+1}$. In case {\bf(V3)} we either embed the entire
tree $T'_{i+1}:=T_{i+1}$, or embed only one vertex $T'_{i+1}:=r_{i+1}$ (that will only happen in case  {\bf (V3c)}). In the latter case,  we keep track of the components of $T_{i+1}-r_{i+1}$ in the set $\mathcal C_{1,i+1}$  (we tacitly assume we set $\mathcal C_{1,i+1}:=\emptyset$ in all cases other than {\bf (V3c)}). The union of the sets $\mathcal C_{1,i}$ will later form the set $\mathcal C_1$. Let us go through
our three cases in detail.

\smallskip

In case {\bf(V1)} we embed $T_{i+1}$ rooted at $r_{i+1}$
using Lemma~\ref{lem:embed:greyFOREST}
\def\Leg{\PARAMETERPASSING{L}{lem:embed:greyFOREST}}
 for one tree (i.e.~$\ell_\Leg:= 1$) with the following sets/parameters:
$H_\Leg:=\Gexp$,
$U_\Leg:=U\cup U_i$, $U^*_\Leg:=\neighbour_{\Gcapt}(x)\cap
(V(\Gexp)\setminus(U\cup U_i))$, $V_1=V_2:=V(\Gexp)$, $Q_\Leg:=1$, $\zeta_\Leg:=\rho$, and 
$\gamma_\Leg:=\gamma$. Note that $|U\cup
U_i|< k$, that $|\neighbour_{\Gcapt}(x)\cap
(V(\Gexp)\setminus (U\cup U_i))|\ge \eta k/1000>32\gamma k/\rho+1$,  
that $v(T_{i+1})\le
\tau k<\rho k/4$ and that $128\gamma<\rho^2$.

\smallskip

In case {\bf(V2)} we embed $T_{i+1}$ rooted at $r_{i+1}$
using Lemma~\ref{lem:embed:avoidingFOREST} 
\def\Lavoid{\PARAMETERPASSING{L}{lem:embed:avoidingFOREST}} for one tree (i.e.~$\ell_\Lavoid:= 1$) with the following setting:
$H_\Lavoid:= G-\HugeVertices$, $\smallatoms_\Lavoid:=\smallatoms$,
$U_\Lavoid:=U\cup U_i$, $U^*_\Lavoid:=\neighbour_{\Gcapt}(x)\cap (\smallatoms\setminus(U\cup U_i))$, $\Lambda_\Lavoid:=\Lambda$, $\gamma_\Lavoid:=\gamma$, and
$\epsilon_\Lavoid:=\epsilon'$. Note that $|U\cup U_i|\le k<\Lambda
k$, that $|\neighbour_{\Gcapt}(x)\cap (\smallatoms\setminus (U\cup U_i))|\ge \eta
k/1000 >2\eps' k$, and that $v(T_{i+1})\le \tau k<\gamma k/2$. 

\smallskip

We commence case~{\bf(V3)} with an auxiliary claim.
\begin{claim}\label{cl:TEC}
There exists a cluster $C_0\in\clusters$ such that $$\deg_{\GD}\left(x,(C_0\cap L')\setminus \left(V(\Gexp)\cup U\cup U_i\cup
  \shadow\left(\ghost(U),\frac{\eta k}{1000}\right)\right)\right)\ge \frac{\eps'}{\gamma^2} \clustersize\;.$$
\end{claim} 
\begin{proof}[Proof of Claim~\ref{cl:TEC}]
Observe that
$L'\setminus
(V(\Gexp)\cup\smallatoms\cup \HugeVertices\cup U\cup U_i)\subset \bigcup
\clusters$. Furthermore, since $x\in\bigcup\clusters$, we have \[E_{\Gcapt}\left[x,L'\setminus \left(V(\Gexp)\cup\smallatoms\cup
U\cup U_i\cup \shadow\left(\ghost(U),\frac{\eta k}{1000}\right)\right)\right]\subset E(\GD)\;.\] By
Fact~\ref{fact:clustersSeenByAvertex}, there are at most $\frac{2(\Omega^*)^2k}{\gamma^2\clustersize}$ clusters $C\in\clusters$ such that
$\deg_{\GD}(x,C)>0$. Using the assumption~{\bf(V3)}, there exists a
cluster $C_0\in\clusters$ such that 
\begin{align*}
\deg_{\GD}\left(x,(C_0\cap L')\setminus
(V(\Gexp)\cup U\cup U_i\cup \shadow(\ghost(U),\frac{\eta k}{1000}))\right)&\ge
\frac{\eta
k}{1000}\cdot\frac{\gamma^2 \clustersize}{2(\Omega^*)^2k}\\
&\overset{\eqref{eq:KONST}}\ge \frac{\eps'}{\gamma^2}\clustersize\;,
\end{align*} 
as desired.
\end{proof}
Let us take a cluster $C_0$ from Claim~\ref{cl:TEC}. We embed the root $r_{i+1}$ 
of $T_{i+1}$ in an arbitrary neighbour $y$ of $x$ in $(C_0\cap L')\setminus
(V(\Gexp)\cup U\cup U_i\cup \shadow(\ghost(U),\frac{\eta k}{1000}))$.

Let $H\subset G$ be the subgraph of $G$ consisting of all edges in dense spots
$\DenseSpots$, and all edges incident with~$\HugeVertices'$. 
As  by~\eqref{eq:clubsuitCOND1}, $y$ has at most $\eta k/100$ neighbours in $\HugeVertices\setminus\HugeVertices'$, and since $y\in L'\subseteq
 \largevertices{9\eta/10}{k}{\Gcapt}$ and $y\notin\shadow(U,\frac{\eta
k}{100})$, we find that
\begin{align*}
\deg_{H}\left(y,V(G)\setminus
((U\cup U_i)\cup(\HugeVertices\setminus\HugeVertices'))\right)
& \ge\left(1+\frac{9\eta}{10}\right)k-\frac{\eta k}{1000}-|U_i|-\frac{\eta
k}{100}\\
& > k-|U_i|+\frac{\eta k}{2}\;.
\end{align*}
Therefore, one of the three following subcases must occur. (Recall that $y\not\in\smallatoms$ as $y\in C_0\in\clusters$.)
\begin{enumerate}
  \item[{\bf (V3a)}] $\deg_{\Gcapt}(y,\smallatoms\setminus (U\cup U_i))\ge \frac{\eta
  k}{6}$, 
  \item[{\bf (V3b)}] $\deg_{\Gblack}(y,\bigcup\clusters
 \setminus (U\cup U_{i}))\ge\frac{\eta k}{6}$, or
  \item[{\bf (V3c)}] $\deg_{\Gcapt}(y,\HugeVertices')\ge k-|U_i|+\frac{\eta
  k}{6}$.
\end{enumerate}
In case~{\bf (V3a)} we embed the components  of $T_{i+1}-r_{i+1}$ (as trees
rooted at the children of $r_{i+1}$) using the same technique as in
case~{\bf (V2)}, with Lemma~\ref{lem:embed:avoidingFOREST}.

\smallskip

\def\LER{\PARAMETERPASSING{L}{lem:embed:regular}}
In~{\bf (V3b)} we embed the components  of $T_{i+1}-r_{i+1}$ (as trees
rooted at the children of $r_{i+1}$). By
Fact~\ref{fact:clustersSeenByAvertex} there exists a cluster $D\in \clusters$ such that 
\begin{equation}\label{V3bsizeX*}
\deg_{\Gblack}(y,D\setminus (U\cup U_i))\ge
\frac{\eta k}{6}\cdot
\frac{\gamma^2\clustersize}{2(\Omega^*)^2k}>\frac{\gamma^2}2\clustersize.
\end{equation}
We use Lemma~\ref{lem:embed:regular} with input $\epsilon_\LER:=\epsilon'$, 
$\beta_\LER:=\gamma^2$, $C_\LER:=D$, $D_\LER:=C_0$, $X^*_\LER=X_\LER:=D\setminus (U\cup U_i)$ and $Y_\LER:=C_0\setminus (U\cup U_i\cup\{y\})$ 
to embed the tree $T_{i+1}$ into
the pair $(C_0,D)$, by embedding the components of $T_{i+1}-r_{i+1}$ one after
the other. The numerical conditions of Lemma~\ref{lem:embed:regular} hold because of Claim~\eqref{cl:TEC} and because of~\eqref{V3bsizeX*}.

\smallskip

In case~{\bf (V3c)} we set $T'_{i+1}:=r_{i+1}$ and define $\mathcal C_{1,i+1}$
as  set of all components of $T_{i+1}-r_{i+1}$. Then $\phi(\parent(\bigcup
\mathcal C_{1,i+1} )\cap V(T'_{i+1}))=\{y\}$ and
\begin{equation}\label{uiuiuiui}
 \deg_{\Gcapt}(y,\HugeVertices')\geq  k-|U_i|+\frac{\eta k}{6}\;.
\end{equation}

When all the trees $T_1,\ldots,T_p$ are processed, we define $T':=\{r\}\cup \bigcup\mathcal C_M\cup \bigcup_{i=1}^pT'_i$, and set $\mathcal C_1:=\bigcup_{i=1}^p\mathcal C_{1,i}$.
Thus also~\eqref{eq:boty} is satisfied by~\eqref{uiuiuiui} for $i=p$, since $|T'|=|U_p|$.
This finishes the proof of the lemma.
\end{proof}

It turns out that our techniques for embedding a tree $T\in\treeclass{k}$ for 
Configurations~$\mathbf{(\diamond2)}$--$\mathbf{(\diamond5)}$ are very similar.
In Lemma~\ref{lem:conf2-5} below we resolve these tasks together. The proof of
Lemma~\ref{lem:conf2-5} follows the same basic strategy for each of the
configurations~$\mathbf{(\diamond2)}$--$\mathbf{(\diamond5)}$ and differs only
in the elementary procedures of embedding shrubs of $T$.

\begin{lemma}\label{lem:conf2-5}
Suppose that we are in Setting~\ref{commonsetting}, and one of the following
configurations can be found in $G$:
\begin{enumerate}[(a)]
  \item Configuration~$\mathbf{(\diamond2)}\left((\Omega^*)^2,
  5(\Omega^*)^9, \rho^3 \right)$,
  \item Configuration~$\mathbf{(\diamond3)}\left((\Omega^*)^2,
  5(\Omega^*)^9, \gamma/2, \gamma^3/100\right)$,
  \item Configuration~$\mathbf{(\diamond4)}\left((\Omega^*)^2,
  5(\Omega^*)^9, \gamma/2, \gamma^4/100\right)$, or
  \item Configuration~$\mathbf{(\diamond5)}\left((\Omega^*)^2,
  5(\Omega^*)^9,\epsilon', 2/(\Omega^*)^3,\frac{1}{(\Omega^*)^5}\right)$,
\end{enumerate}
Let $(T,r)$ be a rooted tree of order $k$ with a $(\tau k)$-fine partition
$(W_A,W_B,\shrubA,\shrubB)$. Then $T\subset G$.
\end{lemma}
\begin{proof} First observe that
each of the configurations given by (a)--(d) contains two sets $\HugeVertices''\subseteq \HugeVertices$ and $V_1\subseteq V(G)\sm\HugeVertices$ with
\begin{align}
\label{eq:sumC1}
\mindeg_{\Gcapt}(\HugeVertices'', V_1)&\ge 5(\Omega^*)^9k\; ,\\
\label{eq:sumC2}
\mindeg_{\Gcapt}(V_1,\HugeVertices'')&\ge \epsilon' k\; .
\end{align}

For any seed $z\in W_A\cup W_B$ we
define $T(z)$ as the forest consisting of all components of $T-(W_A\cup
W_B)$ that contain children of $z$. 
Throughout the proof, we write $\phi$ for the current partial embedding of $T$ into $G$.

\paragraph{Overview of the embedding procedure.} As outlined in Section~\ref{ssec:EmbedOverview25} the embedding scheme is the same for Configurations~$\mathbf{(\diamond 2)}$--$\mathbf{(\diamond 5)}$. The embedding $\phi$ is
defined in two stages. In Stage~1, we embed the seeds $W_A\cup
W_B$, all the internal shrubs, all the end shrubs of $\shrubA$, and a part\footnote{in the sense that individual shrubs $\shrubB$ may be embedded only in part} of the end shrubs of $\shrubB$. In Stage~2 we embed the rest of $\shrubB$. Which part of $\shrubB$ is embedded in Stage~1 and which part in Stage~2 will be determined during Stage~1. We first give a rough outline of  both stages listing some conditions which
we require to be met, and then we describe each of the stages in detail.

Stage~1 is defined in $|W_A\cup\{r\}|$ steps. First we map $r$ to any vertex
in $\HugeVertices''$. Then in each step we pick a vertex $x\in
W_A$ for which the embedding $\phi$ has already been defined but such that $\phi$ is
not yet defined for any of the children of $x$. In this step we embed 
$T(x)$, together with all the children and grandchildren of $x$ in the \kknnaagg which contains $x$. For each child
$y\in W_B\cap \children(x)$, Lemma~\ref{lem:blueShrubSuspend} determines a subforest $T'(y)\subset T(y)$ 
which is embedded in Stage~1, and sets $\mathcal C_1 (y)$ and $\mathcal C_2 (y)$, which will be embedded in Stage~2. 

The
embedding in each step of Stage~1 will be defined so that the following properties hold. 
\begin{enumerate}[(*1)]
\item All vertices from $W_A$ are mapped to $\HugeVertices''$.
\item All vertices except for $W_A$ are mapped to
$V(G)\setminus \HugeVertices$.
\item For each $y\in W_B$, for each $v\in\parent (V(\bigcup \mathcal C_1 (y)))$ we have that $$\deg_G(\phi(v),\HugeVertices')\ge
k + \tfrac {\eta k}{100} - v(T'(y))\;.$$
  \item For each $y\in W_B$, for each $v\in\parent (V(\bigcup \mathcal C_2 (y)))$ we have that $$\deg_G(\phi(v),\HugeVertices')\ge
\tfrac k2 + \tfrac {\eta k}{100} - v\left(T'(y)\cup \bigcup\mathcal C_1(y)\right)\;.$$
\end{enumerate}

In Stage~2, we shall utilize properties~(*3) and~(*4) to embed 
$T_B^*:=\bigcup\shrubB-\bigcup_{y\in W_B} T'(y)$. Stage~2 is substantially simpler than Stage~1; this is due to the fact that $T_B^*$ consists only of end shrubs. 

\paragraph{The embedding step of Stage~1.} The embedding
step is the same for
Configurations~$\mathbf{(\diamond2)}$--$\mathbf{(\diamond5)}$, except for
the embedding of internal shrubs. The order of the embedding steps is illustrated in Figure~\ref{fig:L825}.
\begin{figure}[t]
\centering 
\includegraphics{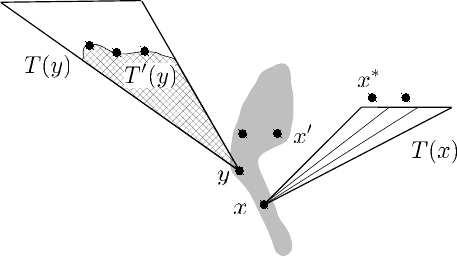}
\caption[Stage~1 of embedding in proof of Lemma~\ref{lem:conf2-5}]{Stage~1 of the embedding in the proof of Lemma~\ref{lem:conf2-5}. Starting from an already embedded seed $x\in W_A$ we extend the embedding (in this order) to\\
(1) all the children $y\in W_B$ of $x$ in the same \kknnaagg (in grey),\\
(2) a part $T'(y)$ of the forest $T(y)$,\\
(3) all the grandchildren $x'\in W_A$ of $x$ in the same \kknnaaggNOSPACE,\\
(4) the forest $T(x)$ together with the bordering cut-vertices $x^*\in W_A$.
}
\label{fig:L825}
\end{figure}

In each step we select a seed $x\in W_A$ already embedded in $G$ but such that none of $\children(x)$ are embedded. By (*1), or by the
choice of $\phi(r)$, we have $\phi(x)\in \HugeVertices''$. So by~\eqref{eq:sumC1} we have
\begin{equation}\label{elventarron}
 \deg_{\Gcapt}(\phi(x), V_1\sm U)\ge 5(\Omega^*)^9 k -k.
\end{equation}

First, we embed successively in $|W_B\cap \children (x)|$ steps the seeds
$y\in W_B\cap \children (x)$ together with  components $T'(y)\subset T(y)$ which will be determined on the way.
Suppose that in a certain step we are to embed $y\in W_B\cap \children (x)$
and the (to be determined) tree $T'(y)$. Let
$$F:=\bigcup_{i=0}^2\shadow^{(i)}_{\Gcapt-\HugeVertices}\left(\ghost(U), \frac{\eta
k}{10^5}\right)\;,$$ where $U$ is the set of vertices used by the embedding $\phi$ in previous steps. Since $|U|\le k$,  Fact~\ref{fact:shadowbound} gives us that $|F|\le \frac {10^{10}(\Omega^*)^2}{\eta^2}k$. We embed $y$
anywhere in $(\neighbour_G(\phi(x))\cap V_1)\setminus F$, cf.~\eqref{eq:sumC1}. Note that then (*2) holds for $y$. We use Lemma~\ref{lem:blueShrubSuspend} in
order to embed $T'(y)\subset T(y)$ (the subtree $T'(y)$ is determined by
Lemma~\ref{lem:blueShrubSuspend}). Lemma~\ref{lem:blueShrubSuspend} ensures
that~(*3) and~(*4) hold and that we have $\phi(V(T'(y)))\subseteq V(G)\setminus
\HugeVertices$. 

Also, we map the vertices $x'\in W_A\cap \children(y)$ to $\HugeVertices''\setminus U$. To justify this step, employing~(*2), it is enough to prove that 
\begin{equation}\label{eq:WNTP}
\deg(\phi(y),\HugeVertices'')\ge |W_A|\;.
\end{equation}
Indeed, on the one hand, we have $|W_A|\le 336/\tau$ by Definition~\ref{ellfine}\eqref{few}. On
the other hand, we have that $\phi(y)\in V_1$, and thus~\eqref{eq:sumC2} applies.
 We can thus embed $x'$ as planned, ensuring (*1), and finishing the step for $y$.

Next, we sequentially embed the components $\tilde T$ of $T(x)$. In
the following, we describe such an embedding procedure only for an internal shrub
$\tilde T$, with $x^*$ denoting the other neighbour of
$\tilde T$ in $W_A$ (cf.~(*1)). The case when $\tilde T$ is an end shrub is analogous: actually it is even easier as we do not
have to worry about placing $x^*$ well.
The actual embedding of $\tilde T$ together with $x^*$ depends on the configuration we are in. We shall slightly abuse notation by letting $U$ now denote everything embedded before the tree $\tilde T$.

\smallskip

For Configuration~$\mathbf{(\diamond2)}$, we use
Lemma~\ref{lem:embed:greyFOREST} for one tree, namely $\tilde T-x^*$, using the following setting:
$Q_\PARAMETERPASSING{L}{lem:embed:greyFOREST}:= 1,
\gamma_\PARAMETERPASSING{L}{lem:embed:greyFOREST}:= \gamma,
\zeta_\PARAMETERPASSING{L}{lem:embed:greyFOREST}:= \rho^3,
H_\PARAMETERPASSING{L}{lem:embed:greyFOREST}:= \Gexp$,  $U_\PARAMETERPASSING{L}{lem:embed:greyFOREST}:= U$, and 
$U^*_\PARAMETERPASSING{L}{lem:embed:greyFOREST}:=(\neighbour_{\Gcapt}(\phi(x))\cap
V_1)\setminus U$ (this last set is large enough by~\eqref{elventarron}). The child of
$x$ gets embedded in $(\neighbour_{\Gcapt}(\phi(x))\cap V_1)\setminus U$, the vertices at odd distance from $x$ get embedded in $V_1$, and the vertices at even distance from $x$ get embedded in $V_2$. In particular, $\parent_T(x^*)$, the parent of~$x^*$, gets embedded in $V_1$. After this, we accomodate $x^*$ in a vertex in $\HugeVertices''\setminus U$ which is adjacent to $\phi(\parent_T(x^*))$. This is possible by the same reasoning as in~\eqref{eq:WNTP}.

\smallskip

For Configuration~$\mathbf{(\diamond3)}$, we use
Lemma~\ref{lem:HE3} to embed $\tilde T$ with the setting
$\gamma_\PARAMETERPASSING{L}{lem:HE3}:= \gamma,
\delta_\PARAMETERPASSING{L}{lem:HE3}:= \gamma^3/100,
U_\PARAMETERPASSING{L}{lem:HE3}:= U$ and
$U^*_\PARAMETERPASSING{L}{lem:HE3}:=(\neighbour_{\Gcapt}(\phi(x))\cap V_1)\setminus U$ (this last set is large enough by~\eqref{elventarron}). 
Then the child of $x$ gets embedded in $(\neighbour_{\Gcapt}(\phi(x))\cap V_1)\setminus U$, vertices of $\tilde T$ of odd distance to $x$ (i.e.~of even distance to the root of $\tilde T$) get embedded in $V_1\setminus U$, and vertices of even distance get embedded in $V_2\setminus U$. We extend the embedding by mapping $x^*$ to a suitable vertex in $\HugeVertices''\setminus U$ adjacent to $\phi(\parent_T(x^*))$ in the same way as above.

\smallskip

For Configuration~$\mathbf{(\diamond4)}$, we use
Lemma~\ref{lem:HE4} to embed $\tilde T$ with the setting
$\gamma_\PARAMETERPASSING{L}{lem:HE4}:= \gamma,
\delta_\PARAMETERPASSING{L}{lem:HE4}:= \gamma^4/100,
U_\PARAMETERPASSING{L}{lem:HE4}:= U$ and
$U^*_\PARAMETERPASSING{L}{lem:HE4}:=(\neighbour_{\Gcapt}(\phi(x))\cap V_1)\setminus U$  (this last set is large enough by~\eqref{elventarron}). 
The fruit  $r'_\PARAMETERPASSING{L}{lem:HE4}$ in the lemma is chosen as $\parent_T(x^*)$. Note that this is indeed a fruit (in $\tilde T$) because of Definition~\ref{ellfine}~\eqref{short}.
Then the child of $x$ gets embedded in
$(\neighbour_{\Gcapt}(\phi(x))\cap V_1)\setminus U$, the vertex $r'_\PARAMETERPASSING{L}{lem:HE4}=\parent_T(x^*)$ gets embedded in $V_1\setminus U$, and the rest of $\tilde T$ gets embedded in $(\smallatoms'\cup V_2)\setminus U$. This allows us to extend the embedding to $x^*$ as above.

\smallskip

In Configuration~$\mathbf{(\diamond5)}$, let $\mathbf W\subset \clusters$  denote the set of those clusters, which have at least an $\frac{1}{2(\Omega^*)^5}$-fraction of their vertices contained in the set $U':=U\cup \shadow_{\Gblack}(U,k/(\Omega^*)^3)$. We get from Fact~\ref{fact:shadowbound} that $|U'|\le 2(\Omega^*)^4 k$, and consequently $|U'\cup \bigcup \mathbf W|\le 4(\Omega^*)^9 k$. By~\eqref{elventarron} we can find a vertex $v\in (\neighbour_G(\phi(x))\cap V_1)\setminus (U'\cup \bigcup \mathbf W)$. 

We use the fact that $v\not \in \shadow_{\Gblack}(U,k/(\Omega^*)^3)$ together with inequality~\eqref{confi5last}
to see that $\deg_{\Gblack}(v,V(\Gblack)\sm U)\geq k/(\Omega^*)^3$.
Now, since there are only boundedly many clusters seen from $v$ (cf. Fact~\ref{fact:clustersSeenByAvertex}),  there must be a cluster $D\in \clusters$ such that 
\begin{equation}\label{eq:extendfromx}
 \deg_{\Gblack}(v,D\setminus U)\ge \frac{\gamma^2}{2\cdot(\Omega^*)^5}|D|\ge\gamma^3 |D|\;.
\end{equation}
 Let $C$ be the cluster containing $v$. We have $|(C\cap V_1)\setminus U|\ge \frac{1}{2(\Omega^*)^5}|C|\ge\gamma^3 |C|$ because of~\eqref{eq:diamond5P4} and since $C\notin\mathbf W$. Thus, by Fact~\ref{fact:BigSubpairsInRegularPairs}, $\big((C\cap V_1)\setminus U,D\setminus U\big)$ is an $2\epsilon'/\gamma^3$-regular pair of density at least $\gamma^2/2$. We can therefore embed $\tilde T$ in this pair using the regularity method. Moreover, by~\eqref{eq:extendfromx}, we can do so by mapping the child $z$ of $x$ to $v$. Thus the parent of $x^*$ (lying at even distance to $z$) will be embedded in $(C\cap V_1)\setminus U$. We can then extend our embedding to $x^*$ as above.

\smallskip

This finishes our embedding of $T(x)$. Note that
in all cases we have $\phi(x^*)\in \HugeVertices''$ and $\phi(V(\tilde T))\subseteq V(G)\setminus \HugeVertices$, as required by~(*1) and~(*2).

\paragraph{The embedding steps of Stage~2.}
For $i=1,2$, set $Z_i:=\bigcup_{y\in W_B} \children(T'(y))\cap \bigcup\mathcal C_i(y)$.

First, we embed all the vertices $z\in Z_2$ in $\HugeVertices'$. By~(*2), until now, only vertices of $W_A\cup Z_2$ are mapped to $\HugeVertices'$, and using~(*4) and the properties~\eqref{few}, \eqref{Bend} and~\eqref{Bsmall} of Definiton~\ref{ellfine}, we see that
\begin{align*}
\deg_G(\phi(\parent(z)),\HugeVertices')
&\geq
\frac{\eta k}{100} + \frac k2 - \sum_{y\in W_B} v\left(T'(y)\cup \bigcup\mathcal C_1(y)\right)\\
&>|W_A|+ |Z_2|\;.
\end{align*}

So there is space for the vertex $z$ in $\HugeVertices'\cap
\phi(\neighbour_G(\parent(z)))$. 

Next, we embed all the vertices $z\in Z_1$ in $\HugeVertices'$. By~(*2), until now only vertices of $W_A\cup Z_2\cup Z_1$ are mapped to $\HugeVertices'$, and by~(*3) we have, similarly as above,
 $$\deg_G(\phi(\parent(z)),
\HugeVertices')>|W_A|+|Z_2|+|Z_1|\;.$$ So $z$ can be embedded in $\HugeVertices'\cap
\neighbour_G(\phi(\parent(z)))$ as planned.

Finally,
 for $z\in Z_1\cup Z_2$, denote by $T_z$ the component of $\mathcal C_1\cup\mathcal C_2$
that contains $z$.  We use
Lemma~\ref{lem:embedC'endshrub} to embed the rest of the
rooted tree $(T_z, z)$. (Note that our parameters work because of~\eqref{eq:KONST}.) 
Once all rooted trees $(T_z, z)$ with $z\in Z_1\cup Z_2$ have been processed, we have finished Stage~2 and thus the proof of the lemma.
\end{proof} 

\subsubsection{Embedding in Configurations $\mathbf{(\diamond6)}$--$\mathbf{(\diamond10)}$}\label{sssec:OrderedSkeleton}
We follow the schemes outlined in Sections~\ref{ssec:EmbedOverview67}, \ref{ssec:EmbedOverview8}, \ref{ssec:EmbedOverview9}, and~\ref{ssec:EmbedOverview10}.

Embedding a tree $T_\PARAMETERPASSING{T}{thm:main}\in\treeclass{k}$ using Configurations~$\mathbf{(\diamond6)}$, $\mathbf{(\diamond7)}$, or
$\mathbf{(\diamond8)}$ has two parts: first the internal part of
$T_\PARAMETERPASSING{T}{thm:main}$ is embedded, and then this partial embedding is extended to end shrubs of $T_\PARAMETERPASSING{T}{thm:main}$ as well. Lemma~\ref{lem:embed:skeleton67} (for configurations $\mathbf{(\diamond6)}$ and $\mathbf{(\diamond7)}$) and Lemma~\ref{lem:embed:skeleton8} (for configuration $\mathbf{(\diamond8)}$) are used for the former part, and Lemmas~\ref{lem:embed:heart1} and~\ref{lem:embed:heart2} (depending on whether we have $\mathbf{(\heartsuit1)}$ or $\mathbf{(\heartsuit2)}$) for the latter. Lemma~\ref{lem:embed:total68} then puts these two pieces together.

Embedding using Configurations~$\mathbf{(\diamond9)}$ and
$\mathbf{(\diamond10)}$ is resolved in
Lemmas~\ref{lem:embed9} and~\ref{lem:embed10}, respectively.

\begin{lemma}\label{lem:embed:skeleton67}
Suppose we are in Setting~\ref{commonsetting} and~\ref{settingsplitting}, and we
have one of the following two configurations:
\begin{itemize}
 \item Configuration~$\mathbf{(\diamond6)}(\delta_6, \tilde\epsilon,d',\mu,1,0)$, or
 \item
 Configuration~$\mathbf{(\diamond7)}(\delta_7,\frac{\eta \gamma}{400},\tilde\epsilon,d',\mu,1,0)$,
\end{itemize}
with $10^5\sqrt\gamma(\Omega^*)^2\le \delta_6^4 \le 1$, 
$10^2\sqrt{\gamma}(\Omega^*)^3/\Lambda\le \delta_7^3<\eta^3\gamma^3/10^6$, $d'>10\tilde \epsilon>0$, and $d'\mu\tau k\ge 4\cdot 10^3$.
Both configurations contain distinguished sets $V_0,V_1\subseteq  \colouringp{0}$ and $V_2,V_3\subseteq  \colouringp{1}$.

Suppose that $(W_A,W_B,\shrubA,\shrubB)$ is a $(\tau k)$-fine partition of a rooted tree
$(T,r)$ of order at most~$k$ such that $|W_A\cup W_B|\leq k^{0.1}$. Let $T'$ be the tree induced by all the cut-vertices
$W_A\cup W_B$ and all the internal shrubs.

Then there exists an embedding $\phi$ of $T'$ such that $$\phi(W_A)\subset V_1
\quad
\phi(W_B)\subset V_0
\quad\mbox{and}\quad
\phi(T'-(W_A\cup W_B))\subset \colouringp{1}\;.
$$
\end{lemma}
\begin{proof}
For simplicity, let us assume that $r\in W_A$. The case when $r\in W_B$ is similar. 
The $(\tau k)$-fine partition
$(W_A,W_B,\shrubA,\shrubB)$ induces a $(\tau k)$-fine partition in $T'$.
By
Lemma~\ref{lem:orderedskeleton}, the tree $T'$ has an ordered skeleton $(X_0, X_1,\ldots, X_m)$ where the $X_i$ are either shrubs or \kknnaaggss ($X_0$ being a \kknnaaggNOSPACE).

Our strategy is as follows. We sequentially
embed the \kknnaaggss and the internal shrubs in the order given by the ordered
skeleton.  For embedding the \kknnaaggss we use Lemma~\ref{lem:embed:greyFOREST} in Preconfiguration~$\mathbf{\mathbf{(exp)}}$, and
Lemma~\ref{lem:embed:superregular} in Preconfiguration~$\mathbf{\mathbf{(reg)}}$. For embedding the internal shrubs, we use
Lemmas~\ref{lem:embedStoch:DIAMOND6} and~\ref{lem:embedStoch:DIAMOND7}
if we have Configurations~$\mathbf{\mathbf{(\diamond6)}}$, and~$\mathbf{\mathbf{(\diamond7)}}$, respectively.

Throughout, $\phi$
denotes the current (partial) embedding of
$(X_0,X_1\ldots,X_m)$.
In consecutive steps, we extend $\phi$.  We  define  auxiliary sets $D_i\subset V(G)$  which will serve for reserving space for the roots of the shrubs $X_i$. So the set $Z_{<i}:=\bigcup_{j<i} (\phi(X_{j})\cup D_{j})$ contains what is already used and what should (mainly) be avoided. 

Let $W_{A,i}:=W_A\cap V(X_i)$, and $W_{B,i}:=W_B\cap V(X_i)$. For each $y\in
W_{A,{j}}$ with $j\leq i$ let $$S_y:=(V_2\cap \neighbour_G(\phi(y)))\setminus
Z_{<i},$$ except if the latter set has size $>k$, in that case we choose a
subset of size $k$. This is a target set for the roots of shrubs adjacent to $y$.

Also, in the case $X_i$ is a shrub, we write $r_i$ for its root, and $f_i$ for the only other vertex neighbouring $W_A\cup W_B$.  Note that $f_i$ is a fruit of $(X_i,r_i)$.

The value $h=6$ or $h=7$ indicates whether we have configuration $\mathbf{\mathbf{(\diamond6)}}$ or $\mathbf{\mathbf{(\diamond7)}}$.
Define 
\begin{align}\label{eq:defFF}
 F_i:=\shadow_{G-\HugeVertices}\left(Z_{<i},\frac{\delta_h k}4\right)\cup Z_{<i}\;.
\end{align}

Define $U_i:=F_i$ if we have Preconfiguration $\mathbf{\mathbf{(exp)}}$ (note that in that case we have Configuration~$\mathbf{\mathbf{(\diamond6)}}$). To define $U_i$ in case of Preconfiguration $\mathbf{\mathbf{(reg)}}$ we make use of the super-regular pairs $(Q^{(j)}_0,Q^{(j)}_1)$ ($j\in\mathcal Y$). Set
\begin{equation}\label{eq:defUU}
U_i:=F_i\cup\bigcup\left\{Q^{(j)}_1\::\: j\in\mathcal Y, |Q^{(j)}_1\cap F_i|\ge \frac{|Q^{(j)}_1|}2\right\}\;.
\end{equation}
In either case, we have $|U_i|\le 2 |F_i|$.

Finally, set 
\begin{align}
 \label{eq:defWW}
W_i&:=\shadow_{G-\HugeVertices}\left(U_i,\frac{\delta_h k}2\right)\cup Z_{<i}\;.
\end{align}

We will now show how to embed successively all $X_i$. At each step $i$, our embedding $\phi$ will have the following properties:
\begin{enumerate}[(a)]
\item $\phi(W_{A,i})\subset V_1\sm F_i$ and $\phi(W_{B,i})\subset V_0$,\label{eatmoreseeds}
\item for each $y\in W_{A,{j}}$ with $j\leq i$ we have   $
|S_y\cap
  \phi(X_i)|\le |S_y\cap D_{i}|+k^{3/4}$,  \label{thereservationfortheroots}
\item $|Z_{< i+1}|\le 2k$,\label{sizeCiDi} 
\item\label{Didisjunkt} $D_i\subseteq \colouringp{1}\sm (\phi (X_i)\cup Z_{<i})$,
\item\label{Xifastdisjunkt} $\phi(X_i-r_i)$ is disjoint from $ \bigcup_{j<i} D_j$,
\item $\phi(f_i)\in V_2\sm W_i$ if $X_i$ is a shrub,\label{eatmorefruit}
\item $\phi (X_i)\subset \colouringp{1}$ if $X_i$ is a shrub.\label{putitintocolour1}
\end{enumerate}

(We remark that since $r_i$ is not defined for \kknnaaggss $X_i$, condition~\eqref{Xifastdisjunkt} means that $\phi(X_i)$ is disjoint from $ \bigcup_{j<i}\cup D_j$ for \kknnaaggss $X_i$.)

It is clear that conditions \eqref{eatmoreseeds} and \eqref{putitintocolour1} ensure that in step $m$ we have found the desired embedding for $T'$.

Before we show how to embed each $X_i$ fulfilling the properties above, let us quickly derive a useful bound.
By Fact~\ref{fact:shadowbound} and~\eqref{sizeCiDi}, we have that $|F_i|\le \frac{9\Omega^*}{\delta_h}k$
 for all $i\le m$. Thus,
using  $|U_i|\le 2 |F_i|$ and again Fact~\ref{fact:shadowbound} and~\eqref{sizeCiDi}, we get that
\begin{equation}\label{eq:Usmallish}
|W_i|\le \frac{38(\Omega^*)^2}{\delta_h^2}k\;.
\end{equation}

\smallskip

Now suppose we are at step $i$ with $0\leq i\leq m$. That is, we have already embedded all $X_j$ with $j<i$, and are about to embed $X_i$. 

\medskip

First assume that $X_i$ is a \kknnaaggNOSPACE. Note that if $i\neq 0$, then there is exactly one fruit $f_\ell$ with $\ell<i$ which neighbours $X_i$. Set $N_i:=\neighbour_G(\phi(f_\ell))$ in this case, and let $N_i:=V(G)$ for $i=0$. 
We distinguish between the two preconfigurations we might be in. 

\smallskip

Suppose first we are in 
Preconfiguration~$\mathbf{\mathbf{(exp)}}$. Recall that then we are in Configuration~$\mathbf{(\diamond6)}$. 

We use Lemma~\ref{lem:embed:greyFOREST} to embed the single tree $X_i$  with the following setting: $\ell_\PARAMETERPASSING{L}{lem:embed:greyFOREST}:=1$,  $V_{1,\PARAMETERPASSING{L}{lem:embed:greyFOREST}}:=V_1$, $V_{2,\PARAMETERPASSING{L}{lem:embed:greyFOREST}}:=V_0$, $U^*_\PARAMETERPASSING{L}{lem:embed:greyFOREST}:=(N_i\cap V_1)\setminus U_i= (N_i\cap V_1)\setminus F_i$, $U_\PARAMETERPASSING{L}{lem:embed:greyFOREST}:=U_i=F_i$, $Q_\PARAMETERPASSING{L}{lem:embed:greyFOREST}:=\frac{18\Omega^*}{\delta_6}$, $\zeta_\PARAMETERPASSING{L}{lem:embed:greyFOREST}:=\delta_6$, and $\gamma_\PARAMETERPASSING{L}{lem:embed:greyFOREST}:=\gamma$. 
Note that $U^*_\PARAMETERPASSING{L}{lem:embed:greyFOREST}$ is large enough by~\eqref{eatmorefruit} for $\ell$ and by~\eqref{COND:D6:2} and~\eqref{COND:D7:2}, respectively.
Lemma~\ref{lem:embed:greyFOREST} gives an embedding of the tree $X_i$ such that $\phi(\Veven(X_i))\subset V_1\setminus F_i$ and $\phi(\Vodd(X_i))\subset V_0\setminus F_i$ , which maps the root of $X_i$ to the neighbourhood of its parent's image. Note that this ensures~\eqref{eatmoreseeds} and~\eqref{Xifastdisjunkt} for step $i$, and setting $D_i:=\emptyset$ we also ensure~\eqref{sizeCiDi} and~\eqref{Didisjunkt}. Property~\eqref{thereservationfortheroots} holds since $V_2\cap \phi (X_i)=\emptyset$. Since $X_i$ is a \kknnaaggNOSPACE,~\eqref{eatmorefruit} and~\eqref{putitintocolour1} are empty.

\smallskip

Suppose now we are in 
Preconfiguration~$\mathbf{\mathbf{(reg)}}$. Then
let $j\in \mathcal Y$ be such that $(N_i\cap Q_1^{(j)})\setminus U_i\neq \emptyset$. Such an index $j$ exists by~\eqref{eatmorefruit} for $\ell$ and by~\eqref{COND:D6:2} and~\eqref{COND:D7:2}, respectively, if $i\neq 0$, and trivially if $i=0$. We shall use Lemma~\ref{lem:embed:superregular} to embed $X_i$ in $(Q^{(j)}_0,Q^{(j)}_1)$. More precisely, we use Lemma~\ref{lem:embed:superregular}
with $A_\PARAMETERPASSING{L}{lem:embed:superregular}:=Q^{(j)}_1$, $B_\PARAMETERPASSING{L}{lem:embed:superregular}:=Q^{(j)}_0$, $\epsilon_\PARAMETERPASSING{L}{lem:embed:superregular}:=\tilde \epsilon$, $d_\PARAMETERPASSING{L}{lem:embed:superregular}:=d'$, $\ell_\PARAMETERPASSING{L}{lem:embed:superregular}:=\mu k$, $U_A:=U_i\cap A$, $U_B:=\phi(W_{B,<i})\cap B$ (then $|U_A|\leq |A|/2$ by the definition of $U_i$ and the choice of $j$).

Lemma~\ref{lem:embed:superregular} yields a $(\Veven(X_i)\hookrightarrow
V_1\setminus F_i, \Vodd(X_i)\hookrightarrow V_0)$-embedding of $X_i$, which maps the root of $X_i$ to the neighbourhood of its parent's image. Setting $D_i:=\emptyset$, we have~\eqref{eatmoreseeds}--\eqref{putitintocolour1}.

\medskip

So let us now assume that $X_i$ is a shrub. The  parent $y$ of
the root $r_{i}$ of $X_{i}$ lies in $W_{A,\ell}$ for some $\ell<i$.  By~\eqref{eatmoreseeds} for $\ell$, we mapped $y$ to a vertex $\phi(y)\in V_1\sm F_\ell$.
 As $\deg_{G}(\phi(y),V_2)\ge \delta_h k$ (by~\eqref{COND:D6:1} and \eqref{COND:D7:1}, respectively), and since $\phi (y)\notin F_\ell$, we have 
 \begin{equation}\label{cerropochoco}
 |S_y|\ge \frac{3\delta_h k}4.
 \end{equation} 
 
 Using~\eqref{thereservationfortheroots} for all $j$ with $\ell\leq j<i$, and using that the sets $D_j$ are pairwise disjoint by~\eqref{Didisjunkt}, we see that
  $$|S_y\cap \phi(X_0\cup\ldots \cup X_{i-1})|=|S_y\cap \phi(X_\ell \cup\ldots
  \cup X_{i-1})|\le |S_y\cap \bigcup_{\ell\le j<i}D_j|+m\cdot k^{3/4}\le 
  |S_y\cap \bigcup_{0\le j<i}D_j| +m\cdot k^{3/4}. $$ Therefore, and as
  by~\eqref{Didisjunkt} and~\eqref{Xifastdisjunkt}, the sets $\phi(X_0\cup\ldots
  X_{i-1})$ and $\bigcup_{0\le j<i}D_j$ are disjoint except for the at most
  $m\le |W_A\cup W_B|\le k^{0.1}$ roots $r_j$ of shrubs $X_j$, and since
  $k\gg1$, we have $$|S_y|\geq |S_y\cap \phi(X_0\cup\ldots\cup X_{i-1})| + 
  |S_y\cap \bigcup_{0\le j<i}D_j| -m\ge 2|S_y\cap \phi(X_0\cup\ldots\cup
  X_{i-1})|-k^{0.9}. $$ Thus,
\begin{equation*}
|S_y\setminus \phi (X_0\cup\ldots\cup X_{i-1})|\ge \frac{|S_y|- k^{0.9}}{2}\overset{\eqref{cerropochoco}}\geq \frac{3\delta_h k}{8}-\frac{k^{0.9}}{2}>\frac{\delta_h k}{3}\;.
\end{equation*}

So for $U^*:=S_{y}\setminus \phi (X_0\cup\ldots\cup X_{i-1})$ we have that $|U^*|\ge \frac{\delta_h
k}{3}$.
 If we have Configuration~$\mathbf{(\diamond6)}$ or $\mathbf{(\diamond7)}$ we use 
Lemma~\ref{lem:embedStoch:DIAMOND6} or~\ref{lem:embedStoch:DIAMOND7},
respectively, with input 
$U_\PARAMETERPASSINGR{L}{lem:embedStoch:DIAMOND6}{lem:embedStoch:DIAMOND7}:=W_{i}$,
$U^*_\PARAMETERPASSINGR{L}{lem:embedStoch:DIAMOND6}{lem:embedStoch:DIAMOND7}:=U^*$,
$L_\PARAMETERPASSINGR{L}{lem:embedStoch:DIAMOND6}{lem:embedStoch:DIAMOND7}:=|W_{A,{i}}|$,
$\gamma_\PARAMETERPASSINGR{L}{lem:embedStoch:DIAMOND6}{lem:embedStoch:DIAMOND7}:=\gamma$,
the  family $\{P_t\}_\PARAMETERPASSINGR{L}{lem:embedStoch:DIAMOND6}{lem:embedStoch:DIAMOND7}:=\{S_y\}_{y\in
W_{A,{j}}, j<i}$, and the rooted tree $(X_{i},r_{i})$ with fruit $f_{i}$. Further,
for Configuration~$\mathbf{(\diamond6)}$, set $\delta_\PARAMETERPASSING{L}{lem:embedStoch:DIAMOND6}:=\delta_6$,  $V_{2,\PARAMETERPASSING{L}{lem:embedStoch:DIAMOND6}}:= V_2$ and $V_{3,\PARAMETERPASSING{L}{lem:embedStoch:DIAMOND6}}:=
V_3$ and for Configuration~$\mathbf{(\diamond7)}$, set $\delta_\PARAMETERPASSING{L}{lem:embedStoch:DIAMOND7}:=\delta_7$, $\ell_{\PARAMETERPASSING{L}{lem:embedStoch:DIAMOND7}}:=1$, $Y_{1,\PARAMETERPASSING{L}{lem:embedStoch:DIAMOND7}}:= V_2$ and $Y_{2,\PARAMETERPASSING{L}{lem:embedStoch:DIAMOND7}}:=
V_3$.
The  output of Lemma~\ref{lem:embedStoch:DIAMOND6} or~\ref{lem:embedStoch:DIAMOND7},
respectively, is the extension of our embedding $\phi$ to $X_{i}$, and
a set
$D_{i}:=C_\PARAMETERPASSINGR{L}{lem:embedStoch:DIAMOND6}{lem:embedStoch:DIAMOND7}\subseteq (V_2\cup V_3)\sm (W_i \cup \phi(X_i))$
for which property~\eqref{eatmoreseeds} (which is empty) and  properties~\eqref{thereservationfortheroots}--\eqref{putitintocolour1} hold. 
\end{proof}
 
\begin{lemma}\label{lem:embed:skeleton8}
Suppose we are in Setting~\ref{commonsetting} and~\ref{settingsplitting} and suppose further we
have Configuration~$\mathbf{(\diamond8)}(\delta,
\frac{\eta\gamma}{400},\epsilon_1,\epsilon_2,d_1,d_2,\mu_1,\mu_2,h_1,0)$, with
$2\cdot 10^5(\Omega^*)^6/\Lambda\le\delta^6$,
$\delta<\gamma^2\eta^4/(10^{16}(\Omega^*)^2)$, $d_2>10\epsilon_2>0$,
$d_2\mu_2\tau k\ge 4\cdot 10^3$, and $\max\{\epsilon_1, \tau/\mu_1\}\le \eta^2\gamma^2d_1/(10^{10}(\Omega^*)^3)$. Recall that we have
distinguished sets $V_0, \ldots, V_4$ and a \semiregular matching $\mathcal N$.

Let $(W_A,W_B,\shrubA,\shrubB)$ be a $(\tau k)$-fine partition of a rooted tree
$(T,r)$ of order at most~$k$. Let $T'$ be the tree induced by all the cut-vertices
$W_A\cup W_B$ and all the internal shrubs. Suppose that 
\begin{equation}\label{eq:takeaway}
v(T')<h_1-\frac{\eta^2 k}{10^5}\;.
\end{equation}

Then there exists an embedding $\phi$ of $T'$ such that $\phi(W_A)\subset V_1$,
$\phi(W_B)\subset V_0$, and $\phi(T')\subset \colouringp{0}\cup\colouringp{1}$.
\end{lemma}
\begin{proof}
We assume that $r\in W_A$. The case when $r\in W_B$ is similar. 

Let $\mathcal K$ be the set of all  \kknnaaggss of the  $(\tau k)$-fine partition  $(W_A,W_B,\shrubA,\shrubB)$ of $T$. For each such \kknnaagg $K\in\mathcal K$ set  $Y_K:=K\cup \children_{T'}(K)$.
 We call the subgraphs $Y_K$ {\em extended \kknnaaggssNOSPACE}.
 Set $\mathcal Y:=\{Y_K:K\in\mathcal K\}$ and $W_C:=V(\bigcup\mathcal Y\sm \bigcup \mathcal K)$. Since $W_C\subseteq V(T')$, we clearly have that $|W_C|\leq|W_A\cup W_B|$.

Note that the forest $T'-\bigcup\mathcal Y$ consists of the set $\mathcal P$ of peripheral subshrubs of internal  shrubs of  $(W_A,W_B,\shrubA,\shrubB)$, and the set $\mathcal S$ of principal subshrubs of internal shrubs of  $(W_A,W_B,\shrubA,\shrubB)$.
It is not difficult to observe that there is a sequence $(X_0, X_1,\ldots, X_m)$ such that $X_i=(M_i, Y_i,\mathcal P_i)$, $M_i\in\mathcal S$ and $\mathcal P_i\subseteq\mathcal P$  for each $i\leq m$, and such that we have the following. 
\begin{enumerate}[(I)]
\item\label{startstartstart}  $M_0=\emptyset$ and $Y_0$ contains $r$.
\item\label{peugeot505goodbye}
 $\mathcal P_i$ are exactly those peripheral subshrubs whose parents lie in $Y_i$.
\item\label{esmeralda} The parent $f_i$ of $Y_i$ lies in $M_i$ (unless $i=0$).
\item\label{prince} The parent $r_i$ of $M_i$ lies in some $Y_j$ with $j<i$ (unless $i=0$),
\item $\bigcup_{i\leq m}V(M_i\cup Y_i\cup\bigcup \mathcal P_i)=V(T')$.
\end{enumerate}
See Figure~\ref{fig:subshrubs} for an illustration.
\begin{figure}[ht]
\centering 
\includegraphics{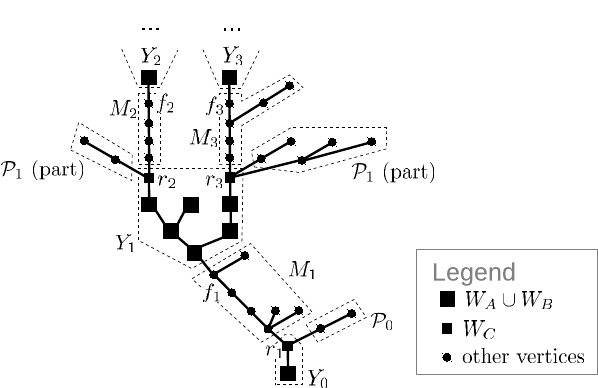}
\caption[Ordering the principal and peripheral subshrubs in
Lemma~\ref{lem:embed:skeleton8}]{An example of a sequence $(X_0,X_1,X_2,X_3,\ldots)$ in Lemma~\ref{lem:embed:skeleton8}.}
\label{fig:subshrubs}
\end{figure}

We now successively embed the elements of $X_i$, except possibly for a part of
the subshrubs in $\mathcal P_i$. The omitted peripheral subshrubs will be
embedded at the very end, after having completed the inductive procedure we are about to describe now.

We shall make use of the following lemmas:
 Lemma~\ref{lem:embed:superregular} (for embedding \kknnaaggssNOSPACE),
 Lemmas~\ref{lem:embed:BALANCED}  and~\ref{lem:embed:regular} (for embedding
 peripheral subshrubs in $\mathcal N$),  Lemma~\ref{lem:embedStoch:DIAMOND7}
 (for embedding principal subshrubs in $V_3\cup V_4$).

Throughout, $\phi$
denotes the current (partial) embedding of
$T'$. In each step $i$ we embed $M_i\cup Y_i$ and a subset of $\mathcal P_i$, and  denote by $\phi (X_i)$ the image of these sets (as far as it is defined). 
We also define an auxiliary set $D_i\subset V(G)$ which will 
ensure that there is enough space for the roots of the subshrubs $M_\ell$ with
$\ell >i$.
Set $$Z_{<i}:=\bigcup_{j<i} (\phi(X_{j})\cup D_{j}).$$ 
Our plan for embedding the various parts of $X_i$ is depicted in Figure~\ref{fig:Embed8detailed}, which is
a refined version of Figure~\ref{fig:DIAMOND8}.
\begin{figure}[ht]
\centering 
\includegraphics{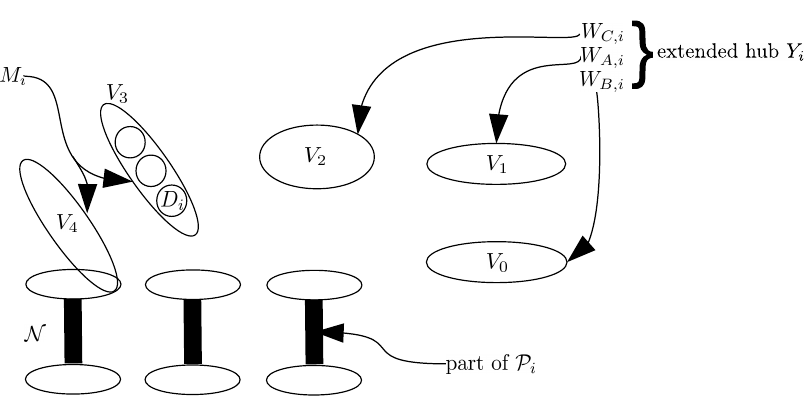}
\caption[Embedding the internal part of the tree in
Configuration~$\mathbf{(\diamond8)}$]{Embedding a part of the internal tree in
 Lemma~\ref{lem:embed:skeleton8}.}
\label{fig:Embed8detailed}
\end{figure} 

Let $W_{O,i}:=W_O\cap V(Y_i)$ for $O=A,B,C$. For each $y\in W_{C,{i}}$  let
$$S_y:=(V_3\cap \neighbour_G(\phi(y)))\setminus Z_{<i},$$ except if this set has
size more than $k$, in which case we choose any subset of size $k$. Similar as
in the preceding lemma, this is a target set for the roots of the
principal subshrub adjacent to $y$.

Fix a matching involution $\mathfrak{d}$ for $\mathcal N$, and 
 for $\ell=1,2$
define
\begin{align}\label{eq:defFF8}
 F^{(\ell)}_i:=Z_{<i}\cup \shadow^{(\ell)}_{G-\HugeVertices}\left(\ghost_{\mathfrak{d}}(Z_{<i}),\frac{\delta k}8\right)\;.
\end{align}

We  use the super-regular pairs $(Q^{(j)}_0,Q^{(j)}_1)$ ($j\in\mathcal Y$) to
define
\begin{equation}\label{eq:defUU8}
U_i:=F^{(2)}_i\cup\bigcup\left\{Q^{(j)}_1\::\: j\in\mathcal Y, |Q^{(j)}_1\cap F^{(2)}_i|\ge \frac{|Q^{(j)}_1|}2\right\}\;.
\end{equation}
We have 
\begin{equation}\label{Uileq2Fi}
|U_i|\le 2 |F^{(2)}_i|.
\end{equation}
Finally, for $\ell=1,2$ set 
\begin{align}
 \label{eq:defWW8}
W^{(\ell)}_i&:=\shadow^{(\ell)}_{G-\HugeVertices}\left(U_i,\frac{\delta k}2\right)\;.
\end{align}

We will now show how to define successively our embedding. At each step $i$, the embedding $\phi$ will be defined for $M_i\cup Y_i$ and a subset of $\mathcal P_i$, and it will have the following properties:
\begin{enumerate}[(a)]
\item $\phi(W_{A,i})\subset V_1\sm F^{(2)}_i$ and $\phi(W_{B,i})\subset V_0$,\label{eatmoreseeds8}
\item $\phi(W_{C,i})\subseteq V_2\sm F^{(1)}_i$,\label{wheretherootsgo}
\item $\phi(f_i)\in V_2\sm (F^{(1)}_i\cup W^{(1)}_i)$,\label{wherethefruitgoes}
\item for each $y\in W_{C,{j}}$ with $j\leq i$ we have   $
|S_y\cap  \phi(X_i)|\le |S_y\cap D_{i}|+ 2k^{3/4}$, 
\label{thereservationfortheroots8}
\item $|Z_{< i+1}|\le 2k$,\label{sizeCiDi8} 
\item\label{Didisjunkt8} $D_i\subseteq V_3\sm (\phi (X_i)\cup Z_{<i})$,
\item\label{Xifastdisjunkt8} $\phi(X_i\sm (V(M_i)\cap \children(W_C)))$ is
disjoint from $
\bigcup_{j<i} D_j $,\footnote{Note that $V(M_i)\cap \children(W_C)$ contains a
single vertex, the root of $M_i$.}
\item $\phi (X_i)\subset \colouringp{1} \cup \phi(Y_i \cup f_i)$,\label{putitintocolour18}
\item if $P\in \mathcal P_i$ is not embedded in step $i$ then for its parent $w\in W_C$ we have that $\deg_{\GD}(\phi(w),V_3)\ge h_1-|\phi(X_i)\cap V(\mathcal N)|-\frac{\eta^2
k}{10^6}$.\label{youknowi'mbad}
\end{enumerate}

Note that for~\eqref{putitintocolour18}, since $f_0$ is not defined, we assume $\phi (f_0)=\emptyset$.

Before continuing, let us remark that~\eqref{putitintocolour18} together with~\eqref{Didisjunkt8} implies that at each step $i$ we have
\begin{equation}\label{eq:littlein0}
\left|Z_{<i}\cap\colouringp{0}\right|\le 3\cdot (|W_A|+|W_B|)\leBy{D\ref{ellfine}\eqref{few}} \frac{2016}\tau< \frac{\delta k}8\;.
\end{equation}

Also note that 
by Fact~\ref{fact:shadowbound} and by~\eqref{sizeCiDi8}, we have
\begin{equation}\label{eq:Fsmallish8}
|F^{(2)}_i|\le \frac{65(\Omega^*)^2}{\delta^2}k\;,
\end{equation}
and
\begin{equation}\label{eq:Usmallish8}
|W^{(2)}_i|\le \frac{520(\Omega^*)^4}{\delta^4}k\;.
\end{equation}

By~\eqref{wheretherootsgo} and by~\eqref{COND:D8:3} we have that $|S_y|\ge
\frac{7\delta k}8$. Now, using~\eqref{thereservationfortheroots8},~\eqref{Didisjunkt8} and~\eqref{Xifastdisjunkt8},  we can calculate similarly as in the previous lemma that at each step $i$ we have
\begin{equation}\label{eq:zindD8}
\left|S_y\setminus \bigcup_{\ell\le i}\phi(X_\ell)\right|\ge \frac{3\delta k}8\;.
\end{equation}

\medskip

Now assume we are at step $i$ of the inductive procedure, that is, we have already dealt with $X_0,\ldots , X_{i-1}$ and wish to embed (parts of) $X_i$.

\bigskip

 We start with embedding $M_i$, except if $i=0$, when we go directly to embedding $Y_0$. We shall embed $M_i$ in
$V_3\cup V_4$, except for the fruit $f_i$, which will be mapped to $V_2$.
The embedding has three stages. First we embed
$M_i-M_i(\uparrow f_i)$, then we embed $f_i$, and finally we embed the
forest $M_i(\uparrow f_i)-f_i$. The embedding of $M_i-M_i(\uparrow f_i)$ is an
application of Lemma~\ref{lem:embedStoch:DIAMOND7} analogous to the case of
Configuration~$\mathbf{(\diamond7)}$ in the previous
Lemma~\ref{lem:embed:skeleton67}. That is, set
$Y_{1,\PARAMETERPASSING{L}{lem:embedStoch:DIAMOND7}}:=V_3$,
$Y_{2,\PARAMETERPASSING{L}{lem:embedStoch:DIAMOND7}}:=V_4$, let
$$U_\PARAMETERPASSING{L}{lem:embedStoch:DIAMOND7}^*:=S_{r_i}\setminus
\bigcup_{\ell< i}\phi(X_i),$$ where $r_i$ 
 lies in $W_C$ by~\eqref{prince}, and
 $$U_\PARAMETERPASSING{L}{lem:embedStoch:DIAMOND7}:=F^{(2)}_i\cup W^{(2)}_i\;.$$ Note that
\[
|U_\PARAMETERPASSING{L}{lem:embedStoch:DIAMOND7}|\le
\frac{10^3(\Omega^*)^4}{\delta^4}k\le \frac{\delta\Lambda}{2\Omega^*}k,
\]
and by~\eqref{eq:zindD8} (which we use for $i-1$), also
\[
|U^*_\PARAMETERPASSING{L}{lem:embedStoch:DIAMOND7}|\ge
\frac{3\delta k}8.
\]
 The family
$\{P_1,\ldots,P_L\}_\PARAMETERPASSING{L}{lem:embedStoch:DIAMOND7}$ is the same as $\{S_y\}_{y\in \bigcup_{j<i}W_{C,j}}$. There is only one tree to be embedded, namely 
$M_i-M_i(\uparrow f_i)$. It is not difficult to check that all the
conditions of Lemma~\ref{lem:embedStoch:DIAMOND7} are fulfilled.
Lemma~\ref{lem:embedStoch:DIAMOND7} gives an embedding of $M_i-M_i(\uparrow
f_i)$ in $V_3\cup V_4\subset \colouringp{1}$ with the property that
$\parent(f_i)$, the parent of the fruit~$f_i$,  is mapped to $V_3\setminus ( F^{(2)}_i \cup W^{(2)}_i )$. The lemma further gives a set
$D':=C_\PARAMETERPASSING{L}{lem:embedStoch:DIAMOND7}$ of size
$v(M_i-M_i(\uparrow f_i))$ such that $$|S_y\cap \phi(M_i-M_i(\uparrow f_i))|\le|S_y\cap
D'|+k^{0.75}$$ for each $y\in \bigcup_{j<i}W_{C,j}$.

Using the degree condition~\eqref{COND:D8:4} we can embed $f_i$
 to $$V_2\setminus ( F^{(1)}_i \cup W^{(1)}_i )$$
(recall that~\eqref{eq:littlein0} asserts that only very little space in $V_2$
is occupied). 
This ensures~\eqref{wherethefruitgoes} for $i$.

To embed $M_i(\uparrow f_i)-f_i$ we use again
Lemma~\ref{lem:embedStoch:DIAMOND7}. The parameters are this time
$Y_{1,\PARAMETERPASSING{L}{lem:embedStoch:DIAMOND7}}:=V_3$,
$Y_{2,,\PARAMETERPASSING{L}{lem:embedStoch:DIAMOND7}}:=V_4$,
\begin{align*}
U_\PARAMETERPASSING{L}{lem:embedStoch:DIAMOND7}^*&:=(\neighbour_G(\phi(f_i))\cap
V_3)\setminus (Z_{<i}\cup\phi(M_i-M_i(\uparrow f_i)) )\;\mbox{, and}\\
U_\PARAMETERPASSING{L}{lem:embedStoch:DIAMOND7}&:=Z_{<i}\cup
\phi(M_i-M_i(\uparrow f_i))\cup D'\;.
\end{align*}
Note that $|U_\PARAMETERPASSING{L}{lem:embedStoch:DIAMOND7}^*|\ge \frac{\delta k}4$
by~\eqref{COND:D8:3}, by the fact that $\phi(f_i)\not\in W^{(1)}_i$, and as $v(T_i)+i<\delta k/8$.
The family $\{P_1,\ldots,P_L\}_\PARAMETERPASSING{L}{lem:embedStoch:DIAMOND7}$ is $\{S_y\}_{y\in \bigcup_{j<i}W_{C,j}}$. The trees to be embedded are
the components of $M_i(\uparrow f_i)-f_i$ rooted at the children of
$f_i$.
All the conditions of Lemma~\ref{lem:embedStoch:DIAMOND7} are fulfilled. The
lemma provides an embedding in $V_3\cup V_4\subset \colouringp{1}$. It further
gives a set $D'':=C_\PARAMETERPASSING{L}{lem:embedStoch:DIAMOND7}$ of size
$v(M_i(\uparrow f_i))-1$ such that $$|S_y\cap \phi(M_i(\uparrow
f_i)-f_i)|\le|S_y\cap D''|+k^{0.75}$$ for each $y\in \bigcup_{j<i}W_{C,j}$.
Then $D_i:=V_3\cap(D'\cup D'')$ is such that for each $y\in \bigcup_{j<i}W_{C,j}$,
\begin{equation}\label{eq:greatlyduplicated}
|S_y\cap \phi(M_i)|\le  |S_y\cap D_i|+2k^{0.75}\;,
\end{equation}
as $S_y\subset V_3$ and $\phi(f_i)\notin V_3$. Note that this choice of $D_i$
also ensures~\eqref{sizeCiDi8} for $i$, and we have by the choices of
$U_\PARAMETERPASSING{L}{lem:embedStoch:DIAMOND7}^*$ and
$U_\PARAMETERPASSING{L}{lem:embedStoch:DIAMOND7}$ in both applications of
Lemma~\ref{lem:embedStoch:DIAMOND7} that
\begin{equation}\label{camundongo}
D_i\subseteq V_3\sm (\phi(M_i)\cup Z_{<i})\quad\text{ and }\quad
\phi(X_i\sm (V(M_i)\cap \children(W_C)))\cap
\bigcup_{j<i}D_j=\emptyset.\,
\end{equation}

\bigskip

We now turn to embedding $Y_i$.
 Our plan is to use first Lemma~\ref{lem:embed:superregular} to embed
$Y_i\sm W_C$ in $(Q^{(j)}_0,Q^{(j)}_1)$ for an appropriate index $j$. After
that, we shall show how to embed $W_{C,i}$.

If $i=0$ then take an arbitrary $j\in\mathcal Y$. Otherwise note that by~\eqref{esmeralda}, the parent $f_i$ of the root of $Y_i$ lies in $M_i$. Note that $f_i$ is a fruit in $M_i$. 
Let $j\in \mathcal Y$ be such that $(\neighbour_G(\phi(f_{i}))\cap
Q_1^{(j)})\setminus U_i\neq \emptyset$. 
 Such an index $j$ exists by~\eqref{COND:D8:2}
 and the fact that $
 \phi(f_{i})\not\in W^{(1)}_{i}$ by~\eqref{wherethefruitgoes} for $i$.
 
 We  use
Lemma~\ref{lem:embed:superregular} with
$A_\PARAMETERPASSING{L}{lem:embed:superregular}:=Q^{(j)}_1$,
$B_\PARAMETERPASSING{L}{lem:embed:superregular}:=Q^{(j)}_0$,
$\epsilon_\PARAMETERPASSING{L}{lem:embed:superregular}:= \epsilon_2$,
$d_\PARAMETERPASSING{L}{lem:embed:superregular}:=d_2$,
$\ell_\PARAMETERPASSING{L}{lem:embed:superregular}:=\mu_2 k$, $U_A:=U_i\cap
A_\PARAMETERPASSING{L}{lem:embed:superregular}$, $U_B:=Z_{<i}\cap
B_\PARAMETERPASSING{L}{lem:embed:superregular}$.  By the choice of $j$ and the
definition of $U_i$, we find that $U_A$ is small enough, and
using~\eqref{eq:littlein0} we see that $U_B$ is also small enough.
Lemma~\ref{lem:embed:superregular} yields a $(\Veven(Y_i-
W_C)\hookrightarrow V_1\setminus F^{(2)}_i, \Vodd(Y_i-
W_C)\hookrightarrow V_0)$-embedding of $Y_i-
W_C$. We clearly see
condition~\eqref{eatmoreseeds8} satisfied for $i$.

We now embed successively the vertices of the set
$W_{C,i}=\{w_\ell:\ell=1,\ldots ,|W_{C,i}|\}$.
By the definition of the set $W_C$, we know that the parent $x$ of~$w_\ell$ lies in $W_{A,i}$.
Combining~\eqref{COND:D8:1} with the fact that $\phi(x)\in V_1\sm F^{(2)}_i$ by~\eqref{eatmoreseeds8} for $i$, we have
that 
$$\left|\neighbour_G\left(\phi(x),V_2\setminus (F^{(1)}_i \sm Z_{<i})\right)\right|\ge \frac{7\delta k}8\;.$$ Thus
by~\eqref{eq:littlein0} and since $V_2\subseteq \colouringp{0}$, we can
accommodate $w_\ell$ in $V_2\setminus F^{(1)}_i$. This is as
desired for~\eqref{wheretherootsgo} in step $i$.

\bigskip

We now turn to $\mathcal P_i$. We will embed a subset of these peripheral subshrubs in
$\mathcal N$. This procedure is  divided into two stages. First we shall embed as many
subshrubs as possible in $\mathcal N$ in a balanced way, with the help of
Lemma~\ref{lem:embed:BALANCED}. When it is no longer
possible to embed any subshrub in a balanced way in $\mathcal N$, we  embed  in $\mathcal N$ as many of the leftover
subshrubs as possible, in an unbalanced way. For this part of the embedding we use
Lemma~\ref{lem:embed:regular}. 

 By~\eqref{peugeot505goodbye} all the parents of the subshrubs in $\mathcal
 P_i$ lie in $W_{C,i}$. For $w_\ell\in W_{C,i}$, let $\mathcal P_{i,\ell}$
 denote the set of all subshrubs in $\mathcal P_i$ adjacent to $w_\ell$. In the
 first stage, we shall embed, successively for $j=1, \ldots ,|W_{C,i}|$, either
 all or none of $\mathcal P_{i,j}$ in a balanced way in $\mathcal N$. Assume inductively that
\begin{equation}\label{eq:INDbalanced}
\phi\Big(\bigcup_{p<j}\mathcal P_{i,p}\Big) \mbox{ is $(\tau k)$-balanced with respect to $\mathcal N$.} 
\end{equation}

Construct a
\semiregular matching $\mathcal N_j$ absorbed by $\mathcal N$ as follows. Let 
$\mathcal N_j:=\{(X_1',X_2')\::\: (X_1,X_2)\in\mathcal N\}$, where for
$(X_1,X_2)\in\mathcal N$ we define $(X'_1,X_2')$ as the maximal balanced
unoccupied subpair seen from $\phi(w_{j})$, i.e., for $b=1,2$, we take
$$X'_b\subset \left(X_b\cap \neighbour_{\Gblack}(\phi(w_{j})\right)\setminus
\left(\phi(\bigcup_{p<j}\mathcal P_{i,p})\cup \bigcup_{\ell <i}\phi (X_\ell)\right)$$ maximal subject to $|X'_1|=|X'_2|$.
If $|V(\mathcal N_j)|\ge \frac{\eta^2 k}{10^7\Omega^*}$ then we shall embed $\mathcal P_{i,j}$, otherwise we do not embed  $\mathcal P_{i,j}$ in this step. 
So assume we decided to embed $\mathcal P_{i,j}$. Recall that the total order of the subshrubs in this set
 is at most $\tau k$.
Using the same argument as for Claim~\ref{cl:Megdes} we have 
$$\left|\bigcup\{X\cup Y\::\:(X,Y)\in\mathcal
N,\deg_{\GD}(\phi(w_{j}),X\cup
Y)>0\}\right|\le\frac{4(\Omega^*)^2}{\gamma^2}k\;.$$ Thus, there exists a subpair
$(X_1',X_2')\in \mathcal N_j$ of some $(X_1,X_2)\in \mathcal N$ with
\begin{equation}\label{eq:SIMcalc}
\frac{|X_1'|}{|X_1|}\ge
\frac{\tfrac{\eta^2}{10^7\Omega^*}k}{\tfrac{4(\Omega^*)^2}{\gamma^2}k}\ge \frac{\gamma^2 \eta^2}{10^8 (\Omega^*)^3}\;.
\end{equation}
In particular, $(X_1',X_2')$ forms a
$\frac{2\cdot 10^8\epsilon_1(\Omega^*)^3}{\gamma^2\eta^2}$-regular pair of
density at least $d_1/2$ by Fact~\ref{fact:BigSubpairsInRegularPairs}.
We use Lemma~\ref{lem:embed:BALANCED} to embed $\mathcal P_{i,j}$ in $\M_\PARAMETERPASSING{L}{lem:embed:BALANCED}:=\{(X_1',X_2')\}$. The family
$\{f_{CD}\}_\PARAMETERPASSING{L}{lem:embed:BALANCED}$ consists of a single
number $f_{(X_1',X_2')}$ which is the discrepancy of $\bigcup_{p<j}\phi (\mathcal P_{i,p})$ with respect to $(X_1,X_2)$. This guarantees
that~\eqref{eq:INDbalanced} is preserved. This finishes the $j$-th step.
We repeat this step until $j=|W_{C,i}|$, then we go to the next stage.

Denote by $\mathcal Q_i$ the set of all $P\in \mathcal P_{i}$ that have not been embedded in the first stage. Note that for each $Q\in \mathcal Q_i$, with $Q\in \mathcal P_{i,j}$, say, and for each $(X_1,X_2)\in\mathcal N$ there is a $b_{(X_1,X_2)}\in\{1,2\}$ such that for
\[
O_j:=\bigcup_{(X_1,X_2)\in\mathcal N}\left(X_{b_{(X_1,X_2)}}\cap \neighbour_{\Gblack}(\phi(w_{j})\right)\setminus
\left(\phi(\bigcup_{p<j}\mathcal P_{i,p})\cup \bigcup_{\ell <i}\phi (X_\ell)\right)
\]
we have that
\begin{equation}\label{boli}
|O_j|<\frac{\eta^2 k}{10^7\Omega^*}.
\end{equation}
The fact that $O_j$ is small implies that there is an $\mathcal N$-cover such that
the $\Gblack$-neighbourhood of $w_j$ restricted to this cover is essentially
exhausted by the image of $T'$.

\medskip
In the second stage, we shall embed some of the peripheral subshrubs of
$\mathcal Q_i$. They will be mapped in an unbalanced way to $\mathcal
N$. We will do this in steps $j=1, \ldots ,|W_{C,i}|$, and denote by
$\mathcal R_j$ the set of those $\mathcal P\subseteq \mathcal Q_i$ embedded
until step $j$. At step $j$, we decide to embed $\mathcal P_{i,j}$ if $\mathcal
P_{i,j}\subseteq \mathcal Q_i$ and

\begin{equation}\label{lacumparsita}
\deg_{\Gblack}\Big(\phi(w_{j}),V(\mathcal N)\setminus \phi(\bigcup\mathcal P_{i}\sm \mathcal Q_i)\Big) - |\bigcup\mathcal R_{ j-1}|\ge\frac{\eta^2 k}{10^6}\;.
\end{equation}

Let $$\tilde{\mathcal N}:=\left\{(X,Y)\in\mathcal N\::\: |(X\cup Y)\cap
Z_{<i}|<\frac{\gamma^2\eta^2}{10^9(\Omega^*)^2}|X|\right\}\;.$$ 

As by~\eqref{wheretherootsgo} we know that $w_{j}$
was embedded in $V_2\sm F^{(1)}_i$, we have
\begin{equation}\label{wasverlorengeht}
\deg_{\Gblack}\left(\phi(w_{j}),V(\mathcal N\sm \tilde{\mathcal N})\right)
\le
\frac{2\cdot 10^9 (\Omega^*)^2}{\gamma^2\eta^2} \cdot\frac{\delta k}{8}
\le 
\frac{\eta^2}{10^7} k\;.
\end{equation}

Using~\eqref{boli},~\eqref{lacumparsita} and~\eqref{wasverlorengeht}, similar calculations as in~\eqref{eq:SIMcalc} show the existence of a pair
$(X,Y)\in\tilde{\mathcal N}$  with
$$\deg_{\Gblack}\left(\phi(w_{j}),(X\cup Y)\setminus ( O_j\cup \phi(\bigcup\mathcal P_{i}\sm \mathcal Q_i))\right) - \left|(X\cup Y)\cap \phi(\bigcup\mathcal R_{ j-1})\right|\ge
\frac{\gamma^2\eta^2}{10^8(\Omega^*)^2}|X\cup Y|\;.$$
Then by the definition
of $\tilde{\mathcal N}$, and setting $Z_{<i}^+:=\ghost_{\mathfrak d}(Z_{<i})$ we get that
 \begin{align*}
 \deg_{\Gblack} & \left(\phi(w_{j}), (X\cup Y)\setminus \left(Z_{<i}^+\cup O_j \cup \phi\left(\bigcup\mathcal P_{i}\sm \mathcal Q_i\right)\right)\right) - \left|(X\cup Y)\cap \phi(\bigcup\mathcal R_{ j-1})\right|\\ &
 \ge \frac{\gamma^2\eta^2}{10^9(\Omega^*)^2}|X\cup Y|. 
\end{align*}
 By the definition of $O_j$, all of the degree counted here goes to one side of the matching edge $(X,Y)$, say to $X$. So
 \begin{align}
 \deg_{\Gblack}\left(\phi(w_{j}),X\setminus \left(Z_{<i}^+\cup \phi\left(\bigcup\mathcal
 P_{i}\sm \mathcal Q_i\cup \bigcup\mathcal R_{j-1}\right)\right)\right) - \left| Y\cap \phi(\bigcup\mathcal R_{ j-1})\right| & \ge 
 \frac{\gamma^2\eta^2}{10^9(\Omega^*)^2}|X|\label{aha1}
 \\ & \ge 
12\frac{\eps_1}{d_1}|X|+\tau k.\label{aha2}
\end{align}
We claim that furthermore,
\begin{equation}
\left|Y\setminus \left(Z_{<i}^+\cup\phi(\bigcup\mathcal P_{i}\sm \mathcal Q_i\cup \bigcup\mathcal R_{j-1})\right)\right|
\ge
\frac{\gamma^2\eta^2}{ 10^{10}(\Omega^*)^2}|Y|
\ge
12\frac{\eps_1}{d_1}|Y|+\tau k\;.\label{jajaja}
\end{equation}
Indeed, otherwise we get by~\eqref{aha1} that
\[
\left|X\setminus \left(Z^+_{<i}\cup\phi(\bigcup\mathcal P_{i}\sm \mathcal Q_i)\right)\right|
>
 \left|Y\setminus \left(Z^+_{<i}\cup\phi(\bigcup\mathcal P_{i}\sm \mathcal Q_i)\right)\right| + \frac{\gamma^2\eta^2}{ 10^{10}(\Omega^*)^2}|X|,
\]
which  is impossible by~\eqref{eq:INDbalanced} and since $|X|\geq \mu_1 k$.

Hence, by~\eqref{aha2} and~\eqref{jajaja}, we can  embed $\mathcal P_{i,j}$ into the unoccupied part $(X,Y)$ using
Lemma~\ref{lem:embed:regular} repeatedly.\footnote{Recall that the
total order of $\mathcal P_{i,j}$ is at most $\tau k$.}

\medskip

Note that  if some $\mathcal
P_{i,j}$ has not been embedded in either of the two stages,  then the vertex
$w_{j}$ must have a somewhat insufficient degree in $\mathcal N$. More precisely, employing~\eqref{lacumparsita} we see that
$\deg_{\Gblack}(\phi(w_{j}),V(\mathcal N))-|\phi(X_i)\cap V(\mathcal
N)|<\frac{\eta^2 k}{10^6}$. Combined with~\eqref{COND:D8:7}, we find that $$\deg_{\GD}(\phi(w_j),V_3)\ge h_1-|\phi(X_i)\cap V(\mathcal N)|-\frac{\eta^2
k}{10^6}\;,$$
in other words, \eqref{youknowi'mbad} holds for $i$. 

This finishes step $i$ of the embedding procedure. 
Recall that the sets $V_3$ and $V(\mathcal N)$ are
disjoint. Hence, by~\eqref{eatmoreseeds8} and~\eqref{wheretherootsgo}, the
principal subshrubs are the only parts of $T'$ that were embedded in $V_3$ (and
possibly elsewhere).
Thus, using~\eqref{camundongo}, we see that~\eqref{Didisjunkt8},~\eqref{Xifastdisjunkt8}  and~\eqref{putitintocolour18} are satisfied for $i$. Also, by~\eqref{eq:greatlyduplicated},~\eqref{thereservationfortheroots8} holds for $i$.

\bigskip

After having completed the inductive procedure, we still have to embed some
peripheral subshrubs.
Let us take sequentially those $P\in\mathcal P$  which  were not embedded. Say $w$ is the parent of $P$.
By~\eqref{youknowi'mbad} we have 
$$\deg_{\GD}(\phi(w),V_3\setminus \textrm{Im}(\phi))\ge
h_1-| \textrm{Im}(\phi)\cap V(\mathcal N)|-| \textrm{Im}(\phi)\cap V_3|-\frac{\eta^2
k}{10^6}\geByRef{eq:takeaway} \frac{\eta^2 k}{10^6}\;.$$
An application of Lemma~\ref{lem:embedStoch:DIAMOND7} in which
$Y_{1,\PARAMETERPASSING{L}{lem:embedStoch:DIAMOND7}}:=V_3$,
$Y_{2,\PARAMETERPASSING{L}{lem:embedStoch:DIAMOND7}}:=V_4$,
$U_\PARAMETERPASSING{L}{lem:embedStoch:DIAMOND7}:=\textrm{Im}(\phi)$, $U^*_\PARAMETERPASSING{L}{lem:embedStoch:DIAMOND7}:=\neighbour_{\GD}(\phi(w))\cap
V_3\setminus \textrm{Im}(\phi)$, and
$\{P_1,\ldots,P_L\}_\PARAMETERPASSING{L}{lem:embedStoch:DIAMOND7}:=\emptyset$
gives an embedding of $P$ in $V_3\cup V_4\subset \colouringp{1}$.

By conditions~\eqref{eatmoreseeds8},~\eqref{wheretherootsgo},~\eqref{wherethefruitgoes} and~\eqref{putitintocolour18}  we have thus found the desired embedding for $T'$.
\end{proof}

\begin{lemma}\label{lem:embed:heart1}
Suppose we are in Setting~\ref{commonsetting} and~\ref{settingsplitting}, and
that the sets $V_0$ and $V_1$ witness
Preconfiguration~$\mathbf{(\heartsuit1)}(2\eta^3 k/10^3,h)$. Suppose that
$U\subset \colouringp{0}\cup\colouringp{1}$. Suppose that
$\{x_j\}_{j=1}^\ell\subset V_0$ and $\{y_j\}_{j=1}^{\ell'}\subset V_1$ are sets of distinct vertices.\footnote{That is, $\{x_j\}_{j=1}^\ell\cup\{y_j\}_{j=1}^{\ell'}=\ell+\ell'$.} Let $\{(T_j,r_j)\}_{j=1}^\ell$ and
$\{(T'_j,r'_j)\}_{j=1}^{\ell'}$ be families of rooted trees such that each
component of $T_j-r_j$ and of $T'_j-r'_j$  has order at most $\tau k$.

If 
\begin{align}\label{eq:MALYSTROM1}
\sum_j v(T_j)&\le \frac h2-\frac{\eta^2 k}{1000}\;,\\
\label{eq:MALYSTROM2}
\sum_j v(T_j)+\sum_j v(T'_j)&\le h-\frac{\eta^2 k}{1000}\;,\mbox{and}\\
\label{eq:MALYSTROM3}
|U|+\sum_j v(T_j)+\sum_j v(T'_j)&\le k\;,
\end{align}
 then there exist  $(r_j\hookrightarrow x_j, V(T_j)\setminus\{r_j\}\hookrightarrow V(G)\setminus U)$-embeddings of $T_j$ and $(r'_j\hookrightarrow y_j, V(T'_j)\setminus\{r'_j\}\hookrightarrow V(G)\setminus U)$-embeddings of $T'_j$ in $G$, all mutually disjoint.
\end{lemma}
\begin{proof}
The embedding has three stages. In Stage~I we embed some components of $T_j-r_j$
(for all $j=1,\ldots, \ell$) in the parts of $(\M_A\cup\M_B)$-edges which are
``seen in a balanced way from $x_j$''. In Stage~II we embed the remaining
components of $T_j-r_j$. Last, in Stage~III we embed all the components
$T'_j-r'_j$ (for all $j=1,\ldots, \ell'$).

Let us first give a bound on the total size of $(\M_A\cup\M_B)$-vertices $C\in\V(\M_A\cup\M_B)$, $C\subset \bigcup \clusters$ seen from a given vertex via edges of $\GD$. This bound will be used repeatedly.
\begin{claim}\label{cl:PROCH}
Let $v\in V(G)$. Then for  $\mathcal U:=\{C\in\V(\M_A\cup \M_B)\::\: C\subset\bigcup \clusters, \deg_{\GD}(x, C)>0\}$ we have 
\begin{align}\label{eq:PROCH1}
 \left|\bigcup \mathcal U\right|&\le  \frac{2(\Omega^*)^2 k}{\gamma^2}\;,\mbox{and}\\
\label{eq:PROCH2}
|\mathcal U|&\le  \frac{2(\Omega^*)^2k}{\gamma^2\pi\clustersize}\;.
\end{align}
\end{claim}
\begin{proof}[Proof of Claim~\ref{cl:PROCH}]
Let $\mathbf U\subset \clusters$ be the set of  those clusters which intersect $\neighbour_{\GD}(x_j)$. Using the same argument as in the proof of Claim~\ref{cl:Megdes} we get that $|\bigcup\mathbf U|\le \frac{2(\Omega^*)^2 k}{\gamma^2}$, i.e.~\eqref{eq:PROCH1} holds. Also,~\eqref{eq:PROCH2} follows since $\M_A\cup\M_B$ is $(\epsilon,d,\pi\clustersize)$-\semiregulars.
\end{proof}

\medskip\noindent\underline{Stage I:} We proceed inductively for $j=1,\ldots, \ell$. Suppose that we embedded some components $\mathcal F_1,\ldots, \mathcal F_{j-1}$ of the forests $T_1-r_1,\ldots,T_{j-1}-r_{j-1}$. We write $F_{j-1}$ for the partial images of this embedding. We inductively assume that
\begin{equation}\label{eq:Fjbal}
\mbox{$F_{j-1}$ is $\tau k$-balanced w.r.t.\ $\M_A\cup\M_B$.} 
\end{equation}

For each $(A,B)\in\M_A\cup \M_B$ with $\deg_{\GD}(x_j,(A\cup B)\setminus \smallatoms)>0$ take a subpair $(A',B')$,
$$A'\subset(A\cap\neighbour_{\GD\cup\Gcapt}(x_j))\colouringpI{2}\setminus F_{j-1} \quad\mbox{and}\quad B'\subset(B\cap\neighbour_{\GD\cup\Gcapt}(x_j))\colouringpI{2}\setminus F_{j-1}\;,$$
such that 
$$|A'|=|B'|=\min\big\{ |(A\cap\neighbour_{\GD\cup\Gcapt}(x_j))\colouringpI{2}\setminus F_{j-1}|, |(B\cap\neighbour_{\GD\cup\Gcapt}(x_j))\colouringpI{2}\setminus F_{j-1}|\big\}\;.$$
These pairs comprise a \semiregular matching $\mathcal N_j$. (Pairs
$(A,B)\in\M_A\cup \M_B$ with 
$$\deg_{\GD}(x_j,(A\cup B)\setminus
\smallatoms)=0$$
are not considered for the construction of $\mathcal N_j$.) 

Let
$\M_j:=\{(A',B')\in\mathcal N_j\::\: |A'|>\alpha |A|\}$, for 
 $$\alpha:=\frac{\eta^3\gamma^2}{10^{10}(\Omega^*)^2}.$$
  By
Fact~\ref{fact:BigSubpairsInRegularPairs} $\M_j$ is a $(2\epsilon/\alpha, d/2,
\alpha\pi \clustersize)$-\semiregular matching.
\begin{claim}\label{cl:MjVSNj}
We have $|V(\mathcal M_j)|\ge |V(\mathcal N_j)|-\frac{\eta^3 k}{10^9}$. 
\end{claim}
\begin{proof}[Proof of Claim~\ref{cl:MjVSNj}]
By~\eqref{eq:PROCH1}, and by Property~\ref{commonsetting3} of
Setting~\ref{commonsetting}, we have 
$$|V(\mathcal M_j)|\ge |V(\mathcal N_j)|-\alpha\cdot 2\cdot \frac{2(\Omega^*)^2 k}{\gamma^2}\;.$$
\end{proof}
Let $\mathcal F_j$ be a maximal set of components of $T_j-r_j$ for which
\begin{equation}\label{eq:maxFj}
v(\mathcal F_j)\le |V(\M_j)|-\frac{\eta^3 k}{10^9}\;.
\end{equation}
Observe that if $\mathcal F_j$ does not contain all the components of $T_j-r_j$ then
\begin{equation}\label{eq:natrave}
v(\mathcal F_j)\ge |V(\M_j)|-\frac{\eta^3 k}{10^9}-\tau k\ge |V(\M_j)|-\frac{2\eta^3 k}{10^9}\;.
\end{equation}

Lemma~\ref{lem:embed:BALANCED} yields an embedding of $\mathcal F_j$ in $\M_j$. Further the lemma together with the induction hypothesis~\eqref{eq:Fjbal} guarantees that the embedding can be chosen so that the new image set $F_{j}$ is $\tau k$-balanced w.r.t.\ $\M_A\cup\M_B$. We fix this embedding, thus ensuring~\eqref{eq:Fjbal} for step $i$. If $\mathcal F_j$ does not contain all the components of $T_j-r_j$ then~\eqref{eq:natrave} gives
\begin{equation}\label{eq:OnLawn}
|V(\M_j)\setminus F_j| \le \frac{2\eta^3 k}{10^9}\;.
\end{equation}

\medskip\noindent\underline{After Stage I:} Let $\mathcal N^*$ be a maximal
\semiregular matching contained in $(\M_A\cup\M_B)\colouringpI{2}$ which avoids
$F_\ell$. We need two auxiliary claims.
\begin{claim}\label{cl:ztratimeMaloW}
We have 
$$\maxdeg_{\GD}\left(V_0\cup V_1,  V(\M_A\cup\M_B)\colouringpI{2}\setminus(V(\mathcal N^*)\cup F_{\ell}\cup\smallatoms)\right)<\frac{\eta^3 k}{10^9}\;.$$
\end{claim}
\begin{proof}[Proof of Claim~\ref{cl:ztratimeMaloW}]
Let us consider an arbitrary vertex $x\in V_0\cup V_1$. By~\eqref{eq:PROCH2} the
number of $(\M_A\cup\M_B)$-vertices $C\subset \bigcup\clusters$ for which
$\deg_{\GD}(x,C)>0$ is at most $\frac{2(\Omega^*)^2k}{\gamma^2\pi\clustersize}$.

Due to~\eqref{eq:Fjbal}, we have
 for each $(\M_A\cup\M_B)$-edge $(A,B)$  that 
\begin{equation}\label{eq:sumicek}
\left|(A\cup B)\colouringpI{2}\setminus (V(\mathcal N^*)\cup F_\ell)\right|\le \tau k\;.
\end{equation}
                                                                                                                                                                                     
Thus summing~\eqref{eq:sumicek} over all $(\M_A\cup\M_B)$-edges $(A,B)$ with $\deg_{\GD}(x,(A\cup B)\setminus \smallatoms)>0$ we get
$$\deg_{\GD}\left(x,  V(\M_A\cup\M_B)\colouringpI{2}\setminus(V(\mathcal
N^*)\cup F_{\ell}\cup\smallatoms)\right)\le
\frac{4(\Omega^*)^2k}{\gamma^2\pi\clustersize} \cdot \tau k\;.$$ The claim now
follows by~\eqref{eq:KONST}.
\end{proof}
\begin{claim}\label{cl:OneSideSee}
Let $j\in[\ell]$ be such that $\mathcal F_j$ does not consist of all the components of $T_j-r_j$. Then there exists an $\mathcal N^*$-cover $\mathcal X_j$ such that $\deg_{\GD}\left(x_j,\bigcup \mathcal X_j\right)\le \frac{3 \eta^3k}{10^9}$.
\end{claim}
\begin{proof}[Proof of Claim~\ref{cl:OneSideSee}]
First, we define an $(\M_A\cup\M_B)$-cover $\mathcal R_j$ as follows. For an $(\M_A\cup\M_B)$-edge $(A,B)$ let $\mathcal R_j$ contain $A$ if
$$|(A\cap\neighbour_{\GD\cup\Gcapt}(x_j))\colouringpI{2}\setminus F_{j-1}|\le |(B\cap\neighbour_{\GD\cup\Gcapt}(x_j))\colouringpI{2}\setminus F_{j-1}|\;,$$ and $B$ otherwise. Observe that by the definition of $\mathcal N_j$, we have
\begin{equation}\label{eq:nulanatrave}
\deg_{\GD}\left(x_j,\bigcup \mathcal R_j\setminus V(\mathcal N_j)\right)=0\;.
\end{equation}
Also, we have $V(\mathcal N^*)\cap \bigcup \mathcal R_j\cap V(\M_j)\subset V(\mathcal N^*)\cap V(\M_j)\subset V(\M_j)\setminus F_j$. In particular,~\eqref{eq:OnLawn} gives that
\begin{equation}\label{eq:wac}
\left|V(\mathcal N^*)\cap \bigcup \mathcal R_j\cap V(\M_j)\right| \le \frac{2\eta^3 k}{10^9}\;.
\end{equation}

Let $\mathcal X_j$ be the restriction of $\mathcal R_j$ to $\mathcal N^*$. We then have 
\begin{align*}
\deg_{\GD}\left(x_j,\bigcup \mathcal X_j\right)&=\deg_{\GD}\left(x_j,V(\mathcal N^*)\cap \bigcup \mathcal R_j\right)\\
\JUSTIFY{by~\eqref{eq:nulanatrave}}&\le \deg_{\GD}\left(x_j, V(\mathcal N^*)\cap\bigcup \mathcal R_j\cap V(\mathcal M_j)\right)+\deg_{\GD}\left(x_j, V(\mathcal N_j)\setminus V(\M_j)\right)\\
\JUSTIFY{by~\eqref{eq:wac}, Claim~\ref{cl:MjVSNj}}&\le \frac{3\eta^3 k}{10^9}\;.
\end{align*}
\end{proof}

For every $j\in[\ell]$ we define $\mathcal N^*_j\subset \mathcal N^*$ as those $\mathcal N^*$-edges $(A,B)$ for which we have $$\big((A\cup B)\setminus\bigcup \mathcal X_j\big)\cap\smallatoms=\emptyset\;.$$

\medskip\noindent\underline{Stage II:} We shall inductively for $j=1,\ldots,
\ell$ embed those components of $T_j-r_j$ that are not included in $\mathcal
F_j$; let us denote the set of these components by $\mathcal K_j$.
There is nothing to do when $\mathcal K_j=\emptyset$, so let us assume otherwise.

We write  $\BL:=\{C\in\clusters\::\:C\subset\largevertices{\eta}{k}{G}\}$. Let
$K\in \mathcal K_j$ be a component that has not been embedded yet. We write $U'$
for the total image of what has been embedded (in Stage~I, and Stage~II so far), combined with $U$. We claim that $x_j$ has a substantial degree into one of four specific vertex sets.
\begin{claim}\label{cl:OneInFour}
At least one of the following four cases occurs.
\begin{itemize}
 \item[$\mathbf{(U1)}$] $\deg_{\GD}\left(x_j, V(\mathcal
 N^*_j)\setminus \bigcup \mathcal X_j\right)-|U'\cap
  V(\mathcal N^*_j)|\ge \frac{\eta^2
 k}{10^4}$,
 \item[$\mathbf{(U2)}$] $\deg_{\GD}\left(x_j, \smallatoms\setminus
 U'\right)\ge \frac{\eta^2 k}{10^4}$,
 \item[$\mathbf{(U3)}$] $\deg_{\Gcapt}\left(x_j, V(\Gexp)\setminus
 U'\right)\ge \frac{\eta^2 k}{10^4}$,
 \item[$\mathbf{(U4)}$] $\deg_{\GD}\left(x_j, \bigcup \BL\setminus (L_\#\cup
 V(\Gexp)\cup U')\right)\ge \frac{\eta^2 k}{10^4}$.
\end{itemize}
\end{claim}
\begin{proof}
Write $U'':=(U')\colouringpI{2}=U'\sm U$. By~\eqref{COND:P1:3}, we have 
\begin{align*}
\frac h2&\le  \deg_{\Gcapt}(x_j,\Vgood\colouringpI{2})\\
&\le 
\deg_{\GD}\left(x_j,V(\mathcal N^*_j)\colouringpI{2}  \setminus  \bigcup
\mathcal X_j\right)
+ \deg_{\GD}\left(x_j,\smallatoms\colouringpI{2}\setminus
(V(\mathcal N^*_j)\cup V(\Gexp) \cup \bigcup
\mathcal X_j)\right)\\
&~~~+
\deg_{\Gcapt}\left(x_j,V(\Gexp)\colouringpI{2}\right)+
\deg_{\GD}\left(x_j,\bigcup\BL\colouringpI{2}\setminus (L_\#\cup V(\Gexp)\cup
V(\mathcal N^*_j))\right)\\
&~~~+\deg_{\GD}\left(x_j,V(\M_A\cup \M_B)\colouringpI{2}\setminus (V(\mathcal
N^*)\cup\smallatoms)\right)+ \deg_{\GD}\left(x_j,\bigcup\mathcal X_j\right)\\
\JUSTIFY{by C\ref{cl:ztratimeMaloW}, C\ref{cl:OneSideSee}}
&\le \deg_{\GD}\left(x_j,V(\mathcal N^*_j)\setminus
 \bigcup \mathcal X_j\right)-\left|U'\cap V(\mathcal N^*_j)\right|\\
&~~~+\deg_{\GD}\left(x_j,\smallatoms\colouringpI{2}\setminus (V(\mathcal
N^*_j)\cup \bigcup \mathcal X_j\cup U'')\right)+
\deg_{\Gcapt}\left(x_j,V(\Gexp)\colouringpI{2}\setminus  U''\right)\\
&~~~+\deg_{\GD}\left(x_j,\bigcup\BL\colouringpI{2}\setminus (L_\#\cup
V(\Gexp)\cup  V(\mathcal N^*_j)\cup U'')\right)\\
&~~~+\frac{4\eta^3 k}{10^9}+|U''|\;.
\end{align*}
The claim follows since $|U''|\le \frac h2-\frac{\eta^2 k}{1000}$ by~\eqref{eq:MALYSTROM1}.
\end{proof}

We now briefly describe how to embed $K$ in each of the cases
$\mathbf{(U1)}$--$\mathbf{(U4)}$.

\begin{itemize}
 \item In case~$\mathbf{(U1)}$ 
recall that each $(\M_A\cup \M_B)$-edge contains at most one $\mathcal
N^*_j$-edge. Thus by~\eqref{eq:PROCH1} we get that there is an $(\M_A\cup \M_B)$-edge $(A,B)$ with
\begin{equation}\label{eq:smethod}
 \deg_{\GD}\left(x_j, (V(\mathcal N^*_j)\cap (A\cup B))\setminus \bigcup
 \mathcal X_j\right)-|V(\mathcal N^*_j)\cap U'\cap (A\cup B)|\ge \frac{\eta^2
 k}{10^4}\cdot \frac{\gamma^2}{2(\Omega^*)^2 k}\cdot |A|. \ \ \ \ 
\end{equation}
Let us fix this edge $(A,B)$, and let $(A',B')$ be the corresponding edge in
$\mathcal N^*_j$. Suppose without loss of generality that $B\in\mathcal X_j$. We
can now embed $K$ in $(A',B')$ using Lemma~\ref{lem:embed:regular}
with the following input: $C_\PARAMETERPASSING{L}{lem:embed:regular}:=A'
,D_\PARAMETERPASSING{L}{lem:embed:regular}:=B',
X_\PARAMETERPASSING{L}{lem:embed:regular}:=A'\setminus U',
X^*_\PARAMETERPASSING{L}{lem:embed:regular}:=\neighbour_{\GD}(x_j, A'\setminus
U'), 
Y_\PARAMETERPASSING{L}{lem:embed:regular}:=B'\setminus U',
\epsilon_\PARAMETERPASSING{L}{lem:embed:regular}:=\frac{8\cdot
10^4(\Omega^*)^2\epsilon}{\gamma^2\eta^2},
\beta_\PARAMETERPASSING{L}{lem:embed:regular}:=d/6$.
 With the help of~\eqref{eq:smethod}, we get that $$\min\{X_\PARAMETERPASSING{L}{lem:embed:regular},Y_\PARAMETERPASSING{L}{lem:embed:regular}\}\ge
|X^*_\PARAMETERPASSING{L}{lem:embed:regular}|\ge \frac
{\gamma^2\eta^2|A|}{4\cdot 10^4(\Omega^*)^2} \ge 4\frac{\eps_\PARAMETERPASSING{L}{lem:embed:regular}}{\beta_\PARAMETERPASSING{L}{lem:embed:regular}}|A'|\;.$$

 \item In Case~$\mathbf{(U2)}$ we embed $K$ using
 Lemma~\ref{lem:embed:avoidingFOREST} with the following input:
 $\epsilon_\PARAMETERPASSING{L}{lem:embed:avoidingFOREST}:= \epsilon'$, $U_\PARAMETERPASSING{L}{lem:embed:avoidingFOREST}:= U'$,
 $U^*_\PARAMETERPASSING{L}{lem:embed:avoidingFOREST}:= \neighbour_{\GD}(x_j,
 \smallatoms\setminus U')$, $\ell:=1$.

 \item In Case~$\mathbf{(U3)}$ we embed $K$ using Lemma~\ref{lem:embed:greyFOREST}
 with the following input: $H_\PARAMETERPASSING{L}{lem:embed:greyFOREST}:=\Gexp$, 
 $V_{1,\PARAMETERPASSING{L}{lem:embed:greyFOREST}}:=V_{2,\PARAMETERPASSING{L}{lem:embed:greyFOREST}}:=V(\Gexp)$,
 $U_\PARAMETERPASSING{L}{lem:embed:greyFOREST}:=U'$,
 $U^*_\PARAMETERPASSING{L}{lem:embed:greyFOREST}:= \neighbour_{\Gexp}(x_j,
 V(\Gexp)\setminus U')$, $Q_\PARAMETERPASSING{L}{lem:embed:greyFOREST}:=1$,
$\zeta_\PARAMETERPASSING{L}{lem:embed:greyFOREST}:= \rho$,
$\ell_\PARAMETERPASSING{L}{lem:embed:greyFOREST}:=1$.

 \item In Case~$\mathbf{(U4)}$ we proceed as follows. As
 $\deg_{\GD}(x_j,\WantiC)<\frac{\eta^2k}{10^5}$ (cf.
 Definition~\ref{def:heart1}), we have $$\deg_{\GD}\left(x_j, \bigcup
 \BL\setminus (L_\#\cup V(\Gexp)\cup \WantiC\cup U')\right)\ge \frac{2
 \eta^2 k}{10^5}\;.
 $$ 
As for~\eqref{eq:smethod}, we use~\eqref{eq:PROCH1} to find a cluster $A\in\BL$ with
\begin{equation}\label{barrioyungay}
 \deg_{\GD}\left(x_j, A\setminus (L_\#\cup V(\Gexp)\cup \WantiC \cup
 U')\right) \ge \frac{2\eta^2 k}{10^5}\cdot \frac{\gamma^2}{2(\Omega^*)^2
 k}\cdot |A|= \frac{\eta^2 \gamma^2}{10^5(\Omega^*)^2
 }\cdot |A|.
\end{equation}
Recall that by the definition of $L_\#$ and $ \WantiC$ (see~\eqref{eq:defLsharp} and~\eqref{eq:defWantiC}), we have that 
$$\mindeg_{\Gcapt}(A\setminus (L_\#\cup \WantiC), V(G)\setminus
\HugeVertices)\ge (1+\frac{4\eta}{5})k\;.$$
Thus, for the set $$A^*:=(\neighbour_{\GD}(x_j)\cap
A)\setminus (L_\#\cup V(\Gexp)\cup \WantiC \cup U')$$
at least one of the following subcases must occur:
\begin{enumerate}
 \item[$\mathbf{(U4a)}$] For at least $\frac12|A^*|$ vertices $v\in A^*$ we have
 $\deg_{\Gcapt}(v,\smallatoms\setminus U')\ge \frac{2\eta k}5$.
  \item[$\mathbf{(U4b)}$] For at least $\frac12|A^*|$ vertices $v\in A^*$ we have
 $\deg_{\Gblack}(v,\bigcup\clusters\setminus U')\ge \frac{2\eta k}5$.
\end{enumerate}
In case $\mathbf{(U4a)}$ we embed $K$ using
Lemma~\ref{lem:embed:avoidingFOREST}. The details are very similar to
$\mathbf{(U2)}$. As for case $\mathbf{(U2b)}$,
let us take an arbitrary vertex $v\in A^*$ with
$\deg_{\Gblack}(v,\bigcup\clusters\setminus U')\ge \frac{2\eta k}5$. In
particular, using~\eqref{eq:PROCH1}, we find a cluster $B\in \clusters$ with 
\begin{equation}\label{stantonmoore}
\deg_{\Gblack}(v,B\setminus
U')\ge \frac{\gamma^2\eta}{10(\Omega^*)^2}|B|\;.
\end{equation}
 Map the root $r_K$ of $K$ to
$v$ and embed $K-r_K$ in $(A,B)$ using
Lemma~\ref{lem:embed:regular}\footnote{Lemma~\ref{lem:embed:regular} deals with
embedding a single tree in a regular pair, whereas $K-r_K$ has several
components. We therefore apply the lemma repeatedly for each component.} with
the following input:
$C_\PARAMETERPASSING{L}{lem:embed:regular}:= B,
D_\PARAMETERPASSING{L}{lem:embed:regular}:= A,
X_\PARAMETERPASSING{L}{lem:embed:regular}:= B\setminus U',
Y_\PARAMETERPASSING{L}{lem:embed:regular}:= A\setminus U',
X^*_\PARAMETERPASSING{L}{lem:embed:regular}:= \neighbour_{\Gblack}(v, B\setminus
U'),
\beta_\PARAMETERPASSING{L}{lem:embed:regular}:= \gamma^2\eta/(10(\Omega^*)^2),
\epsilon_\PARAMETERPASSING{L}{lem:embed:regular}:= \epsilon'$. By~\eqref{barrioyungay} and~\eqref{stantonmoore} we see that $X_\PARAMETERPASSING{L}{lem:embed:regular},
Y_\PARAMETERPASSING{L}{lem:embed:regular}$ and $
X^*_\PARAMETERPASSING{L}{lem:embed:regular}$ are large enough.
\end{itemize}

\medskip\noindent\underline{Stage III:} In this stage we embed the trees
$\{T'_j\}_{j=1}^{\ell'}$. The embedding techniques are as in Stage~II. The cover
$\mathcal F'$ from Definition~\ref{def:heart1} plays the same role as the covers
$\mathcal X_j$ in Stage~II. Observe that $\mathcal F'$ is universal whereas the
covers $\mathcal X_j$ are specific for each vertex $x_j$. A second simplification is that in Stage~III we use
the \semiregular matching $(\M_A\cup \M_B)\colouringpI{2}$ for embedding (in a
counterpart of $\mathbf{(U1)}$) instead of $\mathcal N^*_j$.

Again we proceed inductively for $j=1,\ldots,
\ell$ with embedding the components of $T'_j-r'_j$, which we denote by $\mathcal
K'_j$.
 Let
$K\in \mathcal K'_j$ be a component that has not been embedded yet. We write
$U'$ for the total image of what has been embedded (in Stage~I,~II, and
Stage~III so far), combined with $U$ and let $U''=U'\cap \colouringp{2}$. We
claim that $y_j$ has a substantial degree into one of four specific vertex sets.
\begin{claim}\label{cl:OneInFour'}
At least one of the following four cases occurs.
\begin{itemize}
 \item[$\mathbf{(U1')}$] $\deg_{\GD}\left(y_j, V((\M_A\cup
 \M_B)\colouringpI{2})\setminus (\smallatoms\cup \bigcup \mathcal
 F')\right)\\-\left|U''\cap\left(\bigcup \mathcal F'\cup ( V((\M_A\cup
 \M_B)\colouringpI{2})\setminus \smallatoms\right)\right|\ge \frac{\eta^2 k}{10^4}$,
 \item[$\mathbf{(U2')}$] $\deg_{\GD}\left(y_j, \smallatoms\setminus
 U'\right)\ge \frac{\eta^2 k}{10^4}$,
 \item[$\mathbf{(U3')}$] $\deg_{\Gcapt}\left(y_j, V(\Gexp)\setminus
 U'\right)\ge \frac{\eta^2 k}{10^4}$,
 \item[$\mathbf{(U4')}$] $\deg_{\GD}\left(y_j, \bigcup \BL\setminus (L_\#\cup
 V(\Gexp)\cup U')\right)\ge \frac{\eta^2 k}{10^4}$.
\end{itemize}
\end{claim}
\begin{proof}
 As $y_j\in V_1$, we have that
\begin{align*}
h&\le  \deg_{\Gcapt}(y_j,\Vgood\colouringpI{2})\\
&\le 
\deg_{\GD}\left(y_j,V((\M_A\cup \M_B)\colouringpI{2})\setminus (\smallatoms\cup
V(\Gexp)\cup \bigcup \mathcal F')\right)+
\deg_{\GD}\left(y_j,\smallatoms\colouringpI{2}\setminus (V(\Gexp)\cup  \bigcup
\mathcal F'\right)\\ &~~~+\deg_{\GD}(y_j,\bigcup \mathcal F') +
\deg_{\GD}\left(y_j,\bigcup\BL\colouringpI{2}\setminus (L_\#\cup V(\Gexp)\cup 
V(\M_A\cup \M_B)\right)\\
&~~~+\deg_{\Gcapt}\left(y_j, V(\Gexp)\colouringpI{2}\right)+
\deg_{\GD}\left(y_j, V(\M_A\cup\M_B)\colouringpI{2}\setminus
V((\M_A\cup \M_B)\colouringpI{2})\right)\\ 
\JUSTIFY{by L~\ref{lem:RestrictionSemiregularMatching}}&\le
\deg_{\GD}\left(y_j,V((\M_A\cup \M_B)\colouringpI{2})\setminus (\smallatoms\cup V(\Gexp)\cup \bigcup \mathcal
F')\right)\\ &~~~-\left|U''\cap \left(\bigcup \mathcal F'\cup  (V((\M_A\cup
\M_B)\colouringpI{2})\setminus \smallatoms)\right)\setminus V(\Gexp)\right|\\
&~~~+\deg_{\GD}\left(y_j,\smallatoms\colouringpI{2}\setminus (U''\cup
V(\Gexp)\cup \bigcup \mathcal F')\right)
+\deg_{\Gcapt}\left(y_j,V(\Gexp)\colouringpI{2}\setminus U''\right)\\
&~~~+\deg_{\GD}\left(y_j,\bigcup\BL\colouringpI{2}\setminus (L_\#\cup
V(\Gexp)\cup V(\M_A\cup \M_B)\cup U'')\right)+\frac{2\eta^3
k}{10^3}+\frac {\eta^2k}{10^5}+|U''|\;.
\end{align*}
The claim follows since $|U''|\le \sum_{j}|T_j|+\sum_{j}|T_j'|\le h-\frac
{\eta^2k}{1000}$.

\end{proof}

Cases~$\mathbf{(U1')}$--$\mathbf{(U4')}$ are treated analogously as
Cases~$\mathbf{(U1)}$--$\mathbf{(U4)}$.

\end{proof}

\begin{lemma}\label{lem:embed:heart2}
Suppose we are in Setting~\ref{commonsetting} and~\ref{settingsplitting}, and that the sets $V_0$ and $V_1$ witness Preconfiguration~$\mathbf{(\heartsuit2)}(h)$. Suppose that $U\subset \colouringp{0}\cup\colouringp{1}$ and $|U|\le k$. Suppose that $\{x_j\}_{j=1}^\ell\subset V_0\cup V_1$ are distinct vertices. Let $\{(T_j,r_j)\}_{j=1}^\ell$ be a family of rooted trees such that each component of $T_j-r_j$ has order at most $\tau k$.

If $\sum_j v(T_j)\le h-\eta^2 k/1000$ and $|U|+\sum_j v(T_j)\le k$ then there exist disjoint $(r_j\hookrightarrow x_j, V(T_j)\setminus\{r_j\}\hookrightarrow V(G)\setminus U)$-embeddings of $T_j$ in $G$.
\end{lemma}
\begin{proof}
The proof is contained in the proof of Lemma~\ref{lem:embed:heart1}. It suffices to repeat the first two stages of the embedding process of that proof. In that setting, we use $h_\PARAMETERPASSING{L}{lem:embed:heart1}=2h$. Note that the condition $\{x_j\}\subset V_0$ in the setting of Lemma~\ref{lem:embed:heart1} gives us the same possibilities for embedding as the condition $\{x_j\}\subset V_0\cup V_1$ in the current setting (cf.~\eqref{COND:P1:3} and~\eqref{COND:P2:4}).
\end{proof}

\begin{lemma}\label{lem:embed:total68}
Suppose we are in Setting~\ref{commonsetting} and~\ref{settingsplitting}, and at least one of the following configurations occurs:
\begin{itemize}
 \item
 Configuration~$\mathbf{(\diamond6)}(\frac{\eta^3\rho^4}{10^{14}(\Omega^*)^4)},4\epsilonD,
 \frac {\gamma^3\rho}{32\Omega^*}, \frac{\eta^2\nu}{2\cdot
 10^4},\frac {3\eta^3}{2\cdot 10^3}, h)$,
 \item
 Configuration~$\mathbf{(\diamond7)}(\frac{\eta^3\gamma^3\rho}{10^{12}(\Omega^*)^4)},
 \frac {\eta\gamma}{400},4\epsilonD, \frac {\gamma^3\rho}{32\Omega^*},
 \frac{\eta^2\nu}{2\cdot 10^4},\frac {3\eta^3}{2\cdot 10^3}, h)$, or
 \item
 Configuration~$\mathbf{(\diamond8)}(\frac{\eta^4\gamma^4\rho}{10^{15}(\Omega^*)^5)},
 \frac {\eta\gamma}{400},\frac {4\epsilon}{\proporce{1}},4\epsilonD ,\frac
 d2,\frac {\gamma^3\rho}{32\Omega^*}
 ,\frac{\proporce{1}\pi\clustersize}{2k},\frac {\eta^2\nu}{2\cdot
 10^4} ,h_1,h)$.
\end{itemize}Suppose
that $(W_A,W_B,\shrubA,\shrubB)$ is a $(\tau k)$-fine partition of a rooted tree
$(T,r)$ of order $k$.
If the total order of the end shrubs is at most $h-2\frac{\eta^2 k}{10^3}$ and
the total order of the internal shrubs is at most $h_1-2\frac {\eta^2k}{10^5}$,
then $T\subset G$.
\end{lemma}
\begin{proof}
Let $T'$ be the tree induced by all the cut-vertices $W_A\cup W_B$ and all the
internal shrubs. Summing up the order of the internal shrub and the
cut-vertices, we get that $v(T')<h_1-\frac {\eta^2k}{10^5}$.  Fix an embedding
of $T'$ as in Lemma~\ref{lem:embed:skeleton67} (in configurations~$\mathbf{(\diamond6)}$ and $\mathbf{(\diamond7)}$), or as in Lemma~\ref{lem:embed:skeleton8} (in configuration~$\mathbf{(\diamond8)}$). This embedding now extends to external shrubs by Lemma~\ref{lem:embed:heart1} (in Preconfiguration~$\mathbf{(\heartsuit1)}$, which can only occur in Configuration~$\mathbf{(\diamond6)}$ and $\mathbf{(\diamond7)}$), or by Lemma~\ref{lem:embed:heart2} (in Preconfiguration~$\mathbf{(\heartsuit2)}$).  It is important to remember here that by Definition~\ref{ellfine}\eqref{Bsmall}, the total order of end shrubs in $\shrubB$ is at most half the size of the total order of end shrubs.
\end{proof}

The next lemma completely resolves Theorem~\ref{thm:main} in case of Configuration~$\mathbf{(\diamond9)}$.

\begin{lemma}\label{lem:embed9}
Suppose we are in Setting~\ref{commonsetting} and~\ref{settingsplitting}, and assume we
have Configuration~$\mathbf{(\diamond9)}(\delta,
\frac{2\eta^3}{10^3},$ $h_1,h_2,\epsilon_1, d_1, \mu_1,\epsilon_2, d_2, \mu_2)$
with $d_2>10\epsilon_2>0$, $4\cdot 10^3\le d_2\mu_2\tau k$,
$\max\{d_1,\tau/\mu_1\}\le \gamma^2\eta^2/(4\cdot 10^7(\Omega^*)^2)$,
$d_1^2/6>\epsilon_1\ge \tau/\mu_1$ and $\delta k>10^3/\tau$.

Suppose that $(W_A,W_B,\shrubA,\shrubB)$ is a $(\tau k)$-fine partition of a rooted tree $(T,r)$ of order $k$.
If the total order of the internal shrubs of $(W_A,W_B,\shrubA,\shrubB)$ is at most $h_1-\frac{\eta^2 k}{10^5}$, and the total order of the end shrubs is at most $h_2-2\frac{\eta^2 k}{10^3}$
then $T\subset G$.
\end{lemma}
\begin{proof}
Let $V_0,V_1, V_2,\mathcal N, \{Q_0^{(j)},Q_1^{(j)}\}_{j\in \mathcal Y}$ and
$\mathcal F'$ witness $\mathbf{(\diamond9)}$.
The embedding process has two stages. In the first stage we embed the \kknnaaggss and
the internal shrubs of $T$. In the second stage we embed the end shrubs. The
\kknnaaggss will be embedded in $V_0\cup V_1$, and the internal shrubs will be
embedded in $V(\mathcal N)$. Lemma~\ref{lem:embed:heart1} will be used to embed
the end shrubs.

The \kknnaaggss of $(W_A,W_B,\shrubA,\shrubB)$ are embedded in such a way that $W_A$
is embedded in $V_1$ and $W_B$ is embedded in $V_0$. Since no other part of $T$
is embedded in $V_0\cup V_1$ in the first stage, each \kknnaaggss can be embedded
greedily using the minimum degree condition arising
from the super-regularity of the pairs
$\{(Q_0^{(j)},Q_1^{(j)})\}_{j\in\mathcal Y}$ using the bound on the total order
of \kknnaaggss coming from Definition~\ref{ellfine}\eqref{few}, and using  Lemma~\ref{lem:embed:superregular} with the following input:
$\epsilon_\PARAMETERPASSING{L}{lem:embed:superregular}:= \epsilon_2$,
$d_\PARAMETERPASSING{L}{lem:embed:superregular}:= d_2$,
$\ell_\PARAMETERPASSING{L}{lem:embed:superregular}:= \mu_2k$, $U_A\cup U_B$ is
the image of the seeds $W_A\cup W_B$ embedded so far and 
$\{A_\PARAMETERPASSING{L}{lem:embed:superregular},B_\PARAMETERPASSING{L}{lem:embed:superregular}\}:=
\{Q_0^{(j)},Q_1^{(j)}\}$, where $j\in \mathcal Y$ is arbitrary for the first
\kknnaaggNOSPACE, and for all other \kknnaaggss $P$ has the property that  
\[
\neighbour_{\GD}(\phi(\parent(P)))\cap Q_1^{(j)}\setminus
U_A\neq \emptyset.
\]
The existence of such an index $j$ follows from the fact that 
\begin{equation}\label{wiesoweshalbwarum}
\phi(\parent(P))\in V_2,
\end{equation}
together with condition~\eqref{conf:D9-VtoX}. We shall
 ensure~\eqref{wiesoweshalbwarum} during our embedding of the internal shrubs, see below.

We now describe how to embed an internal shrub $T^*\in\shrubA$ whose parent
$u\in W_A$ is embedded in a vertex $x\in V_1$. Let $w\in V(T^*)$ be the unique
neighbour of a vertex from $W_A\setminus\{u\}$ (cf.\
Definition~\ref{ellfine}\eqref{2seeds}). Let $U$ be the image of the part of $T$
embedded so far. The next claim will be useful for finding a suitable $\mathcal N$-edge for
accommodating $T^*$.
\begin{claim}\label{claim:DTn}
There exists an $\mathcal N$-edge $(A,B)$, or an $\mathcal N$-edge $(B,A)$ such
that $$\min\big\{|\neighbour_{\GD}(x)\cap V_2\cap (A\setminus U)|, |B\setminus
U|\big\} \ge 100 d_1|A|+\tau k\;.$$
\end{claim}
\begin{proof}[Proof of Claim~\ref{claim:DTn}]
For the purpose of this claim we reorient $\mathcal N$ so that $V_2(\mathcal
N)\subset \bigcup \mathcal F'$.

Suppose the claim fails to be true. Then for each $(A,B)\in \mathcal N$ we have
$|\neighbour_{\GD}(x)\cap V_2\cap (A\setminus U)|<100d_1|A|+\tau k$ or $|B\setminus U|<100d_1|A|+\tau k$. In
either case we get
\begin{equation}\label{eq:HvHv}
|\neighbour_{\GD}(x)\cap V_2\cap A|-|U\cap (A\cup B)| <
100d_1|A|+\tau k\;.
\end{equation}
We write $S:=\bigcup \{V(D):D\in\DenseSpots, x\in V(D)\}$. Combining Fact~\ref{fact:sizedensespot} and Fact~\ref{fact:boundedlymanyspots} we get that 
\begin{equation}\label{eq:girls}
|S|\le \frac{2(\Omega^*)^2k}{\gamma^2}\;.
\end{equation}
Let us look at the number
\begin{equation}\label{eq:dlambda}
\lambda:= \sum_{(A,B)\in\mathcal N} \big(|\neighbour_{\GD}(x)\cap V_2\cap A|-|U\cap (A\cup B)|\big)\;.
\end{equation}
For a lower bound on $\lambda$, we write $\lambda=|\neighbour_{\GD}(x)\cap V_2|-|U\cap V(\mathcal N)|$. 
(Note that $V_2\subseteq V(\mathcal N)$ as we are in Configuration~$\mathbf{(\diamond9)}$.)
The first term is at least $h_1$ by~\eqref{conf:D9-XtoV}, while the second term is at most $h_1-\frac{\eta^2 k}{10^5}$ by the assumptions of the lemma. Thus $\lambda\ge \frac{\eta^2 k}{10^5}$.

For an upper bound on $\lambda$ we only consider those $\mathcal N$-edges
$(A,B)$ for which $\neighbour_{\GD}(x)\cap A\neq\emptyset$. In that case
$A\subset S$ (cf. \ref{commonsetting2} of Setting~\ref{commonsetting}). Thus,
since $\mathcal N$ is
$(\epsilon_1,d_1,\mu_1 k)$-\semiregular we get that
\begin{equation}\label{eq:girls2}
\left|\{(A,B)\in\mathcal N\::\: \neighbour_{\GD}\cap A\neq \emptyset\}\right|\le
\frac{|S|}{\mu_1k}\;.
\end{equation}
Thus, 
 \begin{align*}
 \lambda&\le  \sum_{(A,B)\in\mathcal N, \neighbour_{\GD}(x)\cap A\neq\emptyset} \big(|\neighbour_{\GD}(x)\cap V_2\cap A|
 -|U\cap (A\cup B) |\big)\\
\JUSTIFY{by~\eqref{eq:HvHv},~\eqref{eq:girls2}}&\le
100d_1|S|+
\frac{|S|}{\mu_1k}\tau k\\
\JUSTIFY{by~\eqref{eq:girls}}&<\frac{\eta^2 k}{10^5}\;,
\end{align*}
a contradiction. This finishes the proof of the claim.
\end{proof}
By symmetry we suppose that Claim~\ref{claim:DTn} gives an $\mathcal N$-edge
$(A,B)$ such that $\min\big\{|\neighbour_{\GD}(x)\cap V_2\cap (A\setminus U)|,
|B\setminus U|\big\} \ge 100d_1|A|+\tau k$. We apply
Lemma~\ref{lem:embed:regular} with input
$C_\PARAMETERPASSING{L}{lem:embed:regular}:= A$,
$D_\PARAMETERPASSING{L}{lem:embed:regular}:= B$
$X_\PARAMETERPASSING{L}{lem:embed:regular}=X^*_\PARAMETERPASSING{L}{lem:embed:regular}:=\neighbour_{\GD}(x)\cap
V_2\cap (A\setminus U)$, $Y_\PARAMETERPASSING{L}{lem:embed:regular}:=B\setminus
U$ , $\epsilon_\PARAMETERPASSING{L}{lem:embed:regular}:= \epsilon_1$,
$\beta_\PARAMETERPASSING{L}{lem:embed:regular}:= d_1/3$. Then  there exists an
embedding of $T^*$ in $V(\mathcal N)\setminus U$ such that $w$ is embedded in $V_2$. This ensures~\eqref{wiesoweshalbwarum}.

We remark that there may be several internal shrubs extending from $u\in W_A$. However Claim~\ref{claim:DTn} and the subsequent application of Lemma~\ref{lem:embed:regular} allows a sequential embedding of these shrubs.
This finishes the first stage of the embedding process. 

For the second stage, i.e., the embedding of the end shrubs of
$(W_A,W_B,\shrubA,\shrubB)$, we first recall that the total order of end shrubs
in $\shrubA$ is at most $h_2-2\frac{\eta^2 k}{10^3}$, and the total order of end
shrubs in $\shrubB$ is at most $\frac12\big(h_2-2\frac{\eta^2 k}{10^3}\big)$ by
Definition~\ref{ellfine}\eqref{Bsmall}. 
The embedding is a straightforward
application of Lemma~\ref{lem:embed:heart1}.
\end{proof}

The next lemma resolves Theorem~\ref{thm:main} in the presence of Configuration~$\mathbf{(\diamond10)}$.
\begin{lemma}\label{lem:embed10}
Suppose we are in Setting~\ref{commonsetting}.
For every $\eta',d',\Omega>0$ there exists a number $\tilde\epsilon>0$ such that for
every $\nu'>0$ satisfying
\begin{equation}\label{eq:plm}
\frac{\eta'\nu'}{200 \Omega}>\tau
\end{equation} there exists a number $k_0$ such that the following holds for each $k>k_0$.

If $G$ is a graph with Configuration~$\mathbf{(\diamond10)}(\tilde\epsilon,d',\nu'k, \Omega k,\eta')$
then each tree of order $k$ is contained in $G$.
\end{lemma}
\begin{proof}
We give a sketch of a proof, following~\cite{PS07+}. The main difference was indicated in Section~\ref{ssec:EmbedOverview10}.

Suppose we have Configuration~$\mathbf{(\diamond10)}(\tilde\epsilon,d',\nu'k, \Omega k,\eta')$, and are given a rooted tree $(T,r)$ of order $k$ with a $(\tau k)$-fine partition $(W_A,W_B,\shrubA,\shrubB)$ given by Lemma~\ref{lem:TreePartition}.  
By replacing $\mathcal L^*$ by $\mathcal L^*\setminus \V(\M)$,\footnote{This does not change validity of the conditions in Definition~\ref{def:CONF10}.} we can assume that $\mathcal L^*$ and $\V(\M)$ are disjoint.

For each shrub $F\in\mathcal S_A\cup\mathcal S_B$, let $x_F\in V(F)$ be its root, i.e., its minimal element in the topological order. If $F$ is internal then we also define $y_F$ as its (unique) maximal element that neighbours $W_A$.
We can partition the \semiregular matching $\mathcal M$ and the set $\mathcal L^*$ into two parts: $\mathcal M_A\cup \mathcal L^*_A$ and $\mathcal M_B\cup \mathcal L^*_B$ so that the partition satisfies
\begin{align}
\label{eq:splitSA}
\deg_{\tilde G}\big(v,V(\M_A)\cup \bigcup \LargeTen_A\big)&\ge v(\mathcal S_A)+\frac{\eta' k}4\quad\mbox{and}\\
\label{eq:splitSB}
\deg_{\tilde G}\big(w,V(\M_B)\cup \bigcup \LargeTen_B\big)&\ge v(\mathcal S_B)+\frac{\eta' k}4\;,
\end{align}
for all but at most $2\tilde\epsilon |A|$ vertices $v\in A$ and for all but at most $2\tilde\epsilon|B|$ vertices $w\in B$. To see this, observe that the nature of the regularized graph allows us to
treat\footnote{up to a small error} 
conditions~\eqref{eq:splitSA}, \eqref{eq:splitSB}, or that of Definition~\ref{def:CONF10}\eqref{diamond10cond2} in terms of average degrees of vertices in $A$ and $B$, rather than in terms of individual degrees.\footnote{This is also a key property in the classical dense setting of the regularity lemma.} If $A$ and $B$ were connected to each cluster $X\in \mathcal L^*\cup \mathcal V(\mathcal M)$ by regular pairs of the same density, say $d_X$, it would suffice to split $\mathcal L^*$ and $\mathcal M$ in the ratio $v(\mathcal S_A):v(\mathcal S_B)$. In the general setting, this can also be achieved, as was done in~\cite[Lemma~9]{PS07+}. 
Let $h_{A,\mathcal L^*}$, $h_{B,\mathcal L^*}$, $h_{A,\M}$,  and $h_{B,\M}$ be the average degrees of vertices of $A$ and $B$ into $\mathcal L^*_A,\mathcal L^*_B,\M_A$,$\M_B$.

We will now use the regularity to embed the shrubs and the seeds in $\tilde G$. We start with mapping~$r$ to $A$ or $B$ (depending on whether $r\in W_A$ or $r\in W_B$), and proceed along a topological order on~$T$. We denote the partial embedding of $T$ at any particular stage as $\phi$. The vertices of $W_A$ are mapped to $A$, the vertices of $W_B$ are mapped to $B$. As for embedding the shrubs, initially we start with embedding the shrubs of $\mathcal S_A$ to $\M_A$ (we say that \emph{$A$ is in  the $\M$-mode}), and embedding the shrubs of $\mathcal S_B$ to $M_B$ (\emph{$B$ is in  the $\M$-mode}). 
By filling up the $\M$-edges with the shrubs as balanced as possible we can guarantee that we do not run out of space in $\M_A$ before embedding $\mathcal S_A$-shrubs of total order at least $h_{A,\M}-\eta'k/100$. An analogous property holds for embedding $\mathcal S_B$-shrubs. We omit details and instead refer to a very similar procedure in Lemma~\ref{lem:embed9}.\footnote{In Lemma~\ref{lem:embed9} it was shown how to utilize~\eqref{conf:D9-XtoV} for embedding shrubs of order up to $\approx h_1$ in regular pairs.} 

At some moment we may run out of space in $\M_A$, or in $\M_B$. Say that this happens first with the matching $\M_A$. Let $\mathcal S_A^*\subset \mathcal S_A$ be the set of shrubs not embedded so far. We now describe how to proceed when \emph{$A$ is in the $\mathcal L^*$-mode}. In this mode, we will not embed an upcoming shrub $F\in\mathcal S^*_A$, but only reserve a set $U_F$, with $|U_F|\le v(F)$ which serves as a reminder that we want to accommodate $F$ later on. Suppose that the parent $\parent(F)\in W_A$ of $F$ has been mapped to a typical\footnote{in the sense of Definition~\ref{def:CONF10}\eqref{diamond10cond2}} vertex $z\in A$ already. We have 
\begin{equation*}
\deg_{\tilde G}(z,\bigcup \mathcal L^*_A)\ge v(\mathcal S_A^*)+\frac{\eta'k}{100}\ge \sum_{F'} |U_F'|+\frac{\eta'k}{100}\;,
\end{equation*}
where the sum ranges over the already processed $\mathcal S_A^*$-shrubs $F'$. Consequently, there is a cluster $X\in \V$ such that 
\begin{equation}\label{eq:thresholdFilling}
\deg_{\tilde G}\Big(z,X\setminus \bigcup_{F'}U_{F'}\Big)>\frac{\eta'|X|}{100\Omega}\;.
\end{equation} Let us view $F$ as a bipartite graph, and let $a_F$ be the size of its colour class that contains $x_F$. Let $U_F$ be an arbitrary set of $(\neighbour_{\tilde G}(z)\cap X)\setminus \bigcup_{F'}U_{F'}$ of size $a_F$, and also let us fix an image $\phi(x_F)\in U_F$ arbitrarily. If $F$ is an internal shrub, we further define $\phi(y_F)\in U_F\setminus\{\phi(x_B)\}$ arbitrarily. At this stage we consider $F$ as processed. 

Later, of course, also $B$ can switch to the $\mathcal L^*$-mode as well. At that moment, we define $\mathcal S^*_B$, and start to only make reservations $U_K$ in clusters of $\mathcal L^*_B$ instead of embedding shrubs $K\in\mathcal S^*_B$.

After all shrubs of $\mathcal S^*_A\cup\mathcal S^*_B$ have been processed we finalize the embedding. Consider a shrub $F\in\mathcal S^*_A\cup\mathcal S^*_B$. Suppose that $U_F\subset X$ for some $X\in\V$. We use Definition~\ref{def:CONF10}\eqref{diamond10cond3} to find a cluster $Y$ such that $$\density(X,Y)\ge \frac{\left|Y\cap\big(\text{im}(\phi)\cup\bigcup_{F'\text{ yet unembedded}}U_{F'}\big)\right|}{|Y|}+\frac{\eta'}{100\Omega}\;.$$ 
As $\phi(x_F)$ and $\phi(y_F)$ are typical\footnote{in the sense of Definition~\ref{def:CONF10}\eqref{diamond10cond3}}, we can additionally require that 
$$\deg_{\tilde G}(\phi(x_F),Y),\deg_{\tilde G}(\phi(y_F),Y)\ge (\density(X,Y)-\sqrt{\tilde\epsilon})|Y|\;.$$
Therefore, the regularity method allows us to embed $F$ in the pair $(X,Y)$ avoiding the already defined image of $\phi$, and the sets $U_{F'}$ corresponding to yet unembedded shrubs $F'$. The fact that the threshold in~\eqref{eq:thresholdFilling} was taken quite high (compared to the size of the shrubs, see~\eqref{eq:plm}) allows us to avoid atypical vertices. We also need this embedding to be compatible with the existing placements $\phi(x_F)$ and $\phi(y_F)$. In particular, we need to find a path of length $\dist_F(x_F,y_F)$ from $\phi(x_F)$ to $\phi(y_F)$. Here, it is crucial that $\dist_F(x_F,y_F)\ge 4$ (cf.~Definition~\ref{ellfine}\eqref{short}).\footnote{Indeed, it could be that $\neighbour(\phi(x_F))\cap\neighbour(\phi(y_F))=\emptyset$, which would make it impossible to find a path of length~2 from $\phi(x_F)$ to $\phi(y_F)$. If, on the other hand $\dist_F(x_F,y_F)\ge 4$, then we can always find such a path using a look-ahead embedding in the regular pair $(X,Y)$.}
We remark, that in general we cannot guarantee that $X\cap \phi(F)=U_F$. So the set $U_F$ should be regarded merely as a measure of future occupation of $X$, rather than an indication of exact future placement.
\end{proof}

\section{Proof of Theorem~\ref{thm:main}}\label{sec:proof}

The proof builds on the main results from~\cite{cite:LKS-cut0,cite:LKS-cut1,cite:LKS-cut2}. We extend our subscript notation to allow referencing to parameters from~\cite{cite:LKS-cut0,cite:LKS-cut1,cite:LKS-cut2}. For example, $\eta_\EXTERNALPARAMETERPASSING{I}{L}{p0.lem:LKSsparseClass}$ refers to the parameter $\eta$ from Lemma~\ref{p0.lem:LKSsparseClass} from the I$^\text{st}$ paper, that is, from~\cite{cite:LKS-cut0}.

Let $\alpha>0$ be given. We set $$\eta:=\min\{\frac1{30},\frac{\alpha}2\}.$$ We wish to fix further constants satisfying~\eqref{eq:KONST}. A~trouble is that we do not know the right choice of $\Omega^*$ and $\Omega^{**}$ yet. Therefore we take $g:=\lfloor\frac{100}{\eta^2}\rfloor+1$ and fix suitable constants
\begin{align*}
\eta\gg\frac1{\Omega_1}&\gg \frac1{\Omega_2}\gg \ldots\gg \frac1{\Omega_{g+1}}
\gg\rho\gg\gamma\gg d\ge \frac1{\Lambda}\ge \epsilon\ge
\pi\ge  \alphaD \ge
\epsilon'\ge
\nu\gg \tau \gg \frac{1}{k_0}>0\;.
\end{align*}
where the relations between the parameters are more exactly as follows:
\begin{align*}
\frac 1{\Omega_1}&\le \BUG{\frac {\eta^{13}}{10^{33}}}\;,\\
\frac 1{\Omega_{j+1}}&\le \frac
{\eta^{27}}{10^{67}\Omega_j^{36}}\quad\mbox{for each $j=1,\ldots,g$\;,}\\
\rho&\le \BUG{\frac {\eta^{13}}{10^{33}\Omega_{g+1}^5}}\;,\\
\gamma&\le \frac {\eta^{18}\rho^{24}}{10^{90}\Omega_{g+1}^{28}}\;,\\
d&\le \min\left\{\frac {\gamma^2\eta^2}{10^8\Omega_{g+1}^2},
\beta_\EXTERNALPARAMETERPASSING{II}{L}{p1.prop:LKSstruct}(\eta_\EXTERNALPARAMETERPASSING{II}{L}{p1.prop:LKSstruct}:=
\eta, \Omega_\EXTERNALPARAMETERPASSING{II}{L}{p1.prop:LKSstruct}:= \Omega_{g+1},
\gamma_\EXTERNALPARAMETERPASSING{II}{L}{p1.prop:LKSstruct}:= \gamma)\right\}\;,\\
\frac 1\Lambda&\le \min\left\{d,\frac
{\eta^{24}\gamma^{24}\rho}{10^{96}\Omega_{g+1}^{36}}\right\}
\;,
\\
\epsilon&\le \min\left\{\frac 1\Lambda, \frac
{\gamma^2\eta^3d\rho}{10^{13}\Omega_{g+1}^4},
\tilde\eps_\PARAMETERPASSING{L}{lem:embed10}(\eta'_\PARAMETERPASSING{L}{lem:embed10}:=\eta/40,
d'_\PARAMETERPASSING{L}{lem:embed10}:=\gamma^2d/2,
\Omega_\PARAMETERPASSING{L}{lem:embed10}:= \frac {(\Omega_{g+1})^2}{\gamma^2})
\right\}\;,
\\
\pi&\le \min\left\{\epsilon,
\pi_\EXTERNALPARAMETERPASSING{II}{L}{p1.prop:LKSstruct}(\eta_\EXTERNALPARAMETERPASSING{II}{L}{p1.prop:LKSstruct}:=
\eta, \Omega_\EXTERNALPARAMETERPASSING{II}{L}{p1.prop:LKSstruct}:= \Omega_{g+1},
\gamma_\EXTERNALPARAMETERPASSING{II}{L}{p1.prop:LKSstruct}:= \gamma,
\epsilon_\EXTERNALPARAMETERPASSING{II}{L}{p1.prop:LKSstruct}:= \epsilon)\right\}\;,\\
\alphaD&\le \min\left\{\epsilonD,
\alpha_\EXTERNALPARAMETERPASSING{II}{L}{p1.lem:edgesEmanatingFromDensePairsIII}\left(
\Omega_\EXTERNALPARAMETERPASSING{II}{L}{p1.lem:edgesEmanatingFromDensePairsIII}:=
\Omega_{g+1},
\rho_\EXTERNALPARAMETERPASSING{II}{L}{p1.lem:edgesEmanatingFromDensePairsIII}:= \frac
{\gamma^2}4,
\epsilon_\EXTERNALPARAMETERPASSING{II}{L}{p1.lem:edgesEmanatingFromDensePairsIII}:=
\epsilonD,
\tau_\EXTERNALPARAMETERPASSING{II}{L}{p1.lem:edgesEmanatingFromDensePairsIII}:=
2\rho\right)\right\}\;,\\
\epsilon'&\le \min\left\{\frac
{\alphaD^2\gamma^6\rho^2}{10^3\Omega_{g+1}^4},
\epsilon'_\EXTERNALPARAMETERPASSING{II}{L}{p1.prop:LKSstruct}(\eta_\EXTERNALPARAMETERPASSING{II}{L}{p1.prop:LKSstruct}:=
\eta, \Omega_\EXTERNALPARAMETERPASSING{II}{L}{p1.prop:LKSstruct}:= \Omega_{g+1}, \gamma_\EXTERNALPARAMETERPASSING{II}{L}{p1.prop:LKSstruct}:= \gamma,
\epsilon_\EXTERNALPARAMETERPASSING{II}{L}{p1.prop:LKSstruct}:= \epsilon)\right\}\;,\\
\nu&\le \min\left\{\frac {\alphaD\rho}{\Omega_{g+1}}, \epsilon',
\nu_\EXTERNALPARAMETERPASSING{I}{L}{p0.lem:LKSsparseClass}(\eta_\EXTERNALPARAMETERPASSING{I}{L}{p0.lem:LKSsparseClass}:=\alpha,
\Lambda_\EXTERNALPARAMETERPASSING{I}{L}{p0.lem:LKSsparseClass}:=\Lambda,
\gamma_\EXTERNALPARAMETERPASSING{I}{L}{p0.lem:LKSsparseClass}:= \gamma,
\epsilon_\EXTERNALPARAMETERPASSING{I}{L}{p0.lem:LKSsparseClass}:= \epsilon',
\rho_\EXTERNALPARAMETERPASSING{I}{L}{p0.lem:LKSsparseClass}:= \rho)\right\}\;,\\
\tau&\le 2\epsilon\pi\nu\;,\\
\frac 1{k_0}&\le \min\left\{\frac
{\gamma^3\rho\eta^8\tau\nu}{10^3\Omega_{g+1}^3}, \frac 1{k_0^*}\right\}\;,
\end{align*}
with $k_0^*$ set as the maximum of the numbers
 \label{pageref:PAR}
 \begin{align*}
&k_{0,\EXTERNALPARAMETERPASSING{I}{L}{p0.lem:LKSsparseClass}}\left(\eta_\EXTERNALPARAMETERPASSING{I}{L}{p0.lem:LKSsparseClass}:=\alpha,
\Lambda_\EXTERNALPARAMETERPASSING{I}{L}{p0.lem:LKSsparseClass}:=\Lambda, \gamma_\EXTERNALPARAMETERPASSING{I}{L}{p0.lem:LKSsparseClass}:= \gamma,
\epsilon_\EXTERNALPARAMETERPASSING{I}{L}{p0.lem:LKSsparseClass}:= \epsilon',
\rho_\EXTERNALPARAMETERPASSING{I}{L}{p0.lem:LKSsparseClass}:= \rho\right),\\
 & k_{0,\EXTERNALPARAMETERPASSING{II}{L}{p1.lem:edgesEmanatingFromDensePairsIII}}\left(
\Omega_\EXTERNALPARAMETERPASSING{II}{L}{p1.lem:edgesEmanatingFromDensePairsIII}:=
\Omega_{g+1},
\rho_\EXTERNALPARAMETERPASSING{II}{L}{p1.lem:edgesEmanatingFromDensePairsIII}:= \frac
{\gamma^2}4,
\epsilon_\EXTERNALPARAMETERPASSING{II}{L}{p1.lem:edgesEmanatingFromDensePairsIII}:=
\epsilonD,
\tau_\EXTERNALPARAMETERPASSING{II}{L}{p1.lem:edgesEmanatingFromDensePairsIII}:=
2\rho,
\alpha_\EXTERNALPARAMETERPASSING{II}{L}{p1.lem:edgesEmanatingFromDensePairsIII}:=
\alphaD,
\nu_\EXTERNALPARAMETERPASSING{II}{L}{p1.lem:edgesEmanatingFromDensePairsIII}:=
\frac{2\rho}{\Omega_{g+1}}
\right),
\\
 & k_{0,\EXTERNALPARAMETERPASSING{II}{L}{p1.prop:LKSstruct}}\left(\eta_\EXTERNALPARAMETERPASSING{II}{L}{p1.prop:LKSstruct}:=
\eta, \Omega_\EXTERNALPARAMETERPASSING{II}{L}{p1.prop:LKSstruct}:= \Omega_{g+1},
\gamma_\EXTERNALPARAMETERPASSING{II}{L}{p1.prop:LKSstruct}:= \gamma,
\epsilon_\EXTERNALPARAMETERPASSING{II}{L}{p1.prop:LKSstruct}:= \epsilon,
\nu_\EXTERNALPARAMETERPASSING{II}{L}{p1.prop:LKSstruct}:= \nu\right),
\\
 & k_{0,\PARAMETERPASSING{L}{lem:randomSplit}}\left(p_\PARAMETERPASSING{L}{lem:randomSplit}:=
10, \alpha_\PARAMETERPASSING{L}{lem:randomSplit}:= \eta/100\right),
\\
 & k_{0,\PARAMETERPASSING{L}{lem:embed10}}\left(\eta'_\PARAMETERPASSING{L}{lem:embed10}:=\eta/40,
d'_\PARAMETERPASSING{L}{lem:embed10}:=\gamma^2d/2,\tilde
\epsilon_\PARAMETERPASSING{L}{lem:embed10}:= \epsilon,
\Omega_\PARAMETERPASSING{L}{lem:embed10}:= \frac {(\Omega_{g+1})^2}{\gamma^2},
\nu'_\PARAMETERPASSING{L}{lem:embed10}:=\pi\nu\sqrt{\epsilon'}\right).
\end{align*}
In particular, this gives a relation between between $\alpha$ and $k_0$. 
\medskip

Suppose now that $k>k_0$, and $G\in\LKSgraphs{n}{k}{\alpha}$ is a graph, and $T\in\treeclass{k}$ is a tree of order~$k$. Our goal is to show that $T\subset G$.
\medskip

Let us now turn to the proof. First, we
preprocess the tree $T$ by considering any $(\tau k)$-fine partition
$(W_A,W_B,\shrubA,\shrubB)$ of $T$ rooted at an arbitrary root $r$. Such a
partition exists by Lemma~\ref{lem:TreePartition}. Let $m_1$ and $m_2$ be the
total order of the internal shrubs and the end shrubs, respectively. Set $$\proporce{0}:=\frac {\eta}{100}\quad\mbox{and}\quad
\proporce{i}:=\frac{\eta}{100}+\frac{m_i}{(1+\frac{\eta}{30})k}\text{, for $i=1,2$}\;.$$

In particular we have
$\proporce{i}\in[\frac{\eta}{100},1]$ for $i=0,1,2$.

To find a suitable structure in the graph $G$ we proceed as follows. We apply
\cite[Lemma~\ref{p0.lem:LKSsparseClass}]{cite:LKS-cut0} with input graph
$G_\EXTERNALPARAMETERPASSING{I}{L}{p0.lem:LKSsparseClass}:=G$ and parameters
$\eta_\EXTERNALPARAMETERPASSING{I}{L}{p0.lem:LKSsparseClass}:=\alpha$,
$\Lambda_\EXTERNALPARAMETERPASSING{I}{L}{p0.lem:LKSsparseClass}:=\Lambda$,
$\gamma_\EXTERNALPARAMETERPASSING{I}{L}{p0.lem:LKSsparseClass}:=\gamma$,
$\epsilon_\EXTERNALPARAMETERPASSING{I}{L}{p0.lem:LKSsparseClass}:=\epsilon'$,
$\rho_\EXTERNALPARAMETERPASSING{I}{L}{p0.lem:LKSsparseClass}:=\rho$,  the sequence
$(\Omega_j)_{j=1}^{g+1}$, $k_\EXTERNALPARAMETERPASSING{I}{L}{p0.lem:LKSsparseClass}:=k$ and  $b_\EXTERNALPARAMETERPASSING{I}{L}{p0.lem:LKSsparseClass}:=\frac{\rho k}{100\Omega^*}$. The~lemma gives a graph
$G'_\EXTERNALPARAMETERPASSING{I}{L}{p0.lem:LKSsparseClass}\in\LKSsmallgraphs{n}{k}{\eta}$, and an index $i\in[g]$.  Slightly abusing  notation, we call this graph still~$G$. Set
 $\Omega^*:=\Omega_i$ and $\Omega^{**}:=\Omega_{i+1}$.
Now, \cite[Lemma~\ref{p0.lem:LKSsparseClass}\eqref{p0.LKSclassif:prepart}]{cite:LKS-cut0}
yields a $(k,\Omega^{**},\Omega^*,\Lambda,\gamma,\epsilon',\nu,\rho)$-sparse
decomposition $\class=(\HugeVertices, \clusters, \DenseSpots, \Gblack,
\Gexp,\smallatoms)$. Let $\clustersize$ be the size of any cluster in
$\clusters$.
 
We now apply \cite[Lemma~\ref{p1.prop:LKSstruct}]{cite:LKS-cut1} with parameters
$\eta_\EXTERNALPARAMETERPASSING{II}{L}{p1.prop:LKSstruct}:=\eta$,
$\Omega_\EXTERNALPARAMETERPASSING{II}{L}{p1.prop:LKSstruct}:=\Omega_{g+1}$,
$\gamma_\EXTERNALPARAMETERPASSING{II}{L}{p1.prop:LKSstruct}:=\gamma$,
$\epsilon_\EXTERNALPARAMETERPASSING{II}{L}{p1.prop:LKSstruct}:=\epsilon$, 
$k_\EXTERNALPARAMETERPASSING{II}{L}{p1.prop:LKSstruct}:=k$,  and
$\Omega^*_\EXTERNALPARAMETERPASSING{II}{L}{p1.prop:LKSstruct}:=\Omega^*$. Given the graph $G$ with its sparse decomposition $\class$ the lemma gives three
$(\epsilon,d,\pi\clustersize)$-\semiregular matchings $\M_A$, $\M_B$, and
$\Mgood\subset \M_A$ which fulfill the assertion either of case {\bf(K1)}, or
of {\bf(K2)}. The matchings $\M_A$ and $\M_B$ also define the sets $\XA$ and $\XB$.

The additional features provided by \cite[Lemma~\ref{p0.lem:LKSsparseClass}]{cite:LKS-cut0} and \cite[Lemma~\ref{p1.prop:LKSstruct}]{cite:LKS-cut1} guarantee that we are in the situation described in Setting~\ref{commonsetting}. We apply Lemma~\ref{lem:randomSplit} as described in Definition~\ref{def:proportionalsplitting}; the numbers $\proporce{0},\proporce{1},\proporce{2}$ are as defined above. This puts us in the setting described in Setting~\ref{settingsplitting}. We now use~\cite[Lemma~\ref{p2.outerlemma}]{cite:LKS-cut2} to obtain one of the following configurations.
\begin{itemize}
\item$\mathbf{(\diamond1)}$,
\item$\mathbf{(\diamond2)}\BUG{\left( \frac{\eta^{39}\Omega^{**}}{4\cdot
	10^{90}(\Omega^*)^{11}},\frac{\sqrt[4]{\Omega^{**}}}2,\frac{\eta^{13}\rho^2}{128\cdot
	10^{30}\cdot (\Omega^*)^5}\right)}$,
\item$\mathbf{(\diamond3)}\BUG{\left(\frac{\eta^{39}\Omega^{**}}{4\cdot
	10^{90}(\Omega^*)^{11}},\frac{\sqrt[4]{\Omega^{**}}}2,\frac\gamma2,\frac{\eta^{13}\gamma^2}{128\cdot
	10^{30}\cdot(\Omega^*)^5}\right)}$,
\item$\mathbf{(\diamond4)}\BUG{\left(\frac{\eta^{39}\Omega^{**}}{4\cdot
	10^{90}(\Omega^*)^{11}},\frac{\sqrt[4]{\Omega^{**}}}2,\frac\gamma2,\frac{\eta^{13}\gamma^3}{384\cdot
	10^{30}(\Omega^*)^6}\right)}$,
\item$\mathbf{(\diamond5)}\BUG{\left(\frac{\eta^{39}\Omega^{**}}{4\cdot
	10^{90}(\Omega^*)^{11}},
\frac{\sqrt[4]{\Omega^{**}}}2,\frac{\eta^{13}}{128\cdot
	10^{30}\cdot
	(\Omega^*)^3},\frac{\eta}2,\frac{\eta^{13}}{128\cdot
	10^{30}\cdot (\Omega^*)^4}\right)}$,
\item
$\mathbf{(\diamond6)}\big(\frac{\eta^3\rho^4}{10^{14}(\Omega^*)^4},4\epsilonD,\frac{\gamma^3\rho}{32\Omega^*},\frac{\eta^2\nu}{2\cdot10^4
},\frac{3\eta^3}{2000},\proporce{2}(1+\frac\eta{20})k\big)$,
\item $\mathbf{(\diamond7)}\big(\frac
{\eta^3\gamma^3\rho}{10^{12}(\Omega^*)^4},\frac
{\eta\gamma}{400},4\epsilonD,\frac{\gamma^3\rho}{32\Omega^*},\frac{\eta^2\nu}{2\cdot10^4
}, \frac{3\eta^3}{2\cdot 10^3}, \proporce{2}(1+\frac\eta{20})k\big)$,
\item
$\mathbf{(\diamond8)}\big(\frac{\eta^4\gamma^4\rho}{10^{15}
	(\Omega^*)^5},\frac{\eta\gamma}{400},\frac{400\epsilon}{\eta},4\epsilonD,\frac
d2,\frac{\gamma^3\rho}{32\Omega^*},\frac{\eta\pi\clustersize}{200k},\frac{\eta^2\nu}{2\cdot10^4
}, \proporce{1}(1+\frac\eta{20})k,\proporce{2}(1+\frac\eta{20})k\big)$,
\item 
$\mathbf{(\diamond9)}\big(\frac{\rho
	\eta^8}{10^{27}(\Omega^*)^3},\frac
{2\eta^3}{10^3}, \proporce{1}(1+\frac{\eta}{40})k,
\proporce{2}(1+\frac{\eta}{20})k, \frac{400\varepsilon}{\eta},
\frac{d}2,
\frac{\eta\pi\clustersize}{200k},4\epsilonD,\frac{\gamma^3\rho}{32\Omega^*},
\frac{\eta^2\nu}{2\cdot10^4 }\big)$,
\item $\mathbf{(\diamond10)}\big( \epsilon, \frac{\gamma^2
	d}2,\pi\sqrt{\epsilon'}\nu k, \frac
{(\Omega^*)^2k}{\gamma^2},\frac\eta{40} \big)$.
\end{itemize}

Depending on the actual configuration Lemma~\ref{lem:embed:greedy}, Lemma~\ref{lem:conf2-5}, Lemma~\ref{lem:embed:total68}, Lemma~\ref{lem:embed9}, or Lemma~\ref{lem:embed10} guarantees that $T\subset G$. This finishes the proof of the theorem.

\section{Theorem~\ref{thm:main} algorithmically}\label{ssec:algorithmic}\label{sec:conclremarks}
We now discuss the algorithmic aspects of our proof of Theorem~\ref{thm:main}. This discussion also covers the parts developed in the preceding papers of the series~\cite{cite:LKS-cut0,cite:LKS-cut1,cite:LKS-cut2} (although we do refer at one point to a discussion from~\cite{cite:LKS-cut0}).
\begin{figure}[t!]
\begin{center}
\includegraphics{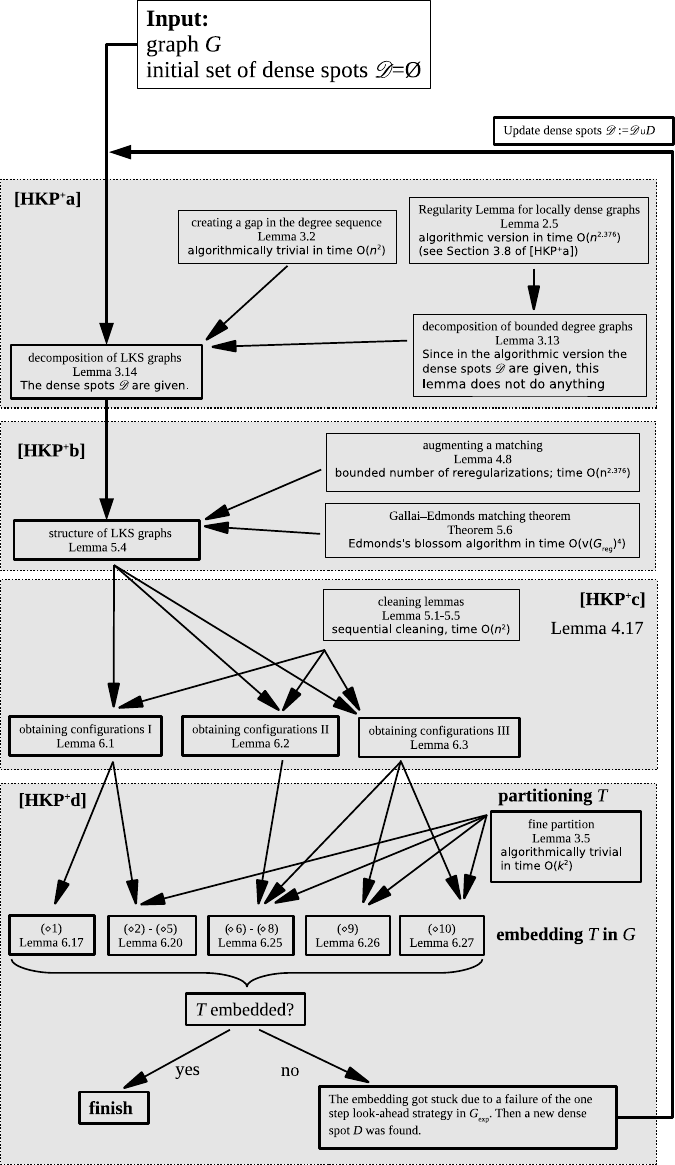}
\caption{A version of \cite[Figure~\ref{p0.fig:proofstructure}]{cite:LKS-cut0} showing the iterative algorithm for finding a copy of $T$ in $G$.}
\label{fig:ProofStructureAlgoritmic}
\end{center}
\end{figure}

The interesting question is if we can provide a fast algorithm which finds a copy of a given tree
$T\in\treeclass{k}$ in any given graph $G\in\LKSgraphs{n}{k}{\alpha}$. We will sketch below that our proof  gives such an algorithm, with running time $O(n^6)$; here the hidden constant in the $O(\cdot)$-notation depends on~$\alpha$ but not on~$k$. 
A picture accompanying the discussion  is given in Figure~\ref{fig:ProofStructureAlgoritmic}.

It can be verified that each of
the steps of our proof --- except the extraction of dense spots in~\cite{cite:LKS-cut0}--- can be turned into a polynomial time algorithm. We return to the extraction of dense spots later, after discussing the other parts of the proof.
\begin{itemize}
\item In \cite[Section~\ref{p0.sssec:DecomposeAlgorithmically}]{cite:LKS-cut0} we discussed the algorithmic aspects of obtaining a sparse decomposition of $G$, which is the main result (\cite[Lemma~\ref{p0.lem:LKSsparseClass}]{cite:LKS-cut0}) of~\cite{cite:LKS-cut0}. This part contains the bottleneck step of the extraction of dense spots (in~\cite[Lemma~\ref{p0.lem:decompositionIntoBlackandExpanding}]{cite:LKS-cut0}).
\item In~\cite{cite:LKS-cut1} we find a ``rough structure'' in $G$. Here, we need to find a matching in $\BGblack$ that is maximal in a certain way, and we also need to ``augment a regularized matching''. The former step can be done using Edmonds's blossom algorithm, and the latter by applying the algorithmic version of the regularity lemma~\cite{Alon94thealgorithmic}. (We used~\cite{Alon94thealgorithmic} already in obtaining a sparse decomposition in~\cite{cite:LKS-cut0}.)
\item In~\cite{cite:LKS-cut2} we apply ``cleaning lemmas'' to refine the rough structure. The cleaning lemmas proceed by sequentially removing of ``bad'' vertices, and the respective badness conditions can be efficiently tested. The cleaning procedure are then put together in~\cite[Lemma~\ref{p2.lem:ConfWhenCXAXB}--\ref{p2.lem:ConfWhenMatching}]{cite:LKS-cut2}. These lemmas are easily turned into algorithms.
\item In the present paper we embed $T$ in $G$ using one of the configurations obtained in~\cite{cite:LKS-cut2}. The basic ingredients of the embedding are the following:
\begin{itemize}
\item \emph{Embedding in huge degree vertices (in $\mathbf{(\diamond2)}$--$\mathbf{(\diamond5)}$)}\\
The two main technical lemmas used are Lemmas~\ref{lem:embedC'endshrub} and~\ref{lem:blueShrubSuspend}. In these lemmas, in each step of the embedding we find a vertex having a substantial degree into one of the specified sets. So, the non-trivial assertions of these lemmas is that these good vertices exist. On the other hand, testing whether a given vertex is good can be done algorithmically (in time $O(n^2$)).
\item \emph{Embedding into regular pairs}\\
The exact setting is described in Lemmas~\ref{lem:embed:regular}--\ref{lem:embed:superregular}, but the way we proceed with embedding of small trees is standard. That is, when we extend an embedding of a small tree or forest in a regular pair $(X,Y)$, we find a vertex of one cluster that has a substantial degree to the unused part of the partner cluster (the existence of which is guaranteed by the regularity). This can be implemented in time $O(|X||Y|)$.
\item \emph{Embedding into $\Gexp$}\\
The embedding procedure for embedding into $\Gexp$ was informally described in~\cite[Section~\ref{p0.sssec:whyGexp}]{cite:LKS-cut0} and the actual setting we use is given in Lemma~\ref{lem:embed:greyFOREST}. The procedure is algorithmic. Indeed, when,  in the proof of Lemma~\ref{lem:embed:greyFOREST}, we extend a partial embedding of a forest, it is enough to avoid the set called $\shadow_{H_\PARAMETERPASSING{L}{lem:embed:greyFOREST}}(U_\PARAMETERPASSING{L}{lem:embed:greyFOREST},\zeta_\PARAMETERPASSING{L}{lem:embed:greyFOREST} k/2)$. This set can be easily determined algorithmically.
\item \emph{Embedding using $\smallatoms$}\\
Let us recall the elementary embedding procedure for $\smallatoms$ as described in Lemma~\ref{lem:embed:avoidingFOREST}. We have a small rooted tree $(T,r)$ (or several small rooted trees), a forbidden set $U$, and a set $U^*\subset\smallatoms$. It is our task to embed $T$, avoiding the set $U$, so that $r$ is placed in $U^*$. For the proof of Lemma~\ref{lem:embed:avoidingFOREST} we only use the ``avoiding'' feature of the avoiding set  given by Definition~\ref{def:avoiding}. That is, for each $y\in U^*$ we test whether there is a dense spot $D\in\DenseSpots$ with $|U\cap V(D)|\le \gamma^2k$ such that $y$ sends a substantial degree to $D$, or whether $y$ belongs to the bad set $Y_\PARAMETERPASSING{ProofL}{lem:embed:avoidingFOREST}$. This test can be made algorithmic  by simply ranging over the at most $O(n/k)$ dense spots in~$\DenseSpots$.
\end{itemize}

The two randomized steps --- random splitting in~\cite[Section~\ref{p2.ssec:RandomSplittins}]{cite:LKS-cut2}, and the use of the stochastic process $\Duplicate$ in Section~\ref{sec:embed} --- can be also efficiently derandomized using a standard technique for derandomizing the Chernoff bound.\footnote{It further seems that the use of randomization in \cite[Section~\ref{p2.ssec:RandomSplittins}]{cite:LKS-cut2} can be eliminated entirely. To this end we would have employ different arguments to obtain the fine structure. We have not worked out details.}
\end{itemize}

\medskip
 Let us now sketch how to deal with extracting dense spots.
The idea is as follows. Initially, we pretend that $\Gexp$ consists of the entire bounded-degree part $G-\HugeVertices$ (cleaned for minimum degree~$\rho k$ as in~\cite[\eqref{p0.eq:almostalldashed}]{cite:LKS-cut0}). With such a supposed sparse decomposition $\class_1$ we go through~\cite[Lemma~\ref{p1.prop:LKSstruct}]{cite:LKS-cut1} and~\cite[Lemma~\ref{p2.outerlemma}]{cite:LKS-cut2} to obtain a configuration. We now start embedding $T$ as in Section~\ref{sec:embed}. (Note that at this moment $\Gblack$ and $\smallatoms$ are absent, and so, the only embedding techniques are those involving $\HugeVertices$ and $\Gexp$.) Now, either we embed $T$, or we fail. The only possible reason for failure is that we were unable to perform the one-step look-ahead strategy described in \cite[Section~\ref{p0.sssec:whyGexp}]{cite:LKS-cut0}, because $\Gexp$ was not really nowhere-dense. (In order to understand fully that this is indeed the only possible reason, the reader is advised to read the explanatory, two-pages-long Section~\ref{p0.sssec:whyGexp} of \cite{cite:LKS-cut0}.) But then we actually localized a dense spot $D_1$. We get an updated supposed sparse decomposition $\class_2$ in which $D_1$ is removed from $\Gexp$ and added to $\DenseSpots$ (and $\Gblack$ and/or $\smallatoms$ are  modified accordingly). We keep iterating. Since in each step we extract at least $O(k^2)$ edges we iterate the above at most $e(G)/\Theta(k^2)=O(\frac nk)$ times. We certainly succeed eventually, since after $\Theta(\frac nk)$ iterations we get an honest sparse decomposition (i.e., a decomposition that would be a valid outcome of~\cite[Lemma~\ref{p0.lem:LKSsparseClass}]{cite:LKS-cut0}, with  $\Gexp$ nowhere-dense).

\medskip
It seems that this iterative method is generally applicable for problems which employ a sparse decomposition.

\section{Acknowledgments} 
The work on this project lasted from the beginning of 2008 until 2014
and we are very grateful to the following institutions and funding bodies for
their support. 

\smallskip

During the work on this paper Hladk\'y was also affiliated with Zentrum
Mathematik, TU Munich and Department of Computer Science, University of Warwick. Hladk\'y was funded by a BAYHOST fellowship, a DAAD fellowship, 
Charles University grant GAUK~202-10/258009, EPSRC award EP/D063191/1, and by an EPSRC Postdoctoral Fellowship while working on the project.

Koml\'os and Szemer\'edi acknowledge the support of NSF grant
DMS-0902241.

Piguet was also affiliated with the Institute of Theoretical Computer Science, Charles University in Prague, Zentrum
Mathematik, TU Munich, the Department of Computer Science and DIMAP,
University of Warwick, and the School of Mathematics, University of Birmingham. The work leading
to this invention was supported by the European Regional Development Fund (ERDF), project ``NTIS --- New Technologies for Information Society'', European Centre of Excellence, CZ.1.05/1.1.00/02.0090.
The research leading to these results has received funding from the European Union Seventh
Framework Programme (FP7/2007-2013) under grant agreement no. PIEF-GA-2009-253925.
Piguet acknowledges the support of the Marie Curie fellowship FIST,
DFG grant TA 309/2-1, a DAAD fellowship,
Czech Ministry of
Education project 1M0545,  EPSRC award EP/D063191/1,
and  the support of the EPSRC
Additional Sponsorship, with a grant reference of EP/J501414/1 which facilitated her to
travel with her young child and so she could continue to collaborate closely
with her coauthors on this project. This grant was also used to host Stein in
Birmingham.

Stein was affiliated with the Institute of Mathematics and Statistics, University of S\~ao Paulo, and the Center for Mathematical Modeling, University of Chile. She was
supported by a FAPESP fellowship, and by FAPESP travel grant  PQ-EX 2008/50338-0, also
CMM-Basal,  FONDECYT grants 11090141 and  1140766. She also received funding by EPSRC Additional Sponsorship EP/J501414/1.

We enjoyed the hospitality of the School of Mathematics of University of Birmingham, Center for Mathematical Modeling, University of Chile, Alfr\'ed R\'enyi Institute of Mathematics of the Hungarian Academy of Sciences and Charles University, Prague, during our long term visits.

The yet unpublished work of Ajtai, Koml\'os, Simonovits, and Szemer\'edi on the Erd\H{o}s--S\'os Conjecture was the starting point for our project, and our solution crucially relies on the methods developed for the Erd\H{o}s-S\'os Conjecture. Hladk\'y, Piguet, and Stein are very grateful to the former group for explaining them those techniques.

\medskip
A doctoral thesis entitled \emph{Structural graph theory} submitted by Hladk\'y in September 2012 under the supervision of Daniel Kr\'al at~Charles University in~Prague is based on the series of the papers~\cite{cite:LKS-cut0,cite:LKS-cut1, cite:LKS-cut2} and the present paper. The texts of the two works overlap greatly. We are grateful to PhD committee members Peter Keevash and Michael Krivelevich. Their valuable comments are reflected in the series. 

\bigskip
We thank the referees for their very helpful remarks.

\bigskip
The contents of this publication reflects only the authors' views and not necessarily the views of the European Commission of the European Union.

\newpage
\printindex{mathsymbols}{Symbol index}
\printindex{general}{General index}

\newpage
\addcontentsline{toc}{section}{Bibliography}
\bibliographystyle{alpha}
\bibliography{bibl}
\end{document}